\documentclass[11pt,a4paper]{article}
\usepackage{amsfonts}    %
\usepackage{amsmath}     %
\usepackage{mathrsfs}    %
\usepackage{amssymb}     %
\usepackage{amsthm}      
\usepackage[T1]{fontenc} %
\usepackage{enumerate}   
\usepackage{stmaryrd}    
\numberwithin{equation}{section} %
\theoremstyle{plain}
\newtheorem{Thm}{Theorem}[section]
\newtheorem{Pro}[Thm]{Proposition}
\newtheorem{Lem}[Thm]{Lemma}
\newtheorem{Cor}[Thm]{Corollary}
\theoremstyle{definition}
\newtheorem{Exm}[Thm]{Example}
\newtheorem{Def}[Thm]{Definition}
\newtheorem{Rem}[Thm]{Remark}
\DeclareMathOperator{\dom}{dom}
\DeclareMathOperator{\ran}{ran}
\DeclareMathOperator{\mul}{mul}

\DeclareMathOperator{\re}{Re}
\DeclareMathOperator{\im}{Im}

\author{Jussi Behrndt and Till Micheler}
\title{Elliptic differential operators on Lipschitz domains and abstract boundary value problems}
\begin{document}
\maketitle
\begin{abstract}
  This paper consists of two parts. In the first part, which is of more abstract nature,
  the notion of quasi boundary triples and associated Weyl functions
  is developed further in such a way that it can be applied to elliptic boundary value
  problems on non-smooth domains. A key feature is the extension of the boundary
  maps by continuity to the duals of certain range spaces, which directly leads to a description
  of all self-adjoint extensions of the underlying symmetric operator with the help of
  abstract boundary values.
  In the second part of the paper a complete description is obtained
  of all self-adjoint realizations of the Laplacian
  on bounded Lipschitz domains, as well as Kre\u{\i}n type resolvent formulas
  and a spectral characterization in terms of energy dependent Dirichlet-to-Neumann
  maps. These results can be viewed as the natural generalization of recent
  results from Gesztesy and Mitrea for quasi-convex domains. In this connection we also
  characterize the maximal range spaces of the Dirichlet and Neumann trace operators on a bounded Lipschitz
  domain in terms of the Dirichlet-to-Neumann map. The general results from
  the first part of the paper are also applied to higher order elliptic operators on smooth domains,
  and particular attention is paid to the second order case which is illustrated with various examples.
\end{abstract} 
%
%
\section{Introduction}\label{sec:intro}
Spectral theory of elliptic partial differential operators has received a lot of attention in the recent past,
in particular, modern techniques from abstract operator theory were applied to extension and spectral problems
for symmetric and self-adjoint elliptic differential operators on bounded and unbounded domains. We refer the reader to the recent
contributions \cite{AmPe04,BeLa07,BeLa12,BeLaLo11,BrGrWo09,BMNW08,Gr11a,Gr11b,Gr12,Ma10} on smooth domains,  
\cite{AE11,AEKS13,GeMi08,GeMi09-2,GeMi09,Gr08,Gr11,PoRa09,P13} on non-smooth domains, and we point out
the paper \cite{GeMi11} by Gesztesy and Mitrea which has inspired parts of the present work.
Many of these contributions are based on the classical works Grubb~\cite{Gr68} and Vi\v{s}ik~\cite{Vi63}
on the parametrization of the closed realizations of a given elliptic differential expression on a smooth domain,
and other classical papers on realizations with local and non-local boundary conditions, see, e.g. \cite{AgDoNi59,BF62,Be65,Br60,F62,Sc59} and
the monograph \cite{LiMa72} by Lions and Magenes.

In \cite{GeMi11} Gesztesy and Mitrea obtain a complete description of the self-adjoint realizations of the Laplacian
on a class of bounded non-smooth, so-called \emph{quasi-convex} domains.
The key feature of quasi-convex domains is that the functions in the domains
of the self-adjoint Dirichlet realization $\Delta_D$
and the self-adjoint Neumann realization $\Delta_N$ possess $H^2$-regularity,
a very convenient property which is well-known to be false for the case of Lipschitz domains; cf.~\cite{JeKe95}.
Denote by $\tau_D$ and $\tau_N$ the Dirichlet and Neumann trace operator, respectively.
Building on earlier work of Maz'ya, Mitrea and Shaposhnikova \cite{MaMiSh10}, see also \cite{BuGe01, DuMu01, GeKr00},
the range spaces $\mathscr G_0 := \tau_D(\dom\Delta_N)$
and $\mathscr G_1 := \tau_N(\dom\Delta_D)$ were characterized for quasi-convex domains in \cite{GeMi11}, and
the self-adjoint realizations of the Laplacian were parametrized
via tuples $\{\mathscr X, L\}$, where $\mathscr X$ is a
closed subspace of the anti-dual $\mathscr G_0'$ or  $\mathscr G_1'$ and $L$ is a self-adjoint operator from $\mathscr X$ to 
$\mathscr X'$.
This parametrization technique has its roots in \cite{Bi56,Kr47} and was used in \cite{Gr68,Vi63}, see also \cite[Chapter~13]{Gr09}.
In \cite{BrGrWo09} the connection to the notion of (ordinary) boundary triples from extension theory of symmetric operators
was made explicit.

The theory of ordinary boundary triples and Weyl functions originates in the works of Ko\u{c}ube\u{\i} \cite{K75}, Bruk \cite{B76},
Gorbachuk and Gorbachuk \cite{GoGo91}, and Derkach and Malamud \cite{DeMa91,DeMa95}. A boundary triple $\{\mathcal G,\Gamma_0,\Gamma_1\}$
for a symmetric operator $A$ in a Hilbert space $\mathcal H$ consists of an auxiliary Hilbert space $\mathcal G$ and two boundary mappings
$\Gamma_0,\Gamma_1:\dom A^*\rightarrow\mathcal G$ which satisfy an abstract Green's identity and a maximality condition.
With the help of a boundary triple the closed extensions of the underlying symmetric operator $A$ can be parametrized in an efficient
way with closed operators and subspaces $\Theta$ in the boundary space $\mathcal G$.
The concept of
ordinary boundary triples was applied successfully to various problems in extension and spectral theory, in particular, in the
context of ordinary differential operators, see \cite{BGP08} for a review and further references. However, for the Laplacian (or more
general symmetric elliptic differential operators) on a domain $\Omega\subset\mathbb R^n$, $n\geq 2$, with boundary $\partial\Omega$
the natural choice $\Gamma_0=\tau_D$ and $\Gamma_1=-\tau_N$ does not lead to an ordinary boundary triple since Green's identity
does not extend to the domain of the maximal operator $A^*$. This simple observation led to a generalization of the concept
of ordinary triples, the so-called {\it quasi boundary triples}, which are designed for applications to PDE problems.
Here the boundary mappings $\Gamma_0=\tau_D$ and $\Gamma_1=-\tau_N$ are only defined on some suitable subset of $\dom A^*$, e.g.\ $H^2(\Omega)$,
and the realizations are labeled with  operators and subspaces $\Theta$ in the boundary space $L^2(\partial\Omega)$
via boundary conditions of the form $\Theta\tau_D f +\tau_N f = 0$, $f\in H^2(\Omega)$.
One of the advantages of this approach is that the Weyl function corresponding
to the quasi boundary triple $\{L^2(\partial\Omega),\tau_D,-\tau_N\}$ coincides (up to a minus sign)
with the usual family of Dirichlet-to-Neumann maps on the boundary $\partial\Omega$, and hence the spectral
properties of a fixed self-adjoint extension can be described
with the Dirichlet-to-Neumann map and the parameter $\Theta$ in the boundary condition.

The aim of the present paper is twofold. Our first objective is to further develop the abstract
notion of quasi boundary triples and their Weyl functions.
The main new feature is that we shall assume that the spaces
\[
\mathscr G_0=\ran\bigl(\Gamma_0 \upharpoonright \ker \Gamma_1\bigr)\quad\text{and}\quad
\mathscr G_1=\ran\bigl(\Gamma_1\upharpoonright\ker\Gamma_0\bigr)
\]
are reflexive Banach spaces densely embedded in the boundary space $\mathcal G$;
this assumption is natural in the context of PDE problems and related Sobolev spaces on the boundary of the domain, and is satisfied in applications
to the Laplacian on Lipschitz domains and other elliptic boundary value problems treated in the second part of the present paper.
In fact, this assumption is the abstract analog of the the properties
of the range spaces in \cite{GeMi11},
and it is also automatically satisfied in many abstract settings,
e.g. for ordinary and so-called generalized boundary triples; cf. \cite{DeMa95}
and Section~\ref{sec:counter} for a counterexample in the general case.
Under the density assumption it then follows that the boundary maps $\Gamma_0$ and $\Gamma_1$ can be extended by continuity to surjective mappings from $\dom A^*$ onto
the anti-duals $\mathscr G_1'$ and $\mathscr G_0'$, respectively.
Then also the $\gamma$-field and the Weyl function admit continuous extensions to operators mapping in between the appropriate spaces;
for the special case of generalized boundary triples and $\mathscr G_0$, $\mathscr G_1$ equipped with particular topologies this was noted in the
abstract setting earlier in
\cite[Proposition~6.3]{DeMa95} and \cite[Lemma 7.22]{DeHaMaSn12}.
Following the regularisation procedure in the PDE case we then show that a quasi boundary triple
with this additional density property can be transformed into a quasi boundary triple which is the restriction of an ordinary boundary triple, and hence
can be extended by continuity;
a similar argument can also be found in a different abstract form in \cite{DeHaMaSn12}.
As a consequence of these considerations we obtain a complete description of all
closed extensions of the underlying symmetric operator in Section~\ref{sec:extension}, as well as
abstract regularity results, Kre\u{\i}n type resolvent formulas and new sufficient criteria for the parameter $\Theta$ in the boundary condition
to imply self-adjointness of the corresponding extension.

The second objective of this paper is to apply the abstract quasi boundary triple technique to various non-local PDE problems.
In particular, in Section~\ref{sec:laplace} we
extend the characterization of the self-adjoint realizations $\Delta_\Theta$ of the Laplacian on quasi convex domains to the more natural case
of usual Lipschitz domains.
Here the Hilbert spaces $\mathscr G_0$ and $\mathscr G_1$ are topologized with the help of the Dirichlet-to-Neumann map
in a similar manner as in \cite{DeHaMaSn12, DeMa95} for abstract generalized boundary triples.
This also leads to a continuous extension of the Dirichlet and Neumann trace operators
on a Lipschitz domain to the maximal domain of the Laplacian, and hence to a description of the
Dirichlet boundary data for $L^2$-solutions of $-\Delta f=\lambda f$.
For the special case of quasi convex domains and $C^{1,r}$-domains with $r\in(\frac 12,1]$ we establish the link
to the approach in \cite{GeMi11}, and recover many of the results in \cite{GeMi11} as corollaries of the abstract methods
developed in Section~\ref{sec:qbt} and Section~\ref{sec:extension}. In
Section~\ref{sec:elliptic} we illustrate the abstract methods in the
classical case of $2m$-th order elliptic differential operators with smooth coefficients on smooth bounded domains, where
the spaces $\mathscr G_0$ and $\mathscr G_1$ coincide with the usual product Sobolev trace spaces on $\partial\Omega$.
Here e.g. some classical trace extension results follow from the abstract theory developed in the first part of the paper.
Finally, we pay particular attention to the second order case on bounded and unbounded domains with compact smooth boundary
in Section~\ref{sec:ellipticsecond}. Here we recover various recent results on the description and the spectral properties of the
self-adjoint extensions of a symmetric second order elliptic differential operator, and extend these by adding, e.g. regularity results.
This section contains also some simple examples, among them self-adjoint extensions with Robin boundary conditions.
One of the examples is also interesting from a more abstract point of view: It turns out that there exist self-adjoint parameters
in the range of the boundary maps of a quasi boundary triple such that the corresponding extension is
essentially self-adjoint, but not self-adjoint.

\vskip 0.2cm\noindent
{\bf Acknowledgement.} 
Till Micheler gratefully acknowledges financial support by the Studienstiftung des deutschen Volkes.
The authors are indebted to Fritz Gesztesy and Marius Mitrea for very valuable comments,
and to Seppo Hassi and Henk de Snoo for pointing out connections to recent abstract results. 
Moreover, the authors also wish to thank Vladimir Lotoreichik, Christian K\"{u}hn, and 
Jonathan Rohleder for many helpful discussions and remarks.

%
%
%
\section{Quasi boundary triples and their Weyl functions}\label{sec:qbt}
The concept of boundary triples and their Weyl functions is a useful and efficient tool in extension and spectral theory
of symmetric and self-adjoint operators, it originates in the works \cite{B76,K75} and was further developed in 
\cite{DeMa91,DeMa95,GoGo91}; cf. \cite{BGP08} for a review. 
In the recent past different generalizations
of the notion of boundary triples were introduced, among them  
boundary relations, boundary pairs and boundary triples associated with quadratic forms,
and other related concepts, see \cite{A00,DeHaMaSn06,DeHaMaSn09,DeHaMaSn12,Po04,Po08,P12,P13,R07,R09}.
The concept of quasi boundary triples and their Weyl functions introduced in \cite{BeLa07} is designed for 
the analysis of elliptic differential operators. It can be viewed as a slight generalization of the notions of boundary and generalized
boundary triples.  
In this section we first recall some definitions and basic properties which can be found in \cite{BeLa07,BeLa12}. 
Our main objective is to show that
under an additional density condition the corresponding boundary maps can be extended by continuity and that the 
corresponding quasi boundary triple can be transformed (or regularized) such that it turns 
into an ordinary boundary triple; cf.~\cite{DeHaMaSn12,W12,W13} for related investigations.
\subsection{Ordinary and quasi boundary triples}
Let throughout this section $A$ be a closed, densely defined, symmetric operator in a separable Hilbert space 
$\mathcal H$.
\begin{Def}\label{qbt}
Let $T\subset A^*$ be a linear operator in $\mathcal H$ such that $\overline T=A^*$.
A triple $\{\mathcal G,\Gamma_0,\Gamma_1\}$ is called \emph{quasi boundary triple} for $T$ 
  if $\mathcal G$ is a Hilbert space and $\Gamma_0,\Gamma_1:\dom T\rightarrow\mathcal G$ are linear mappings such that 
\begin{itemize}
 \item [{\rm (i)}] the \emph{abstract Green's identity}
  \begin{equation}\label{eqn:abstract_green}
    (T f,g)_{\mathcal H} - (f,Tg)_{\mathcal H} 
    = (\Gamma_1 f,\Gamma_0 g)_{\mathcal G}
    - (\Gamma_0 f,\Gamma_1 g)_{\mathcal G} 
  \end{equation}
  holds for all $f,g\in\dom T$,
 \item [{\rm (ii)}] the map $\Gamma := (\Gamma_0,\Gamma_1)^\top : \dom T \to 
  \mathcal G \times \mathcal G$ has dense range,
 \item [{\rm (iii)}] and $A_0:=T\upharpoonright \ker\Gamma_0$ is a self-adjoint operator in $\mathcal H$.
\end{itemize}
In the special case $T=A^*$ a quasi boundary triple $\{\mathcal G,\Gamma_0,\Gamma_1\}$ is called 
\emph{ordinary boundary triple}. 
\end{Def}
Let $\{\mathcal G,\Gamma_0,\Gamma_1\}$ be a quasi boundary triple for $T\subset A^*$. 
Then the mapping $\Gamma = (\Gamma_0,\Gamma_1)^\top: \dom T \to  \mathcal G \times \mathcal G$ is 
closable with respect to the graph norm of $A^*$ 
and $\ker \Gamma = \dom A$ holds; cf.~\cite[Proposition~2.2]{BeLa07}. Moreover, according to \cite[Theorem~2.3]{BeLa07} 
(see also Proposition~\ref{pro:rate} below) we have 
$T=A^*$ if and only if $\ran \Gamma = \mathcal G \times \mathcal G$, in this case 
$\Gamma= (\Gamma_0,\Gamma_1)^\top : \dom A^* \to \mathcal G \times \mathcal G$ 
is onto and continuous with respect to the graph norm of $A^*$, and the 
restriction $A_0=A^*\upharpoonright \ker\Gamma_0$ is automatically self-adjoint.
Thus, the above definition of an ordinary boundary triple coincides with the usual one, see, e.g.\ \cite{DeMa91}.
We also note that a quasi boundary triple is in general not a boundary relation in the sense of \cite{DeHaMaSn06,DeHaMaSn09}, but
it can be viewed as a certain transform of a boundary relation; cf. \cite[Proposition 5.1]{W13}.

For later purposes we recall a variant of \cite[Theorem~2.3]{BeLa07}.

\begin{Pro}\label{pro:rate}
  Let $\mathcal G$ be a Hilbert space and let $T$ be a linear operator in $\mathcal H$.
  Assume that $\Gamma_0,\Gamma_1 : \dom T \to \mathcal G$ are linear mappings such that the following conditions are satisfied:
  \begin{enumerate}[\rm(i)]
  \item $T\upharpoonright\ker \Gamma_0$ contains a self-adjoint linear operator $A$ in $\mathcal H$,
  \item The range and the kernel of $\Gamma := (\Gamma_0,\Gamma_1)^\top : \dom T \to \mathcal G \times \mathcal G$ are dense in
$\mathcal G \times \mathcal G$ and $\mathcal H$, respectively, 
  \item The abstract Green's identity \eqref{eqn:abstract_green} holds for all $f,g\in\dom T$.
  \end{enumerate}
  Then $ S := T \upharpoonright \ker \Gamma$ is a densely defined, closed symmetric operator in $\mathcal H$ and 
  $\{\mathcal G,\Gamma_0,\Gamma_1\}$ is a quasi boundary triple 
  for $S^*$ such that $A=T\upharpoonright\ker\Gamma_0=A_0$. Moreover, $T=S^*$ if and only if $\ran \Gamma = \mathcal G \times \mathcal G$.
\end{Pro}
Not surprisingly, suitable restrictions of ordinary boundary triples lead to quasi boundary triples.
\begin{Pro}\label{pro:qbt_restriction}
  Let $\{\mathcal G,\Gamma_0,\Gamma_1\}$ be an ordinary boundary triple for $A^*$ with $A_0=A^*\upharpoonright \ker\Gamma_0$. Let
  $T\subset A^*$ be such that $A_0\subset T$ and $\overline T = A^*$. 
  Then the restricted triple $\{\mathcal G,\Gamma_0^T,\Gamma_1^T\}$, where  
  $\Gamma_0^T := \Gamma_0\upharpoonright\dom T$ and $\Gamma_1^T := \Gamma_1\upharpoonright\dom T$ is a 
  quasi boundary triple for $T\subset A^*$.
\end{Pro}
\begin{proof}
Clearly, items (i) and (iii) in Definition~\ref{qbt} hold for the restricted triple $\{\mathcal G,\Gamma_0^T,\Gamma_1^T\}$.
Hence it remains to show that $\ran \Gamma^T=\ran(\Gamma_0^T,\Gamma_1^T)^\top$ is dense in $\mathcal G \times \mathcal G$. 
For this let $\hat x \in \mathcal G\times \mathcal G$. Then $\hat x \in \ran \Gamma$ and there exists an element 
  $f\in \dom A^*$ such that $\Gamma f = \hat x$. Since $\overline T=A^*$ there exists a sequence $(f_n)\subset\dom T$ which converges to $f$ in the graph norm of $A^*$. 
As $\Gamma$ is continuous with respect to the graph norm we obtain $\Gamma^T f_n=\Gamma f_n\rightarrow \hat x$ for $n\rightarrow\infty$,
that is, item (ii) in Definition~\ref{qbt} holds and $\{\mathcal G,\Gamma_0^T,\Gamma_1^T\}$ is a quasi boundary triple for $T\subset A^*$.
\end{proof}
The following proposition shows that the converse of Proposition~\ref{pro:qbt_restriction} holds under
an additional continuity assumption. 
In particular, it implies that 
if a quasi boundary triple can be extended to an ordinary boundary triple then this extension is unique. 
\begin{Pro}\label{pro:qbt_restriction2}
  Let $\{\mathcal G,\Gamma_0^T,\Gamma_1^T\}$ be a quasi boundary triple for $T\subset A^*$. Then 
  $\{\mathcal G,\Gamma_0^T,\Gamma_1^T\}$ is a restriction 
  of an ordinary boundary triple $\{\mathcal G,\Gamma_0,\Gamma_1\}$ for $A^*$ on $T$ if and only if the mapping 
  $\Gamma^T = (\Gamma_0^T,\Gamma_1^T)^\top : \dom T \to \mathcal G\times \mathcal G$ is continuous with respect to the graph norm of $A^*$.
\end{Pro}
\begin{proof}
  ($\Rightarrow$) Since 
  $\Gamma : \dom A^* \to \mathcal G\times \mathcal G$ is continuous with respect to the graph norm of $A^*$ the same holds for the
  restriction $\Gamma^T : \dom T \to \mathcal G\times \mathcal G$. 
  
  ($\Leftarrow$) Let $\Gamma = (\Gamma_0,\Gamma_1)^\top:\dom A^*\rightarrow\mathcal G\times\mathcal G$ be the 
  continuous extension of $\Gamma^T$ 
  with respect to the graph norm of $A^*$. Then also the abstract Green's identity extends by continuity from $\dom T$ to $\dom A^*$,
  \begin{equation}\label{absg}
    (A^* f,g)_{\mathcal H} - (f,A^*g)_{\mathcal H} 
    = (\Gamma_1 f,\Gamma_0 g)_{\mathcal G}
    - (\Gamma_0 f,\Gamma_1 g)_{\mathcal G},\quad f,g\in\dom A^*, 
  \end{equation}
  and the range of $\Gamma$ is dense in $\mathcal G\times\mathcal G$. Moreover, from \eqref{absg} it follows that 
  the operator $A^*\upharpoonright\ker\Gamma_0 $ is a symmetric extension of the self-adjoint operator $A_0=T\upharpoonright\ker\Gamma_0^T$
  and hence $A_0=A^*\upharpoonright\ker\Gamma_0 $. 
  We conclude that $\{\mathcal G,\Gamma_0,\Gamma_1\}$ is a quasi boundary
  triple for $\overline T = A^*$, that is, $\{\mathcal G,\Gamma_0,\Gamma_1\}$ is an ordinary boundary triple for $A^*$; 
  cf.~Definition~\ref{qbt}.
  Clearly, $\{\mathcal G,\Gamma_0^T,\Gamma_1^T\}$ 
  is the restriction of this ordinary boundary triple to $T$.  
\end{proof}
A simple and useful example of an ordinary and quasi boundary triple is provided in Lemma~\ref{lem:BTonNeta} below, it 
also implies the well-known fact that a boundary triple or quasi boundary triple exists if and only if $A$ has equal deficiency indices
$n_\pm(A) := \dim \ker(A^*\pm i)$, that is, if and only if $A$ admits self-adjoint extensions in $\mathcal H$. 
Recall first that for a self-adjoint extension $A_0\subset T$ of $A$ and $\eta\in\rho(A_0)$ 
the domains of $T$ and $A^*$ admit the direct sum decompositions
\begin{equation}\label{decot}
 \dom T=\dom A_0\,\dotplus\,\mathcal N_\eta(T)\quad\text{and}\quad
 \dom A^*=\dom A_0\,\dotplus\,\mathcal N_\eta(A^*),
\end{equation}
where $\mathcal N_\eta(T)=\ker(T-\eta)$ and $\mathcal N_\eta(A^*)=\ker(A^*-\eta)$. Note also that $\overline T=A^*$ implies
$\overline{\mathcal N_\eta(T)}=\mathcal N_\eta(A^*)$. 
Moreover we set 
\begin{equation*}
\widehat{\mathcal N}_\eta(T) := \bigl\{ (f_\eta, \eta  f_\eta)^\top:f_\eta\in \mathcal N_\eta(T) \bigr\},\quad
 \widehat{\mathcal N}_\eta(A^*) := \bigl\{ (f_\eta, \eta  f_\eta)^\top:f_\eta\in \mathcal N_\eta(A^*) \bigr\},
\end{equation*}
hence we may write 
$T = A_0\dotplus\widehat{\mathcal N}_\eta(T)$ and $A^* = A_0\dotplus\widehat{\mathcal N}_\eta(A^*)$.
The orthogonal projection in $\mathcal H$ onto the defect subspace $\mathcal N_\eta(A^*)$ will be denoted by $P_\eta$.

In the next lemma a special boundary triple and quasi boundary triple are constructed. The restriction $\eta\in\mathbb R$ below is
for convenience only, an example of a similar 
ordinary boundary triple with $\eta\in\mathbb C\setminus\mathbb R$
can be found in, e.g. \cite{DeMa91} or the monographs \cite{GoGo91,S12}. 
\begin{Lem}\label{lem:BTonNeta}
  Assume that the deficiency indices of $A$ are equal and
  let $\mathcal G$ be a Hilbert space with $\dim\mathcal G=n_\pm(A)$. Let
  $A_0$ be a self-adjoint extension of $A$ in $\mathcal H$, assume that there exists $\eta\in \rho(A_0)\cap\mathbb R$
  and fix a unitary operator $\varphi:\mathcal N_\eta(A^*)\rightarrow\mathcal G$. Then the following statements hold. 
  \begin{enumerate}[\rm (i)] 
  \item The triple $\{\mathcal G,\Gamma_0,\Gamma_1\}$, where 
    \begin{equation*}
    \Gamma_0 f := \varphi f_\eta \quad\text{and}\quad \Gamma_1 f := \varphi P_\eta (A_0-\eta) f_0, 
    \end{equation*}
    and $f\in\dom A^*$ is decomposed in $f=f_0+f_\eta\in\dom A_0 + \mathcal N_\eta(A^*)$, 
    is an ordinary boundary triple for $A^*$ with $A_0=A^*\upharpoonright\ker\Gamma_0$.
  \item If $T$ is an operator such that $A_0\subset T$ and $\overline T=A^*$, then the triple 
    $\{\mathcal G,\Gamma_0^T,\Gamma_1^T\}$, where 
    \begin{equation*}
    \Gamma_0^T f := \varphi f_\eta \quad\text{and}\quad \Gamma_1^T f := \varphi P_\eta (A_0-\eta) f_0,
    \end{equation*}
    and $f\in\dom T$ is decomposed in $f=f_0+f_\eta\in\dom A_0 + \mathcal N_\eta(T)$, 
    is a quasi boundary triple for $T$ with $A_0=T \upharpoonright\ker\Gamma_0^T$ and $\ran \Gamma_1^T=\ran \Gamma_1 =\mathcal G$.
  \end{enumerate}
\end{Lem}
\begin{proof} (i)
  Let $f,g\in \dom A^*$ be decomposed in the form $f=f_0 +
  f_\eta$ and $g=g_0+g_\eta$ with $f_0,g_0 \in \dom A_0$ and $f_\eta,g_\eta \in \mathcal
  N_\eta(A^*)$. Making use of $A_0=A_0^*$ and $\eta\in\mathbb R$ a straightforward computation yields
  \begin{eqnarray*}
    (A^* f,\, g)_{\mathcal H} - (f,\, A^* g)_{\mathcal H} 
    &=&  ((A_0-\eta) f_0,\, g_{\eta})_{\mathcal H} 
    - (f_\eta,\, (A_0-\eta) g_0)_{\mathcal H} \\
    &=& (\varphi P_\eta (A_0-\eta) f_0,\varphi g_{\eta})_{\mathcal G} 
    - (\varphi f_\eta, \varphi P_\eta (A_0-\eta) g_0)_{\mathcal G}\\
    &=& (\Gamma_1 f,\Gamma_0 g)_{\mathcal G} 
    - (\Gamma_0 f,\Gamma_1 g)_{\mathcal G},
  \end{eqnarray*}
  i.e., the abstract Green's identity holds. Moreover, $\Gamma_0:\dom A^*\rightarrow\mathcal G$ is surjective
  and since  $\ran(A_0-\eta)=\mathcal H$ it follows that also
  $\Gamma:\dom A^*\rightarrow\mathcal G\times\mathcal G$ is surjective. This implies that $\{\mathcal G,\Gamma_0,\Gamma_1\}$
  is an ordinary boundary triple for $A$. It is obvious that $A_0=A^*\upharpoonright\ker\Gamma_0$ holds.
  
  (ii) follows from (i) and Proposition~\ref{pro:qbt_restriction}.
\end{proof}
\subsection{Weyl functions and $\gamma$-fields of quasi boundary triples}\label{sec:weyl}
In this subsection the notion and some properties of $\gamma$-fields and Weyl functions associated to quasi boundary triples are 
briefly reviewed. Furthermore, a simple but useful description of the range of the boundary mappings is given in terms of the Weyl function
in Proposition~\ref{pro:weyl2}. 

Let $\{\mathcal G,\Gamma_0,\Gamma_1\}$ be a quasi boundary triple for $T\subset A^*$ and
let $A_0=T\upharpoonright\ker\Gamma_0$. Note that by \eqref{decot} the restriction $\Gamma_0\upharpoonright\mathcal N_\lambda(T)$
is invertible for every $\lambda\in\rho(A_0)$.

\begin{Def}\label{def:weyl} 
The \emph{$\gamma$-field} and 
  the \emph{Weyl function} corresponding to the quasi boundary triple
  $\{\mathcal G,\Gamma_0,\Gamma_1\}$ are defined by
  \begin{equation*} 
  \lambda \mapsto \gamma(\lambda):=(\Gamma_0\upharpoonright\mathcal N_{\lambda}(T))^{-1} 
  \quad\text{and}\quad \lambda \mapsto M(\lambda) := \Gamma_1 \gamma(\lambda), \quad \lambda \in \rho(A_{0}).
  \end{equation*}
\end{Def}

It follows that for $\lambda\in\rho(A_0)$ the operator $\gamma(\lambda)$ is continuous from $\mathcal G$ to $\mathcal H$ 
with dense domain $\dom \gamma(\lambda) = \ran \Gamma_0$  and range $\ran \gamma(\lambda) = \mathcal N_\lambda(T)$,
the function $\lambda \mapsto \gamma(\lambda) g$ is holomorphic on $\rho(A_0)$ for every $g\in \ran \Gamma_0$, and
the relations 
\begin{equation}\label{eqn:gamma_field} 
  \gamma(\lambda) = \left(I + (\lambda - \mu) (A_0 - \lambda)^{-1}\right ) \gamma(\mu)\quad\text{and}\quad
  \gamma (\lambda)^* = \Gamma_1 (A_{0}-\bar \lambda)^{-1} 
\end{equation}
hold for all $\lambda,\mu\in\rho(A_0)$; cf.~\cite[Proposition~2.6]{BeLa07}.
Note that $\gamma (\lambda)^*:\mathcal H\rightarrow\mathcal G$ is continuous and 
that $(\ker\gamma(\lambda)^*)^\bot=\overline{\ran\gamma(\lambda)}=\mathcal N_\lambda(A^*)$
yields the orthogonal space decomposition
\begin{equation}\label{eqn:Hdec}
  \mathcal H = \ker \gamma(\lambda)^* \oplus \mathcal N_\lambda(A^*).
\end{equation}
For $\lambda\in \rho(A_0)$ the values $M(\lambda)$ of the Weyl function are operators in $\mathcal G$ with dense domain $\ran\Gamma_0$
and range contained in $\ran\Gamma_1$. 
If, in addition, $A_1=T\upharpoonright\ker\Gamma_1$ is self-adjoint in $\mathcal H$ then $M(\lambda)$
maps $\ran\Gamma_0$ onto $\ran\Gamma_1$ for all $\lambda \in \rho(A_0) \cap \rho(A_1)$. 
Furthermore, 
$M(\lambda) \Gamma_0 f_\lambda = \Gamma_1 f_\lambda$ holds for all $f_\lambda \in \mathcal N_\lambda (T)$ and this
implies the identity
\begin{equation}\label{gammaid}
  \Gamma_1 f = M(\lambda) \Gamma_0 f+ \Gamma_1 f_0, \qquad f=f_0+f_\lambda\in\dom A_0 \dotplus\mathcal N_\lambda(T).  
\end{equation}
We also mention that for $\lambda,\mu\in\rho(A_0)$ the Weyl function is connected with the $\gamma$-field via  
\begin{equation}\label{mg}
  M(\lambda)x - M(\mu)^*x = (\lambda -\bar \mu) \gamma (\mu)^* \gamma(\lambda)x,\qquad x\in\ran\Gamma_0,
\end{equation}
and, in particular, $M(\lambda)$ is a symmetric operator in $\mathcal G$ for $\lambda\in\mathbb R\cap\rho(A_0)$. It is important to note
that
\begin{equation}\label{dominv}
 \ran\Gamma_0=\dom M(\lambda)\subset \dom M(\mu)^*,\qquad\lambda,\mu\in\rho(A_0).
\end{equation}

The subspaces $\mathscr G_0$ and $\mathscr G_1$ of $\mathcal G$ in the next definition will play a fundamental role throughout
this paper.\footnote{We emphasize that $\mathscr G_0$ and $\mathscr G_1$ in Definition~\ref{def:g0g1} do, in general, not coincide with the
spaces $\mathcal G_0=\ran\Gamma_0$ and $\mathcal G_1=\ran\Gamma_1$; this notation was used in \cite{BeLa07,BeLa12}. The symbols $\mathcal G_0$ and 
$\mathcal G_1$ will not be used in the present paper.}
\begin{Def}\label{def:g0g1}  
  Let $\{\mathcal G,\Gamma_0,\Gamma_1\}$ be a quasi boundary triple for $T\subset A^*$. 
  Then we define the spaces 
  \begin{equation*}
    \mathscr G_0 := \ran \bigl(\Gamma_0 \upharpoonright \ker\Gamma_1\bigr)\qquad\text{and}\qquad 
    \mathscr G_1 := \ran \bigl(\Gamma_1 \upharpoonright \ker\Gamma_0\bigr).
  \end{equation*}
\end{Def}
Observe that for the spaces $\mathscr G_0$ and $\mathscr G_1$ in Definition~\ref{def:g0g1} we have $\mathscr G_0 \times \mathscr G_1 \subset \ran \Gamma$. 
Note also that the second identity in \eqref{eqn:gamma_field} implies 
\begin{equation}\label{xxx}
\ran\gamma(\lambda)^*=\mathscr G_1,\qquad  \lambda\in\rho(A_0).
\end{equation}
\begin{Pro}\label{pro:weyl2}
Let $\{\mathcal G,\Gamma_0,\Gamma_1\}$ be a quasi boundary triple for $T\subset A^*$ with $A_0=T\upharpoonright\ker\Gamma_0$ and 
Weyl function $M$, and let $\mathscr G_0$ and $\mathscr G_1$ be as in Definition~{\rm\ref{def:g0g1}}.
Then the following assertions hold for all $\lambda \in \rho(A_0)$.
\begin{enumerate}[\rm(i)]
\item $M(\lambda)$ maps $\mathscr G_0$ into $\mathscr G_1$ and if, in addition, $A_1=T\upharpoonright\ker\Gamma_1$ is
  self-adjoint, then $M(\lambda)\upharpoonright\mathscr G_0$ is a bijection onto $\mathscr G_1$ for 
  $\lambda \in \rho(A_0) \cap \rho(A_1)$,
\item the range of the boundary mapping $\Gamma=(\Gamma_0,\Gamma_1)^\top$ is
  \begin{equation}
    \ran \Gamma = \left\{ \begin{pmatrix} x \\ x' \end{pmatrix} \in \ran\Gamma_0 \times \ran \Gamma_1 
    : x' = M(\lambda) x + y,\, y \in \mathscr G_1 \right\} \label{eqn:weyl2_1}
  \end{equation}
  and, in particular, $\dom M(\lambda)^* \cap \mathscr G_1^\bot = \{0\}$. 
\end{enumerate}
\end{Pro}

\begin{proof}
  (i) We verify $M(\lambda)x\in  \mathscr G_1$ for $x\in \mathscr G_0$. By definition of $\mathscr G_0$ there exists 
  $f_1\in \ker \Gamma_1$ such that $\Gamma_0 f_1 = x$. 
  Together with $\Gamma_0 \gamma(\lambda) x = x$ we conclude $\gamma(\lambda) x -f_1 \in \ker \Gamma_0$ 
  and 
  \begin{equation*}
  M(\lambda) x= \Gamma_1\gamma(\lambda) x = \Gamma_1 (\gamma(\lambda) x -f_1) \in \mathscr G_1.
  \end{equation*}
  Assume now that $A_1$ is self-adjoint and let $\lambda \in \rho(A_0) \cap \rho(A_1)$. Since 
  $M(\lambda) : \ran \Gamma_0 \to \ran \Gamma_1$ is a bijection it suffices to check
  that $M(\lambda)\upharpoonright\mathscr G_0$ maps onto $\mathscr G_1$.
  For $y\in \mathscr G_1$ there exists $f_0\in \ker\Gamma_0$ with $\Gamma_1 f_0=y$ and $x\in\ran\Gamma_0$ with
  $M(\lambda)x=y$. Hence we obtain
  \begin{equation*}
  \Gamma_1f_0= y=M(\lambda) x = \Gamma_1 \gamma(\lambda) x 
  \end{equation*} 
  and therefore $ \gamma(\lambda) x - f_0 \in \ker \Gamma_1$ and $\Gamma_0 ( \gamma(\lambda) x - f_0) = x \in \mathscr G_0$.
  This completes the proof of item (i).

  \noindent (ii) We show first that $\ran\Gamma$ is contained in the right hand side of \eqref{eqn:weyl2_1}.
  Let $\hat x=( x,\,x')^\top \in \ran \Gamma$ and choose $f=f_0+f_\lambda \in \dom T=\dom A_0\dotplus\mathcal N_\lambda(T)$ 
  such that $\Gamma f= \hat x$. 
  From \eqref{gammaid} and $\Gamma_0 f=x$ we conclude  
  \begin{equation*}
  x' = \Gamma_1 f = M(\lambda) \Gamma_0 f + \Gamma_1 f_0 = M(\lambda) x + y,\quad\text{where}\quad 
  y :=\Gamma_1 f_0 \in \mathscr G_1,
  \end{equation*}
  and hence $\hat x$ belongs to the right hand side of \eqref{eqn:weyl2_1}.
  
  Conversely, let $x\in\ran\Gamma_0$ and $x' = M(\lambda) x + y$ with some $y\in \mathscr G_1$. 
  Then there exist $f_0 \in \ker \Gamma_0 $ with $\Gamma_1 f_0 = y$ 
  and $f_\lambda\in \mathcal N_\lambda(T)$ with $\Gamma_0 f_\lambda = x$.
  Setting $f:=f_0+f_\lambda \in \dom T$ we find $\Gamma_0 f = x$
  and from \eqref{gammaid} we obtain 
  \[
  x' =  M(\lambda) x + y = M(\lambda) \Gamma_0 f+ \Gamma_1 f_0 = \Gamma_1 f,
  \]
  that is,  $(x,x')^\top \in \ran\Gamma$ and the identity \eqref{eqn:weyl2_1} is proved. 
  
  The remaining assertion in (ii) follows from the representation \eqref{eqn:weyl2_1} and the 
  fact that $\ran\Gamma$ is dense in $\mathcal G\times\mathcal G$.
\end{proof}
Let again $\{\mathcal G,\Gamma_0,\Gamma_1\}$ be a quasi boundary triple for $T\subset A^*$ with $A_0=T\upharpoonright\ker\Gamma_0$ and 
Weyl function $M$. 
For $\lambda\in\rho(A_0)$ define the operators
\begin{equation}\label{def:ReImM}
\begin{split}
  \re M(\lambda) &:= \frac 12     \left(M(\lambda)+M(\lambda)^*\right),\quad \dom \bigl(\re M(\lambda)\bigr) =\ran\Gamma_0,\\
  \im M(\lambda) &:= \frac{1}{2i} \left(M(\lambda)-M(\lambda)^*\right),\quad \dom \bigl(\im M(\lambda)\bigr) =\ran\Gamma_0.
\end{split}
\end{equation}
Then $M(\lambda) = \re M(\lambda) + i \im M(\lambda)$ and it follows
from \eqref{mg} that
\begin{equation*}
\im M(\lambda)=\im \lambda\,\gamma(\lambda)^*\gamma(\lambda),\qquad\lambda\in\rho(A_0),
\end{equation*}
holds. 
Hence $\im M(\lambda)$ is a densely defined, invertible bounded operator in $\mathcal G$ 
with $\ran (\im M(\lambda)) \subset \mathscr G_1$; cf.~\eqref{eqn:gamma_field}.
Therefore we may rewrite Proposition~\ref{pro:weyl2}~(ii) in the form 
\begin{equation*}
  \ran \Gamma = \left\{ \begin{pmatrix} x \\ x' \end{pmatrix} \in \ran\Gamma_0 \times \ran \Gamma_1 
  : x' = \re M(\lambda) x + y,\, y \in \mathscr G_1 \right\}.
\end{equation*}
The continuous extension of $\im M(\lambda)$ onto $\mathcal G$ is given by the closure
\begin{equation}\label{eqn:closeImWeyl}
  \overline{\im M(\lambda)}=\im \lambda\,\gamma(\lambda)^*\overline{\gamma(\lambda)},\qquad\lambda\in\rho(A_0).
\end{equation}
It is important to note that for $\lambda\in\mathbb C\setminus\mathbb R$ we have
\begin{equation}\label{eqn:kerImWeyl}
  \ker\bigl(\overline{\im M(\lambda)}\bigr)= \ker \overline{\gamma(\lambda)}=\bigl(\ran\gamma(\lambda)^*\bigr)^\bot=\mathscr G_1^\bot,
\end{equation}
which may be nontrivial; cf.~Proposition~\ref{gehtschief}.
\subsection{Extensions of boundary mappings, $\gamma$-fields and Weyl functions}\label{2.3}
Let $\{\mathcal G,\Gamma_0,\Gamma_1\}$ be a quasi boundary triple for $T\subset A^*$. In this section we investigate the 
case where the space $\mathscr G_1=\ran(\Gamma_1\upharpoonright\ker\Gamma_0)$ 
in Definition~\ref{def:g0g1} is dense in $\mathcal G$. Under this assumption we show that 
the boundary map $\Gamma_0$ and the $\gamma$-field admit continuous extensions. If, in addition,   
$\mathscr G_0=\ran(\Gamma_0\upharpoonright\ker\Gamma_1)$ is dense in $\mathcal G$ and $A_1 = T\upharpoonright \ker \Gamma_1$ is self-adjoint in $\mathcal H$ 
then also $\Gamma_1$ and the Weyl function $M$ admit continuous extensions. We point out that in general
$\mathscr G_1$ (or $\mathscr G_0$) is not dense in $\mathcal G$, see Proposition~\ref{gehtschief} for a counterexample.

The next proposition is a variant of \cite[Proposition~6.3]{DeMa95} (see also \cite[Lemma 7.22]{DeHaMaSn12}) for quasi boundary triples and 
their Weyl functions. It was proved for generalized boundary triples in \cite{DeMa95}, where the additional assumption that 
$\mathscr G_1$ is dense in $\mathcal G$ is automatically satisfied; cf.~\eqref{eqn:kerImWeyl} and 
\cite[Lemma~6.1]{DeMa95}. In the following $\mathscr G_1'$ stands for the anti-dual space of $\mathscr G_1$. 
\begin{Pro}\label{pro:DeMa95}
  Let $\{\mathcal G, \Gamma_0, \Gamma_1\}$ be a quasi boundary triple for $T\subset A^*$ with 
  Weyl function $M$, set $\Lambda := \overline{\im M(i)}$ and assume, in addition, that  
  $\mathscr G_1$ is dense in $\mathcal G$.
  Then 
  \begin{equation*}
   \mathscr G_1=\ran \Lambda^{1/2}
  \end{equation*}
  and if $\mathscr G_1$ is equipped with the norm induced by the inner product
  \begin{equation}\label{eqn:G_1_innerproduct}
    (\Lambda^{-1/2} x , \Lambda^{-1/2} y)_{\mathcal G},\qquad x, y \in \mathscr G_1,
  \end{equation}
  then the following assertions hold.
  \begin{itemize}
  \item [{\rm (i)}] $\gamma(i)$ extends to an isometry $\widetilde\gamma(i)$ from $\mathscr G_1'$ onto $\mathcal N_i(A^*)$,
  \item [{\rm (ii)}] $\im M(i)$ extends to an isometry from $\mathscr G_1'$ onto $\mathscr G_1$. 
  \end{itemize}
\end{Pro}
\begin{proof}
Since the space $\mathscr G_1$ is dense in $\mathcal G$ the bounded self-adjoint operator 
$\Lambda = \overline{\im M(i)} = \gamma(i)^*\overline{\gamma(i)}$ 
is injective and non-negative; cf.~\eqref{eqn:closeImWeyl} and
\eqref{eqn:kerImWeyl}. Hence $\ran \Lambda$ and $\ran \Lambda^{1/2}$ are dense in $\mathcal G$. As in the 
proof of \cite[Proposition~6.3]{DeMa95} we equip $\mathscr G := \ran \Lambda^{1/2}$ with the inner product 
\begin{equation*}
(\Lambda^{-1/2} x , \Lambda^{-1/2} y )_{\mathcal G}, \qquad x, y\in \mathscr G.
\end{equation*}
Then $\mathscr G$ is a Hilbert space which is densely embedded in $\mathcal G$ and hence gives rise to
a Gelfand triple $\mathscr G \hookrightarrow\mathcal G\hookrightarrow\mathscr G'$,
where $\mathscr G'$ is the completion of $\mathcal G$ equipped with the inner product 
$
(\Lambda^{1/2} x , \Lambda^{1/2} y )_{\mathcal G}$, $x, y\in \mathcal G$.
As in \cite[Proposition~6.3]{DeMa95} one verifies that the mapping $\gamma(i)$ admits a continuation to an isometry  
$\widetilde\gamma(i)$ from $\mathscr G'$ onto $\mathcal N_i(A^*)$ and the mapping
$\im M(i)$ admits a continuation to an isometry $\widetilde \Lambda$ from $\mathscr G'$ onto $\mathscr G$ 
with $\Lambda \subset \widetilde \Lambda = \gamma(i)^* \widetilde \gamma(i)$. This implies
$\mathscr G = \ran \gamma(i)^* = \mathscr G_1$ by \eqref{xxx} and assertions (i) and (ii) follow.
\end{proof}
The next proposition contains a simple but far-reaching observation: If $\mathscr G_1$ is dense in $\mathcal G$ and $\mathscr G_1$ is equipped
with a Hilbert or Banach space norm such that $\Gamma_1 (A_0-\bar \lambda)^{-1}:\mathcal H\to \mathscr G_1$ is continuous
then the boundary map $\Gamma_0$ can be extended by continuity onto $\dom A^*$. Although Proposition~\ref{pro:DeMa95} provides
a possible norm on $\mathscr G_1$  it is essential for later applications to allow other norms
which are a priori not connected with the Weyl function.
\begin{Pro}\label{pro:Gamma_0_extension}
Let $\{\mathcal G,\Gamma_0,\Gamma_1\}$ be a quasi boundary triple for $T\subset A^*$ with 
$A_0 = T\upharpoonright\ker\Gamma_0$ and assume, in addition, that $\mathscr G_1$ is dense in $\mathcal G$.
Then for any norm $\|\cdot\|_{\mathscr G_1}$ such that 
\begin{enumerate}[\rm (i)]
\item $(\mathscr G_1,\|\cdot \|_{\mathscr G_1})$ is a reflexive Banach space continuously embedded in $\mathcal G$ and
\item the operator 
  \begin{equation*}
  \Gamma_1 (A_0-\bar \lambda)^{-1}:\mathcal H\to \mathscr G_1
  \end{equation*}
  is continuous for some, and hence for all, $\lambda \in \rho(A_0)$,
\end{enumerate}
hold
the boundary mapping $\Gamma_0$ admits a unique surjective, continuous extension
\begin{equation*}\label{wtg0}
  \widetilde \Gamma_0 :(\dom A^*,\Vert\cdot\Vert_{A^*})\to \mathscr G_1',
\end{equation*}
where $\mathscr G_1'$ is the anti-dual space of $\mathscr G_1$.
Moreover, the norm $|\!|\!| \cdot |\!|\!|_{\mathscr G_1}$ induced by the inner product 
\eqref{eqn:G_1_innerproduct} 
is equivalent to any norm $\|\cdot \|_{\mathscr G_1}$ on $\mathscr G_1$ with the properties {\rm (i)-(ii)}.
\end{Pro}

\begin{proof}
  Fix some $\lambda\in\rho(A_0)$ and define $S :=\Gamma_1 (A_{0}-\bar \lambda)^{-1}=\gamma(\lambda)^*$.
  As $\ker S=(\ran \gamma(\lambda))^\bot=\mathcal N_\lambda(A^*)^\bot$ the restriction of $S$ onto $\mathcal N_\lambda(A^*)$ is 
  an isomorphism from $\mathcal N_\lambda(A^*)$ onto $\mathscr G_1$ by (ii). Hence 
  the adjoint operator $S':\mathscr G_1'\rightarrow \mathcal H$ is bounded, invertible and by the closed range theorem 
  $\ran S'=\mathcal N_\lambda(A^*)$.
  The inverse $(S')^{-1}$ is regarded as an isomorphism from 
  $\mathcal N_\lambda(A^*)$ onto $\mathscr G_1'$ in the sequel.
  For $x \in \ran \Gamma_0\subset\mathscr G_1'$ and $h\in \mathcal H$ it follows from 
  \begin{equation*}
    ( S' x, h)_{\mathcal H} 
    = \langle x , S h \rangle_{\mathscr G_1'\times \mathscr G_1} 
  =(x, S h)_{\mathcal G} = (x, \Gamma_1 (A_{0}-\bar \lambda)^{-1} h)_{\mathcal G}= (\gamma(\lambda) x, h)_{\mathcal H},
  \end{equation*} 
  that $S'\upharpoonright\ran\Gamma_0=\gamma(\lambda)$. We define the mapping 
  \[
  \widetilde \Gamma_0 : \dom A^* \to \mathscr G_1',\qquad f \mapsto\widetilde \Gamma_0 f = (S')^{-1} f_\lambda,
  \] where 
  $f=f_0+f_\lambda\in\dom A_0+\mathcal N_\lambda(A^*)=\dom A^*$. 
  For $f \in \dom T$ decomposed in the form $f=f_0+f_\lambda$ with $f_0\in\dom A_{0}$ 
  and $f_\lambda \in \mathcal N_\lambda(T)$ we have
  \begin{equation*}
  \widetilde \Gamma_0 f 
  = (S')^{-1} f_\lambda 
  = (S')^{-1} \gamma(\lambda) \Gamma_0 f_\lambda 
  = (S')^{-1} S' \Gamma_0 f_\lambda 
  = \Gamma_0 f_\lambda 
  = \Gamma_0 f,
  \end{equation*}
  and hence $\widetilde \Gamma_0$ is an extension of $\Gamma_0$. 
  It remains to check that $\widetilde \Gamma_0$ is continuous. For this let $f=f_0+f_\lambda\in\dom A^*$ and note that
  $f_\lambda=f- (A_{0}-\lambda)^{-1} (A^*-\lambda) f$ holds. Since $(S')^{-1}:\mathcal N_\lambda(A^*)\rightarrow\mathscr G_1'$ is bounded we find
  \begin{equation*}
    \begin{split}
      \| \widetilde \Gamma_0 f \|_{\mathscr G_1'} 
      &= \|  (S')^{-1} f_\lambda \|_{\mathscr G_1'} 
      \leq \|  (S')^{-1}\|  \bigl(\| f \|_{ \mathcal H} + \| (A_{0}-\lambda)^{-1} (A^*-\lambda) f  \|_{\mathcal H}\bigr) \\
      &\leq c  \| f \|_{A^*}
    \end{split}
  \end{equation*}
  with some constant $c>0$.
 
  Let $|\!|\!| \cdot |\!|\!|_{\mathscr G_1}$ be the norm induced by the inner product \eqref{eqn:G_1_innerproduct}
  and let $\|\cdot \|_{\mathscr G_1}$ be an arbitrary norm on $\mathscr G_1$  such that 
  $(\mathscr G_1,\, \|\cdot \|_{\mathscr G_1})$ is a reflexive Banach space densely embedded in $\mathcal G$
  and $\gamma(i)^* =  \Gamma_1 (A_0 + i)^{-1}$ is continuous from $\mathcal H$ to ($\mathscr G_1, \| \cdot \|_{\mathscr G_1})$. 
  Recall that $\ker\gamma(i)^*=\mathcal N_i(A^*)^\bot$; cf.~\eqref{eqn:Hdec}. 
  It follows from Proposition~\ref{pro:DeMa95} that $\gamma(i)^*$ is an isometry from $\mathcal N_i(A^*)$ onto 
  $(\mathscr G_1,|\!|\!| \cdot |\!|\!|_{\mathscr G_1})$ and hence $(\gamma(i)^*\upharpoonright \mathcal N_i(A^*))^{-1}$ is an isometry from 
  $(\mathscr G_1,|\!|\!| \cdot |\!|\!|_{\mathscr G_1})$ onto $\mathcal N_i(A^*)$. Therefore we obtain 
  \[
  |\!|\!| x |\!|\!|_{\mathscr G_1} = \| (\gamma(i)^*\upharpoonright \mathcal N_i(A^*))^{-1} x \|_{\mathcal H} \leq c' \| x \|_{\mathscr G_1}
  \]
  with $c'>0$ for all $x\in \mathscr G_1$. 
  Hence $I : (\mathscr G_1, \| \cdot \|_{\mathscr G_1}) \to (\mathscr G_1, |\!|\!| \cdot |\!|\!|_{\mathscr G_1})$ 
  is continuous and this implies the norm equivalence $|\!|\!|\cdot|\!|\!|_{\mathscr G_1} \sim \|\cdot\|_{\mathscr G_1}$.
\end{proof}
If $\{\mathcal G,\Gamma_0,\Gamma_1\}$ is a quasi boundary triple for $T\subset A^*$ with Weyl function $M$ and the additional property that 
$A_1 = T\upharpoonright\ker\Gamma_1$
is self-adjoint, then the triple $\{\mathcal G, -\Gamma_1, \Gamma_0\}$ is also a quasi boundary triple for $T\subset A^*$ with Weyl 
function $\lambda \mapsto -M(\lambda)^{-1},\, \lambda\in\rho(A_1)$. This fact together with Proposition~\ref{pro:Gamma_0_extension} 
implies the following statement.
\begin{Cor}\label{cor:Gamma_1_extension} 
Let $\{\mathcal G, \Gamma_0, \Gamma_1\}$ be a quasi boundary triple for $T\subset A^*$ and assume, 
in addition, that $A_1 = T\upharpoonright\ker\Gamma_1$ is self-adjoint
  in $\mathcal H$ and $\mathscr G_0$ is dense in $\mathcal G$.
Then for any norm $\|\cdot\|_{\mathscr G_0}$ such that 
\begin{itemize}
\item [{\rm (i)}] $(\mathscr G_0,\|\cdot \|_{\mathscr G_0})$ is a reflexive Banach space continuously embedded in $\mathcal G$ and
\item [{\rm (ii)}] the operator 
  \begin{equation*}
  \Gamma_0 (A_1-\bar \lambda)^{-1}:\mathcal H\to \mathscr G_0
  \end{equation*}
  is continuous for some, and hence for all, $\lambda \in \rho(A_1)$,
\end{itemize}
hold the boundary mapping $\Gamma_1$ admits a unique surjective, continuous extension
  \begin{equation*}
    \widetilde \Gamma_1 :(\dom A^*,\Vert\cdot\Vert_{A^*})\to \mathscr G_0',
  \end{equation*}
where $\mathscr G_0'$ is the anti-dual space of $\mathscr G_0$.
\end{Cor}
We note that in the situation of the above corollary it follows that the closure of $\im(-M(i)^{-1})$ is an invertible bounded 
operator defined on $\mathcal G$. Making use of Proposition~\ref{pro:DeMa95} for the quasi boundary triple 
$\{\mathcal G, -\Gamma_1, \Gamma_0\}$
and setting $\Sigma:=\overline{\im(-M(i)^{-1})}$ we then conclude that the norm $|\!|\!| \cdot |\!|\!|_{\mathscr G_0}$ 
induced by the inner product 
\[
(\Sigma^{-1/2} x,\Sigma^{-1/2} y)_{\mathcal G},\qquad x,y\in\mathscr G_0,
\]
is equivalent to any norm $\|\cdot \|_{\mathscr G_0}$ on $\mathscr G_0$ which satisfies {\rm (i)-(ii)} in Corollary~\ref{cor:Gamma_1_extension}.

The next theorem is strongly inspired by regularisation techniques used in extension theory of symmetric partial differential
operators; cf.~\cite{Gr68,Vi63}. It will be shown that a quasi boundary triple $\{\mathcal G, \Gamma_0, \Gamma_1\}$ with 
the additional property that $\mathscr G_1$ is dense in $\mathcal G$ can be transformed and extended to an ordinary boundary triple. 
Such a type of transform appears also in \cite{BeLa12,BrGrWo09} and in a more abstract form in \cite{DeHaMaSn12}, see also \cite{W12,W13}. 
Here we discuss only a situation which is relevant in applications, namely we assume that the spectrum of 
the self-adjoint operator $A_0=T\upharpoonright\ker\Gamma_0$ does not cover the whole real line.
The more general case is left to the reader; cf. Remark~\ref{rem:techn}.
Recall that for the Gelfand triple 
$\mathscr G_1\hookrightarrow\mathcal G\hookrightarrow\mathscr G_1'$ there exist isometric isomorphisms 
$\iota_+ : \mathscr G_1 \to \mathcal G$ and $\iota_- : \mathscr G_1' \to \mathcal G$ such that
\begin{equation}\label{iotas}
  ( \iota_- x',\, \iota_+ x)_{\mathcal G} =  \langle x',\, x  \rangle_{\mathscr G_1' \times \mathscr G_1} 
  \quad\text{for all}\quad x\in \mathscr G_1,\, x'\in \mathscr G_1'.
\end{equation}
Here and in the following $\mathscr G_1$ is equipped with some
norm $\| \cdot\|_{\mathscr G_1}$ such that {\rm(i)} and {\rm(ii)} in Proposition~\ref{pro:Gamma_0_extension} hold. 
Recall that according to Proposition~\ref{pro:DeMa95} such a norm always exists (if $\mathscr G_1$ is dense in $\mathcal G$) and that all such norms are equivalent by 
Proposition~\ref{pro:Gamma_0_extension}.

\begin{Thm}\label{thm:regularisation}
  Let $\{\mathcal G, \Gamma_0, \Gamma_1\}$ be a quasi boundary triple for $T\subset A^*$ with $A_0 = A^*\upharpoonright \ker\Gamma_0$, 
  assume that there exists $\eta \in \rho(A_0)\cap\mathbb R$ and that  $\mathscr G_1$ is dense in $\mathcal G$.
  Then the triple $\{\mathcal G, \Upsilon_0, \Upsilon_1\}$ with boundary mappings 
  $\Upsilon_0, \Upsilon_1 : \dom A^* \to \mathcal G$ given by
  \begin{equation*}
  \Upsilon_0 f := \iota_- \widetilde \Gamma_0 f, \quad \Upsilon_1 f := \iota_+ \Gamma_1 f_0, 
\quad f=f_0+f_\eta \in \dom A_{0} \dotplus \mathcal N_\eta(A^*),
  \end{equation*}
   is an ordinary boundary triple for $A^*$ with 
   \begin{equation*} 
   A^* \upharpoonright \ker \Upsilon_0 = A_0 \quad\text{and}\quad 
   A^* \upharpoonright \ker \Upsilon_1 = A \dotplus \widehat{\mathcal N}_\eta(A^*).
   \end{equation*}
\end{Thm}
\begin{proof}
  We verify that the restriction $\{\mathcal G,\Upsilon_0^T,\Upsilon_1^T\}$, 
  \begin{equation*}
  \Upsilon_0^T f = \iota_- \Gamma_0 f, \quad \Upsilon_1^T f = \iota_+ \Gamma_1 f_0, \quad f=f_0+f_\eta \in \dom A_{0} \dotplus \mathcal N_\eta(T),
  \end{equation*}
 of the triple $\{\mathcal G,\Upsilon_0,\Upsilon_1\}$ on $T$ 
  is a quasi boundary triple for $T\subset A^*$, 
  such that the boundary mapping $\Upsilon^T = (\Upsilon_0^T,\Upsilon_1^T)^\top : \dom T \to \mathcal G \times \mathcal G$ is continuous with
  respect to the graph norm of $A^*$. Then Proposition~\ref{pro:qbt_restriction2} implies that $\{\mathcal G, \Upsilon_0, \Upsilon_1\}$ 
  is an ordinary boundary triple for $A^*$.
  
  Note first that $\ker \Upsilon_0^T=\ker \Gamma_0$ holds. Thus $T\upharpoonright \ker\Upsilon_0^T$ coincides with the self-adjoint 
  linear operator $A_0$ in $\mathcal H$ and (iii) in Definition~\ref{qbt} holds.
  In order to check Green's identity observe that for all $f\in\dom T$ the identity 
  $\Upsilon^T_1 f = \iota_+ (\Gamma_1 f - M(\eta) \Gamma_0 f)$ holds by \eqref{gammaid}. 
  Here $M$ is the Weyl function of the quasi boundary triple
  $\{\mathcal G, \Gamma_0, \Gamma_1\}$ and since by assumption $\eta\in\mathbb R\cap\rho(A_0)$ 
  the operator $M(\eta)$ is symmetric in $\mathcal G$; cf.~\eqref{mg}.
  Making use of \eqref{iotas} and the fact that $\langle\cdot,\cdot\rangle_{\mathscr G_1 \times \mathscr G_1'}$ is the continuous extension of the 
  scalar product in $\mathcal G$ we compute for all $f,\, g\in \dom T$
  \begin{eqnarray*}
     && (\Upsilon^T_1 f ,\, \Upsilon^T_0 g )_{\mathcal G} - ( \Upsilon^T_0 f,\, \Upsilon^T_1 g )_{\mathcal G} \\
    && \quad\qquad = \bigl\langle \Gamma_1 f - M(\eta) \Gamma_0 f , \Gamma_0 g\bigr\rangle_{\mathscr G_1 \times \mathscr G_1'} - 
               \bigl\langle  \Gamma_0 f, \Gamma_1 g - M(\eta) \Gamma_0 g\bigr\rangle_{\mathscr G_1' \times \mathscr G_1} \\ 
   && \quad\qquad = \bigl( \Gamma_1 f - M(\eta) \Gamma_0 f , \Gamma_0 g\bigr)_{\mathcal G}  - 
               \bigl(  \Gamma_0 f, \Gamma_1 g - M(\eta) \Gamma_0 g\bigr)_{\mathcal G} \\
    && \quad\qquad = ( \Gamma_1 f,\Gamma_0 g )_{\mathcal G} 
    -   ( \Gamma_0 f,\Gamma_1 g )_{\mathcal G} \\
    && \quad \qquad = ( T f ,\, g )_{\mathcal H} - ( f ,\, T g )_{\mathcal H}.
  \end{eqnarray*}
  Now we verify that $\ran \Upsilon^T$ is dense in $\mathcal G \times \mathcal G$. 
  For this let $\hat x = (x,\, x')^\top\in \mathcal G \times \mathcal G$. Then there exists $\xi' \in \mathscr G_1$ 
  such that $\iota_+ \xi' = x'$ 
  and $f_0 \in \ker \Gamma_0 = \dom A_0$ such that $\Gamma_1  f_0 = \xi'$. Note that $\ran \Upsilon_0^T$ is dense in $\mathcal G$ since
  $\ran\Gamma_0$ is dense in $\mathcal G$. Hence we find
  a sequence $(f_n)\subset \mathcal N_\eta(T)$ such that $\Upsilon_0^T f_n\rightarrow x$, $n\rightarrow\infty$. It follows from
  $\Upsilon_0^Tf_0=0$ and the definition  of $\Upsilon_1^T$ that 
  \begin{equation*}
  \Upsilon^T ( f_0 + f_n ) 
  = \begin{pmatrix} \Upsilon^T_0 ( f_0 + f_n ) \\ \Upsilon^T_1 ( f_0 + f_n ) \end{pmatrix} 
  = \begin{pmatrix} \Upsilon^T_0 f_n \\ \iota_+\Gamma_1 f_0 \end{pmatrix} 
  = \begin{pmatrix}  \Upsilon^T_0 f_n \\ x' \end{pmatrix} 
  \end{equation*}
  tends to $\hat x$ for $n\rightarrow\infty$. Hence (ii) in Definition~\ref{qbt} holds and it follows that 
  $\{\mathcal G, \Upsilon_0^T,\Upsilon_1^T\}$
  is a quasi boundary triple.

  Now we have to check that $\Upsilon_0^T,\Upsilon_1^T:\dom T\rightarrow\mathcal G$ are continuous with respect to the graph norm.
  It follows from Proposition~\ref{pro:Gamma_0_extension} that this is even true for $\Upsilon_0=\iota_-\widetilde\Gamma_0$,
  and hence also for the restriction $\Upsilon^T_0$. 
  For $f=f_0+f_\eta\in\dom T$ with $f_0 \in \dom A_{0}$ 
  and $f_\eta \in \mathcal N_\eta (T)$ we have
  \begin{equation*}
    \Upsilon^T_1 f 
    = \iota_+  \Gamma_1 f_0 
    =  \iota_+ \Gamma_1 (A_0-\eta)^{-1} (T -\eta) f 
  \end{equation*}
  and together with Proposition~\ref{pro:Gamma_0_extension}~(ii) we conclude that $\Upsilon^T_1$ is also continuous 
  with respect to the graph norm.
  
  It remains to check that $\ker \Upsilon_1 = \dom A \dotplus  \mathcal N_\eta(A^*)$. 
  For the inclusion ``$\subset$'' let $f \in \ker \Upsilon_1$ with $f= f_0 + f_\eta \in \dom A_0 \dotplus\mathcal N_\eta(A^*)$.
  Since $\Gamma_1 f_0=0$ we find
  $f_0\in\dom A_0\cap\ker\Gamma_1=\dom A$ and hence $f\in\dom A\dotplus\mathcal N_\eta(A^*)$. 
  The inclusion ``$\supset$'' follows immediately from $\dom A\subset\ker\Gamma_1$ and $\Gamma_1 f_\eta=0$ for 
  $f_\eta\in\mathcal N_\eta(A^*)$. 
\end{proof}

\begin{Rem}\label{rem:techn}
  We note that the assumption $\eta\in\mathbb R$ in Theorem~\ref{thm:regularisation} can be dropped.
  In fact, if $\eta\in\mathbb C\setminus\mathbb R$
  replace $M(\eta)$ and $\mathcal N_\eta(A^*)$ by $\re M(\eta)$ (see \eqref{def:ReImM}) and
  \begin{equation*}
  \mathcal Q_\eta(A^*):=\{f_\eta+f_{\bar\eta}:f\in\dom A^*\},
  \end{equation*}
  respectively. Here $f=f_{0\eta}+f_\eta=f_{0\bar\eta}+f_{\bar\eta}\in\dom A^*$ with $f_{0\eta},f_{0\bar\eta}\in\dom A_0$ and
  $f_\eta\in\mathcal N_\eta(A^*)$, $f_{\bar\eta}\in\mathcal N_{\bar\eta}(A^*)$.
  Instead of \eqref{gammaid} use the following formula 
  \begin{equation*}
  \Gamma_1 f_0 = \Gamma_1 f - \re M(\eta) \Gamma_0 f, \quad f=f_0+\frac 12(f_{\eta}+f_{\bar\eta})\in \dom A_0 \dotplus \mathcal Q_\eta(A^*),
  \end{equation*}
  when verifying Green's identity in the proof of Theorem~\ref{thm:regularisation}.
\end{Rem}
With the help of the extensions $\widetilde\Gamma_0$ and $\widetilde \Gamma_1$ of the boundary mappings 
$\Gamma_0$ and $\Gamma_1$, respectively, also the $\gamma$-field and Weyl function can be extended by continuity.
Observe that by Theorem~\ref{thm:regularisation} we have $\ker\widetilde\Gamma_0=\ker\Upsilon_0=\dom A_0$ and hence
$\widetilde \Gamma_0 \upharpoonright \mathcal N_\lambda(A^*)$, $\lambda\in\rho(A_0)$, is invertible.

\begin{Def}\label{def:extweyl}
  Let $\{\mathcal G, \Gamma_0, \Gamma_1\}$ be a quasi boundary triple for $T\subset A^*$ with $\gamma$-field $\gamma$, Weyl function $M$ and 
  $A_j = T\upharpoonright\ker\Gamma_j$, $j=0,1$. 
  \begin{enumerate}[\rm (i)] 
  \item Assume that $\mathscr G_1$ is dense in $\mathcal G$ and let 
    $\widetilde \Gamma_0 : \dom A^* \to \mathscr G_1'$ be the continuous extension of $\Gamma_0$ from 
    Proposition~\ref{pro:Gamma_0_extension}. Then the extended $\gamma$-field $\widetilde \gamma$ corresponding to the 
    quasi boundary triple $\{\mathcal G, \Gamma_0, \Gamma_1\}$ is defined by
    \begin{equation*}
    \lambda \mapsto \widetilde \gamma(\lambda) 
    := \bigl(\widetilde \Gamma_0 \upharpoonright \mathcal N_\lambda(A^*)\bigr)^{-1} : \mathscr G_1' \to \mathcal H, \qquad \lambda \in  \rho(A_0).
    \end{equation*}
  \item Assume that $\mathscr G_0$ and $\mathscr G_1$ are dense in $\mathcal G$, that $A_1$ is self-adjoint in $\mathcal H$, and 
    let 
    $\widetilde \Gamma_1 : \dom A^* \to \mathscr G_0'$ be the continuous extension of $\Gamma_1$ from Corollary~\ref{cor:Gamma_1_extension}.
    Then the extended Weyl function $\widetilde M$ corresponding to the 
    quasi boundary triple $\{\mathcal G, \Gamma_0, \Gamma_1\}$ is defined by
    \begin{equation*}
    \lambda \mapsto \widetilde M(\lambda) := \widetilde \Gamma_1 \widetilde \gamma(\lambda) : \mathscr G_1' \to \mathscr G_0', \qquad
    \lambda \in  \rho(A_0).
    \end{equation*}
  \end{enumerate}
\end{Def}
We mention that the values of the extended $\gamma$-field $\widetilde \gamma$ 
are bounded linear operators from $\mathscr G_1'$ to $\mathcal H$,
where $\mathscr G_1$ is equipped with a norm such that {\rm(i)} and {\rm(ii)} in Proposition~\ref{pro:Gamma_0_extension} hold.
If also $\mathscr G_0$ is equipped with a norm such that {\rm(i)} and {\rm(ii)} in Corollary~\ref{cor:Gamma_1_extension} hold
then the values of the extended Weyl function $\widetilde M$ are 
bounded linear operators from $\mathscr G_1'$ to $\mathscr G_0'$. Therefore the adjoints
\begin{equation*}
\widetilde \gamma(\lambda)' : \mathcal H \to \mathscr G_1 \quad\text{and}\quad  \widetilde M(\lambda)' : \mathscr G_0 \to \mathscr G_1
\end{equation*}
are continuous for all $\lambda \in \rho(A_0)$. Moreover we obtain the simple identity 
\begin{equation}\label{eqn:exweylid}
  \widetilde M(\lambda) \widetilde \Gamma_0 f_\lambda = \widetilde\Gamma_1 f_\lambda\quad\text{for all}
  \quad f_\lambda \in \mathcal N_\lambda(A^*),\, \lambda \in \rho(A_0).
\end{equation}
In the next two lemmas some basic, but important, facts about the extended boundary mappings, 
the extended $\gamma$-field and the extended Weyl function are summarized. As above it is assumed that $\mathscr G_1$ is dense in $\mathcal G$ and 
that $\mathscr G_1$ is equipped 
with a norm such that {\rm(i)} and {\rm(ii)} in Proposition~{\rm\ref{pro:Gamma_0_extension}} hold. 
\begin{Lem}\label{lem:exgamma}
  Let $\{\mathcal G, \Gamma_0, \Gamma_1\}$ be a quasi boundary triple for $T\subset A^*$ with $\gamma$-field $\gamma$, and  
  $A_0 = T\upharpoonright\ker\Gamma_0$ such that $\rho(A_0)\cap\mathbb R\not=\emptyset$.
  Assume that $\mathscr G_1$ is dense in $\mathcal G$. Then the following statements hold.
  \begin{enumerate}[\rm (i)]
  \item $\ker \widetilde \Gamma_0 = \ker \Gamma_0 = \dom A_0$,
  \item $\widetilde \gamma(\lambda)$  is an isomorphism from $\mathscr G_1'$ onto $\mathcal N_{\lambda}(A^*) \subset \mathcal H$ 
        for all $\lambda\in \rho(A_0)$,
  \item $\widetilde \gamma (\lambda)'= \Gamma_1 (A_{0}-\bar \lambda)^{-1} : \mathcal H \to \mathscr G_1$ is continuous and 
        surjective for all $\lambda\in \rho(A_0)$,
  \item the identity 
    \begin{equation*}
    \widetilde \gamma(\lambda) = \left( I + (\lambda - \mu) (A_0 - \lambda)^{-1}\right ) \widetilde \gamma(\mu)
    \end{equation*}
    holds for all $\lambda,\mu \in \rho(A_0)$.
  \end{enumerate}
\end{Lem}

\begin{proof}
  Let $\{\mathcal G, \Upsilon_0, \Upsilon_1\}$ be the ordinary boundary triple for $A^*$ from Theorem~\ref{thm:regularisation} 
  and denote the corresponding
  $\gamma$-field with $\beta$.
  Then according to Theorem~\ref{thm:regularisation} statement (i) follows from 
  \begin{equation*}
  \ker \Gamma_0 = \dom A_0 = \ker \Upsilon_0 = \ker \iota_- \widetilde\Gamma_0 = \ker \widetilde \Gamma_0,
  \end{equation*}
  see the text before Definition~\ref{def:extweyl}.
  From Proposition~\ref{pro:Gamma_0_extension} 
  we obtain that $\widetilde \Gamma_0 : (\dom A^*,\, \|\cdot\|_{A^*}) \to \mathscr G_1'$ is continuous and surjective with
  $\ker \widetilde\Gamma_0 = \dom A_0$; cf.~(i). Hence $\widetilde \Gamma_0 : \mathcal N_{\lambda}(A^*) \to \mathscr G_1'$ 
  is bijective and continuous and this implies (ii). 
  The identity
  \begin{equation*}
  \beta(\lambda) = \left(I + (\lambda - \mu) (A_0 - \lambda)^{-1}\right ) \beta(\mu),\qquad \lambda,\mu \in \rho(A_0),
  \end{equation*}
  (see \eqref{eqn:gamma_field}) together with the straightforward computation
  \begin{equation*}
  \beta(\lambda) 
  = (\Upsilon_0 \upharpoonright \mathcal N_\lambda (A^*) )^{-1} 
  = (\iota_- \widetilde\Gamma_0 \upharpoonright \mathcal N_\lambda (A^*) )^{-1} 
  = \widetilde\gamma(\lambda) \iota_-^{-1}
  \end{equation*} 
  implies (iv). To proof statement (iii) we only have to show that the identity 
  $\widetilde \gamma (\lambda)'= \Gamma_1 (A_{0}-\bar \lambda)^{-1}$
  holds.
  With $f\in \mathcal H$ and $x\in \mathcal G$ it follows from
  \begin{eqnarray*}
    ( \beta(\lambda)^* f,\, x)_{\mathcal G} 
    &=& ( f, \beta(\lambda) x)_{\mathcal H} 
     =  ( f, \widetilde \gamma(\lambda) \iota_-^{-1} x)_{\mathcal H} \\
    &=& \langle  \widetilde \gamma(\lambda)' f, \iota_-^{-1} x \rangle_{\mathscr G_1\times \mathscr G_1'} 
     =  ( \iota_+ \widetilde \gamma(\lambda)' f,\iota_- \iota_-^{-1} x)_{\mathcal G} \\
    &=& ( \iota_+ \widetilde \gamma(\lambda)' f, x)_{\mathcal G} 
  \end{eqnarray*}
  that $\iota_+ \widetilde \gamma(\lambda)' = \beta(\lambda)^* = \Upsilon_1 (A_0-\bar\lambda)^{-1} = \iota_+ \Gamma_1 (A_0-\bar\lambda)^{-1}$. 
  Hence we obtain statement (iii).
\end{proof}
\begin{Lem}\label{lem:exweyl}
  Let the assumption be as in Lemma~{\rm\ref{lem:exgamma}} and 
  assume, in addition, that $\mathscr G_0$ is dense in $\mathcal G$ and that $A_1=T\upharpoonright\ker\Gamma_1$ is self-adjoint in $\mathcal H$
  such that $\rho(A_1)\cap\mathbb R\not=\emptyset$. 
  Moreover, equip $\mathscr G_0$ with a norm which satisfies {\rm (i)-(ii)} in Corollary~{\rm\ref{cor:Gamma_1_extension}}.
  Then the following statements hold for all $\lambda \in \rho(A_0)$.
  \begin{enumerate}[\rm(i)]
  \item $\ker \widetilde\Gamma_1 = \ker \Gamma_1 = \dom A_1$,
  \item $\widetilde\Gamma_1 f = \widetilde M(\lambda) \widetilde\Gamma_0 f+ \Gamma_1 f_0$ 
    for all $f=f_0+f_\lambda\in\dom A_0 \dotplus\mathcal N_\lambda(A^*)$,
  \item $\widetilde  M(\lambda)' x = M(\lambda)^* x=M(\bar\lambda)x$ for all $x\in \mathscr G_0$,
  \item if, in addition, $\lambda \in \rho(A_1)$ then $\widetilde M(\lambda) : \mathscr G_1' \to \mathscr G_0'$ and 
        $M(\lambda) \upharpoonright \mathscr G_0 :  \mathscr G_0 \to \mathscr G_1$ are isomorphisms,
  \item the range of the boundary mapping $\widetilde \Gamma$ is given by
    \begin{equation*}
      \ran \widetilde \Gamma = \left\{ \begin{pmatrix} x \\ x' \end{pmatrix} \in \mathscr G_1' \times \mathscr G_0' 
      : x' = \widetilde M(\lambda) x + y,\, y \in \mathscr G_1 \right\}.
    \end{equation*} 
  \end{enumerate}
\end{Lem}
\begin{proof}
  Statement (i) follows in the same way as in Lemma~\ref{lem:exgamma} and from 
  the fact that $\{\mathcal G,-\Gamma_1,\Gamma_0\}$ is a quasi boundary triple for $T\subset A^*$.

  The identity \eqref{eqn:exweylid} together with $f=f_0+f_\lambda\in\dom A_0 \dotplus\mathcal N_\lambda(A^*)$ yields the identity
  \begin{equation*}
  \widetilde\Gamma_1 f = \widetilde\Gamma_1 f_0 + \widetilde\Gamma_1 f_\lambda 
  = \Gamma_1 f_0 + \widetilde M(\lambda) \widetilde \Gamma_0 f_\lambda
  =  \Gamma_1 f_0 + \widetilde M(\lambda) \widetilde \Gamma_0 f,
  \end{equation*}
  therefore (ii) holds; cf.~\eqref{gammaid}. In order to verify (iii) note first that
  according to \eqref{dominv} we have $\mathscr G_0 \subset \ran \Gamma_0 = \dom M(\lambda) = \dom M(\bar\lambda) \subset \dom M(\lambda)^*$.
  For $x\in \mathscr G_0$ and $y\in \ran \Gamma_0 \subset \mathcal G \subset \mathscr G_j'$, $j=0,1$, we compute 
  \begin{eqnarray*}
  ( M(\lambda)^* x,\, y)_{\mathcal G} 
    &=& ( x,\, M(\lambda) y)_{\mathcal G} 
     = \langle x,\, \widetilde M(\lambda) y \rangle_{\mathscr G_0 \times \mathscr G_0'} \\
    &=& \langle \widetilde M(\lambda)' x,\,  y \rangle_{\mathscr G_1 \times \mathscr G_1'} 
     = ( \widetilde M(\lambda)' x,\, y)_{\mathcal G}.
   \end{eqnarray*}
  As $\ran\Gamma_0$ is dense in $\mathcal G$ this implies  $ M(\lambda)^* x = \widetilde M(\lambda)' x$ and $M(\bar\lambda)x=M(\lambda)^*x$ 
  holds by \eqref{mg}-\eqref{dominv}. 

  By Lemma~\ref{lem:exgamma}~(ii) the operator $\widetilde \gamma(\lambda)$ is an isomorphism from $\mathscr G_1'$ onto $\mathcal N_{\lambda}(A^*)$. 
  Since $A_1$ is self-adjoint in $\mathcal H$ we have $\dom A^* = \dom A_1 \dotplus \mathcal N_\lambda(A^*)$ for 
  $\lambda\in\rho(A_1)$. Therefore the first assertion in (iv) follows from (i) and Corollary~\ref{cor:Gamma_1_extension}. 
  The second assertion in (iv) is a consequence of (iii).
  Finally, statement (v) follows from (ii) in the same way as in the proof of Proposition~\ref{pro:weyl2}~(ii). 
\end{proof}

Since $\ker\Gamma_1=\ker\widetilde\Gamma_1$ and $\ker\Gamma_0=\ker\widetilde\Gamma_0$ hold by Lemma~\ref{lem:exweyl}~(i) and Lemma~\ref{lem:exgamma}~(i)
we conclude that the spaces $\mathscr G_0$ and $\mathscr G_1$ in Definition~\ref{def:g0g1} remain the same for the extended
boundary mappings, i.e.,
\begin{eqnarray*}
&&\mathscr G_0 = \ran \bigl(\Gamma_0 \upharpoonright \ker \Gamma_1\bigr) = \ran \bigl(\widetilde \Gamma_0 \upharpoonright \ker \widetilde \Gamma_1\bigr),\\
&&\mathscr G_1 = \ran \bigl(\Gamma_1 \upharpoonright \ker \Gamma_0\bigr) = \ran \bigl(\widetilde \Gamma_1 \upharpoonright \ker \widetilde \Gamma_0\bigr).
\end{eqnarray*}
For later purposes we also note that for a quasi boundary triple $\{\mathcal G, \Gamma_0, \Gamma_1\}$ as in 
Lemma~\ref{lem:exgamma} and \ref{lem:exweyl}, with $\gamma$-field $\gamma$, Weyl function $M$, their extensions
$\widetilde\gamma(\lambda):\mathscr G_1'\rightarrow\mathcal H$ and $\widetilde M(\lambda):\mathscr G_1'\rightarrow\mathscr G_0'$,
and the corresponding ordinary boundary triple $\{\mathcal G, \Upsilon_0, \Upsilon_1\}$ from Theorem~\ref{thm:regularisation} with 
$\gamma$-field $\beta$, Weyl function $\mathcal M$ the following relations hold:
\begin{equation}\label{weihnachten}
 \beta(\lambda)=\widetilde\gamma(\lambda)\iota_-^{-1}\quad\text{and}\quad
 \mathcal M(\lambda) = \iota_+ ( \widetilde M(\lambda) - \widetilde M(\eta) ) \iota_-^{-1},\quad\lambda\in\rho(A_0).
\end{equation}
In fact, the identity $\beta(\lambda)=\widetilde\gamma(\lambda)\iota_-^{-1}$ was already shown in the proof of Lemma~\ref{lem:exgamma}
and the second relation in \eqref{weihnachten} is a direct consequence of the definition of the Weyl function $\mathcal M$, Lemma~\ref{lem:exweyl}~(ii), and the
particular form of the ordinary boundary triple  $\{\mathcal G, \Upsilon_0, \Upsilon_1\}$ in Theorem~\ref{thm:regularisation}. In fact, for
$f_\lambda\in\mathcal N_\lambda(A^*)$ decomposed in the form $f_\lambda=f_0+f_\eta$ with $f_0\in\dom A_0$, $f_\eta\in\mathcal N_\eta(A^*)$, one has
\begin{equation*}
 \begin{split}
 \iota_+\bigl(\widetilde M(\lambda)-\widetilde M(\eta)\bigr)\iota_-^{-1}\Upsilon_0 f_\lambda
 &=\iota_+\bigl(\widetilde M(\lambda)-\widetilde M(\eta)\bigr)\widetilde\Gamma_0 f_\lambda\\
 &=\iota_+ \bigl(\widetilde\Gamma_1 f_\lambda - \widetilde M(\eta)\widetilde\Gamma_0 f_\lambda\bigr)\\
 &=\iota_+\Gamma_1 f_0=\Upsilon_1 f_\lambda.
 \end{split}
\end{equation*}

\subsection{A counterexample}\label{sec:counter}

In this supplementary subsection we show that the assumption $\mathscr G_1^\bot=\{0\}$, which is essential for  Proposition~\ref{pro:DeMa95}, 
Proposition~\ref{pro:Gamma_0_extension}, Corollary~\ref{cor:Gamma_1_extension} and Theorem~\ref{thm:regularisation}, is not satisfied
automatically. For this we construct a quasi boundary triple $\{\mathscr H,\Upsilon_0 ,\Upsilon_1 \}$ with the property
$\mathscr G_1^\bot\not=\{0\}$ as a transform of the quasi boundary triple in Lemma~\ref{lem:BTonNeta}~(ii). 

%
\begin{Pro}\label{gehtschief}
  Let $\{\mathcal N_\eta(A^*),\Gamma_0^T,\Gamma_1^T\}$ be the quasi boundary triple for $T\subset A^*$ 
  from Lemma~{\rm\ref{lem:BTonNeta}~(ii)} with $\varphi =I$, $\mathcal G = \mathcal N_\eta(A^*)$, and 
  let $\mathscr H$ be an auxiliary Hilbert space. Choose a densely defined, bounded operator $\gamma:\mathscr H\rightarrow \mathcal N_\eta(A^*)$
  such that
  \begin{equation*}
  \ker\gamma=\{0\},\quad \ran\gamma=\mathcal N_\eta(T)\quad\text{and}\quad \ker\overline\gamma\not=\{0\},
  \end{equation*}
  and let $M$ be an (unbounded) self-adjoint operator in $\mathscr H$ defined on $\dom\gamma$.
  Then $\{\mathscr H,\Upsilon_0 ,\Upsilon_1 \}$, where
  \begin{equation*}
  \Upsilon_0 f := \gamma^{-1} \Gamma_0^T f,\quad 
  \Upsilon_1 f := \gamma^* \Gamma_1^T f +  M \gamma^{-1} \Gamma_0^T f,\qquad f\in\dom T,
  \end{equation*}
  is a quasi boundary triple for $T\subset A^*$ such that $A_0=T\upharpoonright\ker\Upsilon_0$,
  \begin{equation*}
  \mathscr G_1=\ran(\Upsilon_1\upharpoonright\ker\Upsilon_0)=\ran\gamma^*\quad\text{and}\quad\mathscr G_1^\bot=\ker\overline\gamma\not=\{0\}.
  \end{equation*}
  In particular, if $M(\cdot)$ is the Weyl function corresponding to 
  the quasi boundary triple $\{\mathscr H,\Upsilon_0 ,\Upsilon_1 \}$ then we have 
  $M(\eta)=M$ and $\overline{\im M(\lambda)}$ is not invertible for any $\lambda\in\mathbb C\setminus\mathbb R$. 
\end{Pro}

\begin{proof} We verify that $\{\mathscr H,\Upsilon_0,\Upsilon_1\}$ is a quasi boundary triple for $T\subset A^*$. 
  Since $M$ is self-adjoint in $\mathscr H$ and $\{\mathcal N_\eta(A^*),\Gamma_0^T,\Gamma_1^T\}$ is a quasi boundary triple we have
  \begin{equation*}
    \begin{split}
      &(\Upsilon_1 f,\, \Upsilon_0 g)_{\mathscr H} - (\Upsilon_0 f,\, \Upsilon_1 g)_{\mathscr H}  \\
      &\qquad\qquad = (\gamma^* \Gamma_1^T f,\, \gamma^{-1} \Gamma_0^T g)_{\mathscr H} - (\gamma^{-1} \Gamma_0^T f,\, \gamma^* \Gamma_1^T g)_{\mathscr H}\\
      &\qquad\qquad= ( \Gamma_1^T f,\, \gamma \gamma^{-1} \Gamma_0^T g)_{\mathcal N_\eta(A^*)} - (\gamma \gamma^{-1} \Gamma_0^T f,\Gamma_1^T g)_{\mathcal N_\eta(A^*)}\\ 
      &\qquad\qquad= ( \Gamma_1^T f,\Gamma_0^T g)_{\mathcal N_\eta(A^*)} - (\Gamma_0^T f,\Gamma_1^T g)_{\mathcal N_\eta(A^*)}\\
      &\qquad\qquad= ( T f,\, g)_{\mathcal H} - (f,\, T g)_{\mathcal H}
    \end{split}
  \end{equation*}
  for all $f,\, g\in \dom T$, and hence the abstract Green's identity holds. Observe that 
  \begin{equation*}
  A_0=T\upharpoonright\ker\Gamma_0^T=T\upharpoonright\ker\Upsilon_0
  \end{equation*}
  holds since by assumption $\gamma$ is a bijection from $\dom\gamma$ onto $\mathcal N_\eta(T)$.
  
  Next it will be shown that the range of $\Upsilon := (\Upsilon_0, \Upsilon_1)^\top$ is dense in $\mathscr H\times \mathscr H$.
  Since $\gamma^{-1}$ is a bijection from $\mathcal N_\eta(T)$ onto $\dom\gamma$ we have
  \begin{eqnarray*}
    \ran \Upsilon 
    &=& \left\{ \begin{pmatrix} \gamma^{-1} \Gamma_0^T f \\ 
      \gamma^* \Gamma_1^T f +  M \gamma^{-1} \Gamma_0^T f \end{pmatrix} : f \in \dom T \right\}  \\
    &=& \left\{ \begin{pmatrix} \gamma^{-1} f_\eta \\ 
      \gamma^* \Gamma_1^T f_0 +  M \gamma^{-1} f_\eta \end{pmatrix} : f = f_0 + f_\eta \in \ker \Upsilon_0 \dotplus \mathcal N_\eta(T)  \right\}  \\
    &=& \left\{ \begin{pmatrix} x \\ 
      y +  M x \end{pmatrix}  : x\in \dom \gamma,\, y\in \ran \gamma^* \right\}.
  \end{eqnarray*}
  Here we have used in the last step that $\ran\Gamma_1^T=\mathcal N_\eta(A^*)$ by Lemma~\ref{lem:BTonNeta}~(ii).
  Suppose that $(z,\, z')\in (\ran \Upsilon)^\bot$. Then 
  \begin{equation}\label{eqn:nonhalfextqbt}
    (z,\, x)_{\mathscr H} + (z',\, y)_{\mathscr H}+ (z',\, M x)_{\mathscr H} = 0 
  \end{equation}
  for all $x\in \dom \gamma$ and all $y\in \ran \gamma^*$. We note that if $z'=0$ then $z=0$ as $\dom\gamma$ is dense in $\mathscr H$. 
  Assume first that $z' \in \ker \overline\gamma = (\ran\gamma^*)^{\bot}$. Then $(z',\, y)_{\mathscr H} = 0$, $y\in\ran\gamma^*$, 
  and \eqref{eqn:nonhalfextqbt} yields
  \begin{equation*}
  (z',\, M x)_{\mathscr H} = (-z,\, x)_{\mathscr H}
  \end{equation*}
  for all $x\in \dom M$. As $M$ is self-adjoint we conclude $z'\in \dom M = \dom \gamma$ and from $\ker\gamma=\{0\}$ we find $z'=0$. 
  Assume now that $z'\not\in\ker\overline\gamma= (\ran\gamma^*)^{\bot}$. Then there exists $y\in\ran\gamma^*$ such that $(-z',\, y)_{\mathscr H}\not=0$.
  Since $\dom M$ is dense in $\mathscr H$ there exists a sequence $(z_n')\subset\dom M$ such that $z_n' \to z'$  for $n\to\infty$.
  From \eqref{eqn:nonhalfextqbt} we then obtain 
  \begin{equation*}
    \begin{split}
      \lim_{n\to\infty} (M z_n' + z,\, x)_{\mathscr H}&=(z,\, x)_{\mathscr H} + \lim_{n\to\infty} (z_n',\, M x)_{\mathscr H}\\
      &=(z,\, x)_{\mathscr H} + (z',\, M x)_{\mathscr H}=(-z',\, y)_{\mathscr H}\not=0
    \end{split}
  \end{equation*}
  for all $x\in \dom M$, a contradiction. 
  We conclude $z'=z=0$ and hence $\ran \Upsilon$ is dense in $\mathscr H\times \mathscr H$.
  
  Since $\ker\Upsilon_0=\ker\Gamma_0^T$ and $\ran(\Gamma_1^T\upharpoonright\ker\Gamma_0^T)=\mathcal N_\eta(A^*)$ we have 
\[
\mathscr G_1=\ran(\Upsilon_1\upharpoonright\ker\Upsilon_0)=\ran\bigl(\gamma^*\Gamma_1^T\upharpoonright\ker\Gamma_0^T\bigr)=\ran\gamma^*
\]
and therefore $\mathscr G_1^\bot=\ker\overline\gamma\not=\{0\}$ by assumption. Finally, if $M(\cdot)$ is the Weyl function 
corresponding to the quasi boundary triple $\{\mathscr H,\Upsilon_0 ,\Upsilon_1 \}$ then it follows from $\Gamma_1^T f_\eta=0$, 
$f_\eta\in\mathcal N_\eta(T)$, and
$M \Upsilon_0 f_\eta= M\gamma^{-1}\Gamma_0^T f_\eta= \Upsilon_1 f_\eta$
that $M(\eta)=M$ holds. The fact that $\overline{\text{\rm Im}\, M(\lambda)}$ is not invertible for $\lambda\in\mathbb C\setminus\mathbb R$
is immediate from \eqref{eqn:kerImWeyl}.
\end{proof}
%
%
%
\section{Extensions of symmetric operators}\label{sec:extension}

The main objective of this section is to parameterize the extensions of a symmetric operator $A$ with the help of
a quasi boundary triple $\{\mathcal G, \Gamma_0, \Gamma_1\}$ for $T\subset A^*$. In contrast to ordinary boundary triples
there is no immediate direct connection between the properties of the extensions 
\begin{equation}\label{eqn:Atheta}
  A_\vartheta=T\upharpoonright\bigl\{ f\in \dom T : \Gamma f \in \vartheta \bigr\}
\end{equation}
and the properties of the corresponding parameters $\vartheta$ in $\mathcal G\times\mathcal G$, as, e.g.\ self-adjointness. 
The key idea in Theorem~\ref{thm:main1} and Theorem~\ref{thm:main2} is to mimic a regularization procedure 
which is used in the investigation of elliptic 
differential operators and goes back to \cite{Gr68,Vi63}, see also \cite{BeLa12,BrGrWo09,DeHaMaSn12,GeMi11,Ma10,Po08,PoRa09}. 
This also leads to an abstract complete description of the extensions $A_\vartheta\subset A^*$ via the extended boundary mappings
$\widetilde\Gamma_0$ and $\widetilde\Gamma_1$ in Theorem~\ref{thm:main3}. The general results are illustrated with 
various examples and sufficient conditions on the parameters
to imply self-adjointness, as well as a variant of Kre\u{\i}n's formula is discussed.

\subsection{Parameterization of extensions with quasi boundary triples}\label{sec:parameterization}
Let in the following $A$ be a closed, densely defined, symmetric operator in the Hilbert space $\mathcal H$ with equal, in general, 
infinite deficiency indices.
In the first theorem in this subsection we recall one of the key features of ordinary boundary triples 
$\{\mathcal G,\Gamma_0,\Gamma_1\}$ for $A^*$: A complete description and 
parameterization of the extensions $A_\Theta$ of $A$ given by
\begin{equation*}
A_\Theta :=A^*\upharpoonright\bigl\{ f\in \dom A^* : \Gamma f \in \Theta \bigr\}
\end{equation*}
and their properties in terms of linear relations $\Theta$ 
in the boundary space $\mathcal G$, see, e.g.\ \cite{DeMa91,DeMa95,GoGo91}.
\begin{Thm}\label{thm:boundary_triple}
   Let $\{\mathcal G,\Gamma_0,\Gamma_1\}$ be an ordinary boundary triple for $A^*$. Then the mapping 
   \footnote{Here and in the following the expression $\Gamma_1-\Theta\Gamma_0$ is understood in the sense of linear relations if $\Theta$ is a linear relation,
   that is, $\Theta\Gamma_0$ is the product of the relation $\Theta$ with (the graph of the mapping) $\Gamma_0$ and the sum of $\Gamma_1$ and
   $-\Theta\Gamma_0$ is in sense of linear relations; cf. Appendix for the defintions.}
   \begin{equation*}
   \Theta \mapsto A_\Theta=A^*\upharpoonright\bigl\{ f\in \dom A^* : \Gamma f \in \Theta \bigr\} 
   = A^*\upharpoonright\ker(\Gamma_1-\Theta\Gamma_0)
   \end{equation*}
   establishes a bijective correspondence between the set of closed linear relations $\Theta$ in $\mathcal G$ and the set of closed
   extensions $A_\Theta\subset A^*$ of $A$. 
   Furthermore, 
   \begin{equation*}
   A_{\Theta^*}=A_\Theta^*
   \end{equation*}
   and the operator $A_\Theta$
   is symmetric (self-adjoint, (maximal) dissipative, (maximal) accumulative) in $\mathcal H$
   if and only if the closed linear relation $\Theta$ 
   is symmetric (self-adjoint, (maximal) dissipative, (maximal) accumulative, respectively) in $\mathcal G$.
\end{Thm} 

Not surprisingly Theorem~\ref{thm:boundary_triple} does not hold for quasi boundary triples 
$\{\mathcal G,\Gamma_0,\Gamma_1\}$, see, e.g.\ \cite[Proposition~4.11]{BeLa07} for a counterexample. 
In particular,  
$\vartheta=\{0\}\times\mathscr G_1 \subset\ran \Gamma$ (see Definition~\ref{def:g0g1} and Proposition~\ref{pro:weyl2}~(ii)) 
is symmetric and not self-adjoint in $\mathcal G$ but the corresponding extension $A_\vartheta$ of $A$ in \eqref{eqn:Atheta} 
coincides with the self-adjoint operator $A_0=T\upharpoonright\ker\Gamma_0$ in $\mathcal H$. 
Note that for a quasi boundary triple $\{\mathcal G,\Gamma_0,\Gamma_1\}$ the range of the boundary map $\Gamma=(\Gamma_0,\Gamma_1)^\top$ is
only dense in $\mathcal G\times\mathcal G$, so that for a linear relation $\vartheta$ in $\mathcal G$ only the 
part $\vartheta\cap\ran\Gamma$ can be ``detected'' by the boundary maps. 
However, even for a self-adjoint linear relation $\vartheta\subset\ran\Gamma$
the corresponding extension $A_\vartheta$ of $A$ in \eqref{eqn:Atheta} is in general not self-adjoint, see Example~\ref{cex:1}.
Nevertheless, the following weaker statement is a direct consequence of the abstract Green's identity \eqref{eqn:abstract_green}; 
cf.~\cite[Proposition 2.4]{BeLa07}.
\begin{Lem}\label{lem:symmetric}
  Let $\{\mathcal G,\Gamma_0,\Gamma_1\}$ be a quasi boundary triple 
  for $T\subset A^*$.
  Then the mapping 
  \begin{equation*}
  \vartheta\mapsto A_\vartheta=T\upharpoonright\bigl\{ f\in \dom T : \Gamma f \in \vartheta \bigr\} 
  \end{equation*}
  establishes a bijective correspondence between the set of symmetric linear relations $\vartheta\subset\ran\Gamma$ in $\mathcal G$ and the set of 
  symmetric extensions  $A_\vartheta\subset T$ of $A$ in $\mathcal H$.
\end{Lem}
We also mention that for a quasi boundary triple $\{\mathcal G,\Gamma_0,\Gamma_1\}$ and linear relations 
$\theta \subset \vartheta \subset\ran \Gamma$ one has $A_{\theta}\subset A_{\vartheta} \subset T$; cf.~\eqref{eqn:Atheta}.

In the next theorem we make use of a different
type of parameterization to characterize the restrictions of $T$ with the help of a quasi boundary triple. 
The idea of the proof is to relate the given quasi boundary triple 
$\{\mathcal G,\Gamma_0,\Gamma_1\}$ to the quasi boundary triple in Lemma~\ref{lem:BTonNeta}~(ii) 
and to transform the parameters accordingly.
We also point out that in contrast to most of the results in Section~\ref{2.3} here it is not assumed that the space
$\mathscr G_1=\ran(\Gamma_1\upharpoonright\ker\Gamma_0)$ 
is dense in $\mathcal G$.
\begin{Thm}\label{thm:main1}
  Let $\{\mathcal G, \Gamma_0, \Gamma_1\}$ be a quasi boundary triple for $T\subset A^*$ 
  with $\gamma$-field $\gamma$ and Weyl function $M$. 
  Assume that for $A_0 = T \upharpoonright\ker\Gamma_0$ there exists $\eta \in \rho(A_0)\cap\mathbb R$ 
  and fix a unitary operator $\varphi:\mathcal N_\eta(A^*)\rightarrow\mathcal G$.
  Then the mapping 
  \begin{equation*}
    \Theta\mapsto A_\vartheta = T\upharpoonright\bigl\{ f\in \dom T : \Gamma f \in \vartheta \bigr\}
    \quad\text{with}\quad \vartheta = \gamma(\eta)^* \varphi^* \Theta \varphi  \gamma(\eta) + M(\eta)
  \end{equation*}
  establishes a bijective correspondence between all
  closed (symmetric, self-adjoint, (maximal) dissipative, (maximal) accumulative) 
  linear relations $\Theta$ in $\mathcal G$ 
  with $\dom \Theta \subset \ran (\varphi\upharpoonright \mathcal N_\eta (T))$ and 
  all closed (symmetric, self-adjoint, (maximal) dissipative, (maximal) accumulative, respectively)
  extensions $A_\vartheta\subset T$ of $A$ in $\mathcal H$.
\end{Thm}

\begin{proof}
  Let $\Theta$ be a linear relation in $\mathcal G$ 
  and decompose $f\in \dom T$ in  $f=f_0+ f_\eta$, where $f_0 \in \dom A_0$ 
  and $f_\eta \in \mathcal N_\eta(T)$.
  Then $\Gamma f \in \gamma(\eta)^* \varphi^* \Theta \varphi \gamma(\eta) + M(\eta)$ is equivalent to 
  \begin{equation*}
    \Gamma_1 f =  \gamma(\eta)^*  \varphi^* x + M(\eta) \Gamma_0 f \quad \text{with} \quad
    \begin{pmatrix}\varphi \gamma(\eta) \Gamma_0 f \\ x\end{pmatrix} \in \Theta, 
  \end{equation*}
  and by \eqref{gammaid} this can be rewritten as 
  \begin{equation}\label{abcd}
    \Gamma_1 f_0=\gamma(\eta)^* \varphi^*  x  \quad \text{with}\quad 
    \begin{pmatrix} \varphi f_\eta \\ x\end{pmatrix} \in \Theta.
  \end{equation}
  Denote the orthogonal projection in $\mathcal H$ onto $\mathcal N_\eta(A^*)$ by $P_\eta$. 
  Making use of \eqref{eqn:gamma_field} and \eqref{eqn:Hdec} we find
  \begin{equation*}
  \Gamma_1 f_0 = \gamma(\eta)^* (A_0-\eta) f_0 = \gamma(\eta)^* P_\eta(A_0-\eta) f_0
  \end{equation*}
  and as $\gamma(\eta)^*\upharpoonright\mathcal N_\eta(A^*)$ is invertible we conclude together with \eqref{abcd}
  \begin{equation}\label{equiv}
  \Gamma f \in \gamma(\eta)^* \varphi^* \Theta \varphi \gamma(\eta) + M(\eta) \quad \text{if and only if} \quad 
  \begin{pmatrix}\varphi f_\eta \\ \varphi P_\eta  (A_0-\eta)f_0 \end{pmatrix}\in \Theta
  \end{equation}
  for all $f=f_0+f_\eta\in\dom T$.

  According to Proposition~\ref{pro:qbt_restriction} and Lemma~\ref{lem:BTonNeta} the quasi boundary 
  triple $\{\mathcal G,f\mapsto \varphi f_\eta,f\mapsto \varphi P_\eta  (A_0-\eta)f_0 \}$ is the restriction of the 
  ordinary boundary
  triple $\{\mathcal G,f\mapsto \varphi f_\eta,f\mapsto \varphi P_\eta  (A_0-\eta)f_0 \}$ for $A^*$. Now the statement is a consequence of
  Theorem~\ref{thm:boundary_triple}. In fact, if e.g.\ $\Theta$ is self-adjoint in $\mathcal G$ 
  with $\dom \Theta \subset \ran (\varphi\upharpoonright \mathcal N_\eta (T))$, then by Theorem~\ref{thm:boundary_triple} the operator 
\begin{equation}\label{schoeneformel}  
  A^*\upharpoonright\left\{f_0+f_\eta=\dom A_0\dotplus\mathcal N_\eta(A^*) : \begin{pmatrix}\varphi f_\eta \\ 
    \varphi P_\eta  (A_0-\eta)f_0 \end{pmatrix}\in \Theta\right\}
\end{equation}
  is a self-adjoint restriction of $A^*$ in $\mathcal H$. 
  As $\dom \Theta \subset \ran (\varphi\upharpoonright \mathcal N_\eta (T))$ we conclude that the domain of the operator in 
  \eqref{schoeneformel}
  is contained in $\dom T$. 
  Hence by \eqref{equiv} the operator in \eqref{schoeneformel} can be written as
  \begin{equation}\label{schoeneformel2}
    A_\vartheta = T\upharpoonright\bigl\{ f\in \dom T : \Gamma f \in \vartheta \bigr\}
    \quad\text{with}\quad \vartheta = \gamma(\eta)^* \varphi^* \Theta \varphi  \gamma(\eta) + M(\eta)
  \end{equation}
  and $A_\vartheta$ is a self-adjoint operator in $\mathcal H$. 
Conversely, by Theorem~\ref{thm:boundary_triple} for any self-adjoint extension $A_\vartheta$ 
  of $A$ which is contained in $T$ there exists a self-adjoint relation $\Theta$ in $\mathcal G$ such that $A_\vartheta$ can be 
  written in the form \eqref{schoeneformel}, 
  where $\mathcal N_\eta(A^*)$ can be replaced by $\mathcal N_\eta(T)$. 
  Therefore $\dom \Theta \subset \ran (\varphi\upharpoonright \mathcal N_\eta (T))$ and
  together with \eqref{equiv} we conclude that $A_\vartheta$ can be written in the form \eqref{schoeneformel2}.
\end{proof}
The next theorem is of similar flavor as Theorem~\ref{thm:main1} but more explicit and relevant for elliptic boundary value
problems; cf.~Section~\ref{sec:applications}. 
Under the additional assumption that the space $\mathscr G_1=\ran(\Gamma_1\upharpoonright\ker\Gamma_0)$ in Definition~\ref{def:g0g1} 
is dense in $\mathcal G$ 
a more natural parameterization of the extensions is found. 
Here we will again make use of the Gelfand triple 
$\mathscr G_1\hookrightarrow\mathcal G\hookrightarrow\mathscr G_1'$ and the corresponding
isometric isomorphisms $\iota_+$ and $\iota_-$ in \eqref{iotas}. 
We also note that after suitable modifications the assumption $\eta\in\mathbb R$ can be dropped, see Remark~\ref{rem:techn}.
\begin{Thm}\label{thm:main2}
  Let $\{\mathcal G,\Gamma_0,\Gamma_1\}$ be a quasi boundary triple for $T\subset A^*$ with $A_0 = T \upharpoonright \ker\Gamma_0$
  and Weyl function $M$. Assume that there exists $\eta \in \rho(A_0)\cap\mathbb R$ and that  $\mathscr G_1$ is dense in $\mathcal G$.
  Then the mapping 
  \begin{equation*}
    \Theta\mapsto A_\vartheta = T\upharpoonright\bigl\{ f\in \dom T : \Gamma f \in \vartheta \bigr\}
    \quad\text{with}\quad \vartheta = \iota_+^{-1} \Theta \iota_- + M(\eta) 
  \end{equation*}
  establishes a bijective correspondence between all
  closed (symmetric,  self-adjoint, (maximal) dissipative, (maximal) accumulative) 
  linear relations $\Theta$ in $\mathcal G$ 
  with $\dom \Theta \subset \ran \iota_- \Gamma_0$ and 
  all closed (symmetric, self-adjoint, (maximal) dissipative, (maximal) accumulative, respectively) extensions $A_\vartheta\subset T$ of $A$ 
  in $\mathcal H$.
\end{Thm}

\begin{proof}
  Let $\Theta$ be a linear relation in $\mathcal G$ 
  and decompose $f\in \dom T$ in the form $f=f_0 + f_\eta$ with $f_0 \in \dom A_0$ 
  and $f_\eta \in \mathcal N_\eta(T)$.
  Then $\Gamma f \in \iota_+^{-1} \Theta \iota_- + M(\eta)$ if and only if 
  \begin{equation}\label{eqn:main2-1}
    \Gamma_1 f =  \iota_+^{-1} x + M(\eta) \Gamma_0 f \quad \text{with} \quad
    \begin{pmatrix}\iota_- \Gamma_0 f \\ x\end{pmatrix} \in \Theta. 
  \end{equation}
  Equation \eqref{gammaid} implies $\Gamma_1 f - M(\eta) \Gamma_0 f = \Gamma_1 f_0$ and since $f\in \dom T$ we have  
  $\Gamma_0 f = \widetilde \Gamma_0 f$, where $\widetilde\Gamma_0$ is the continuous extension of $\Gamma_0$ to $\dom A^*$ 
  from Proposition~\ref{pro:Gamma_0_extension}. 
  Hence \eqref{eqn:main2-1} is equivalent to
  \begin{equation}
    \begin{pmatrix}\iota_- \widetilde \Gamma_0 f \\ \iota_+ \Gamma_1 f_0\end{pmatrix}\in \Theta.
  \end{equation}
  According to Theorem~\ref{thm:regularisation} the triple 
  $\{\mathcal G,\,  f\mapsto \iota_-\widetilde \Gamma_0 f ,\, f\mapsto  \iota_+ \Gamma_1 f_0 \}$ is an ordinary boundary triple for $A^*$.
  Now the statement follows from Theorem~\ref{thm:boundary_triple} in the same form as in the proof
  of Theorem~\ref{thm:main1}. 
\end{proof}

\begin{Cor}\label{cor:main2} 
  Let the assumptions be as in Theorem~\ref{thm:main2} and let 
  $\vartheta$ be a linear relation in $\mathcal G$.
  Then the extension $A_{\vartheta}$ of $A$ in $\mathcal H$ given by
  \begin{equation}\label{eqn:Avartheta}
  A_{\vartheta} = T\upharpoonright \bigl\{ f\in \dom T : \Gamma f \in \vartheta \bigr\}
  \end{equation}
  is closed (symmetric,  self-adjoint, (maximal) dissipative, (maximal) accumulative) in $\mathcal H$ 
  if and only if the linear relation 
  \begin{equation*} 
  \Theta = \iota_+ (\vartheta - M(\eta) ) \iota_-^{-1}\quad\text{with} \quad \dom \Theta \subset \ran \iota_- \Gamma_0
  \end{equation*}
  is closed (symmetric,  self-adjoint, (maximal) dissipative, (maximal) accumulative) in $\mathcal G$.
\end{Cor}

\begin{proof}
  ($\Rightarrow$) Assume that $A_\vartheta$ in \eqref{eqn:Avartheta} 
  is a closed (symmetric, self-adjoint, (maximal) dissipative, (maximal) accumulative) operator in $\mathcal H$.
  According to Theorem~\ref{thm:main2} 
  there exists a closed (symmetric,  self-adjoint, (maximal) dissipative, (maximal) accumulative, respectively) 
  linear relation 
  $\Theta$ in $\mathcal G$ with $\dom \Theta \subset \ran \iota_- \Gamma_0$
  and
  \begin{equation}\label{eqn:vartheta}
    A_\vartheta = A_\theta = T\upharpoonright\bigl\{f\in\dom T:\Gamma f\in\theta\bigr\}\quad\text{with}\quad 
    \theta = \iota_+^{-1} \Theta \iota_- + M(\eta).
  \end{equation}
  From $\iota_+^{-1} \Theta \iota_- \subset \ran \Gamma_0 \times \mathscr G_1$ and Proposition~\ref{pro:weyl2}~(ii) we conclude 
  $\theta \subset \ran \Gamma$. Furthermore, we have $\theta = \vartheta\cap \ran \Gamma$, (see the text below Lemma~\ref{lem:symmetric}).
  Solving equation \eqref{eqn:vartheta} leads to the identity
  \begin{equation*}
  \Theta = \iota_+ (\theta -  M(\eta) ) \iota_-^{-1} = \iota_+ (\vartheta -  M(\eta) ) \iota_-^{-1}.
  \end{equation*}
  ($\Leftarrow$) Let $\Theta = \iota_+ (\vartheta -  M(\eta) ) \iota_-^{-1}$ with $\dom \Theta \subset \ran \iota_- \Gamma_0$
  be a closed (symmetric,  self-adjoint, (maximal) dissipative, (maximal) accumulative) linear relation in $\mathcal G$.
  From $\vartheta -  M(\eta)=\iota_+^{-1}\Theta\iota_- \subset \ran\Gamma_0\times\mathscr G_1$ and Proposition~\ref{pro:weyl2}~(ii)  
  we obtain $\theta = \iota_+^{-1} \Theta \iota_- + M(\eta)$ with $\theta=\vartheta\cap\ran\Gamma$. 
  According to Theorem~\ref{thm:main2} the extension $A_\theta = A_\vartheta$ given by \eqref{eqn:Avartheta} 
  is closed (symmetric,  self-adjoint, (maximal) dissipative, (maximal) accumulative) in $\mathcal H$.
\end{proof}
We recall that a symmetric linear relation $\Theta$ in $\mathcal G$ with $\ran \Theta = \mathcal G$ is self-adjoint 
in $\mathcal G$ with $0\in\rho(\Theta)$. This together with  Corollary~\ref{cor:main2} yields the following example.
\begin{Exm}\label{exm:assumption_vartheta}
  Let the assumptions be as in Corollary~\ref{cor:main2} and
  let $\vartheta$ be a symmetric linear relation in $\mathcal G$ such that $\ran (\vartheta - M(\eta) ) = \mathscr G_1$. Then
  \begin{equation*}
  A_{\vartheta} = T\upharpoonright \bigl\{ f\in \dom T : \Gamma f \in \vartheta \bigr\}
  \end{equation*}
  is a self-adjoint extension of $A$ in $\mathcal H$. 
\end{Exm}
In the next result the assumptions on the quasi boundary triple are strengthened further such that both boundary maps $\Gamma_0$
and $\Gamma_1$ extend by continuity to $\dom A^*$. In that case one obtains a description of all extensions $A_\vartheta\subset A^*$
which is very similar to the parameterization in Theorem~\ref{thm:main2}. 
The additional abstract regularity result will turn out
to be useful when considering the regularity of solutions of elliptic boundary value problems in Section~\ref{sec:applications}.
\begin{Thm}\label{thm:main3} 
  Let the assumptions be as in Theorem~{\rm\ref{thm:main2}} 
  and assume, in addition, that $A_1 = T \upharpoonright \ker\Gamma_1$
  is self-adjoint in $\mathcal H$, $\eta \in \rho(A_0)\cap\rho(A_1)\cap \mathbb R$, and that $\mathscr G_0$ dense in $\mathcal G$.
  Let $\widetilde M$ be the extension of the Weyl function $M$ from Definition~{\rm\ref{def:extweyl}~(ii)}. 
  Then the mapping 
  \begin{equation*}
    \Theta\mapsto \widetilde A_\vartheta = A^*\upharpoonright\bigl\{ f\in \dom A^* : \widetilde \Gamma f \in \vartheta \bigr\}
    \quad\text{with}\quad \vartheta = \iota_+^{-1} \Theta \iota_- + \widetilde M(\eta)
  \end{equation*}
  establishes a bijective correspondence between all
  closed (symmetric,  self-adjoint, (maximal) dissipative, (maximal) accumulative) 
  linear relations $\Theta$ in $\mathcal G$ and 
  all closed (symmetric,  self-adjoint, (maximal) dissipative, (maximal) accumulative, respectively) extensions 
  $\widetilde A_\vartheta\subset A^*$ of $A$ in $\mathcal H$.
  
  Moreover, the following abstract regularity result holds: If $\Theta$ is a linear relation in $\mathcal G$ and 
  $S$ is an operator in $\mathcal H$ such that $T \subset S \subset A^*$ then
  \begin{equation*}
  \dom \Theta \subset \ran \bigl(\iota_{-} \widetilde \Gamma_0\upharpoonright {\dom S}\bigr) \quad \text{implies} \quad 
  \dom \widetilde A_{\vartheta} \subset \dom S.
  \end{equation*} 
\end{Thm}

\begin{proof}
  The proof of the first part is very similar to the proof of Theorem~\ref{thm:main2} and will not be repeated here. 
  We show the abstract regularity result. Let $\Theta$ and $S$ be as in the theorem and assume that $\dom\Theta$ is contained
  in the range of the map $\iota_{-} \widetilde \Gamma_0\upharpoonright {\dom S}$. Let 
  \begin{equation*}
  \widetilde A_\vartheta =A^*\upharpoonright \bigl\{ f\in \dom A^* : \widetilde \Gamma f\in \iota_+^{-1} \Theta \iota_- + \widetilde M(\eta)\bigr\}
  \end{equation*} 
  be the corresponding extension and let $f\in \dom \widetilde A_\vartheta$. 
  As $\widetilde \Gamma f \in  \iota_+^{-1} \Theta \iota_- + \widetilde M(\eta)$ 
  we have  $\iota_- \widetilde \Gamma_0 f \in \dom \Theta$. 
  Since $\dom \Theta \subset \ran (\iota_{-} \widetilde \Gamma_0\upharpoonright {\dom S})$ there exists an element $g\in \dom S$ such that 
  $\iota_- \widetilde \Gamma_0 f=\iota_- \widetilde \Gamma_0 g$ holds. Hence we conclude 
  $f-g\in \ker\widetilde \Gamma_0 = \dom A_0 \subset \dom S$, so that $f = g + (f-g) \in \dom S$.   
\end{proof}
The next corollary is a counterpart of Corollary~\ref{cor:main2} and can be proved in the same way using 
Lemma~\ref{lem:exweyl}~(v) instead of Proposition~\ref{pro:weyl2}~(ii).

\begin{Cor}\label{cor:main3}
 Let the assumptions be as in Theorem~{\rm\ref{thm:main3}} and let $\vartheta$ be a linear relation in
 $\mathscr G_1'\times \mathscr G_0'$.
 Then the extension $\widetilde A_{\vartheta}$ of $A$ in $\mathcal H$ given by
 \begin{equation*}
 \widetilde A_{\vartheta} = A^* \upharpoonright \bigl\{ f\in \dom A^* : \widetilde \Gamma f \in \vartheta \bigr\}
 \end{equation*}
 is closed (symmetric,  self-adjoint, (maximal) dissipative, (maximal) accumulative) in $\mathcal H$ 
 if and only if the linear relation 
 \begin{equation*} 
 \Theta = \iota_+ (\vartheta - \widetilde M(\eta) ) \iota_-^{-1} 
 \end{equation*}
 is closed (symmetric,  self-adjoint, (maximal) dissipative, (maximal) accumulative) in $\mathcal G$. 
\end{Cor}

A simple application of Theorem~\ref{thm:main3} is discussed in the next example.
\begin{Exm}\label{exm:kreinneumann}
  Set $\Theta=0$ in Theorem~\ref{thm:main3}. Then $\vartheta=\widetilde M(\eta)$ and it follows that  
  \begin{equation*}
  \widetilde A_\vartheta 
  = A^*\upharpoonright \bigl\{ f\in \dom A^* : \widetilde M(\eta) \widetilde \Gamma_0 f = \widetilde \Gamma_1 f \bigr\}
  \end{equation*}
  is a self-adjoint extension of $A$ in $\mathcal H$. 
  From Lemma~\ref{lem:exweyl}~(ii) we obtain that the condition 
  $\widetilde M(\eta) \widetilde \Gamma_0 f = \widetilde \Gamma_1 f$
  is equivalent to $\Gamma_1 f_0 =0$, where $f=f_0+f_\eta\in\dom A_0\dotplus\mathcal N_\eta(A^*)$. 
  This implies that $\widetilde A_\vartheta = A\dotplus\widehat{\mathcal N}_\eta(A^*)$, which coincides with the Kre\u{\i}n-von Neumann extension 
  if $A$ is uniformly positive and $\eta=0$.
\end{Exm}

\subsection{Sufficient conditions for self-adjointness and a variant of Kre\u{\i}n's formula}\label{sec:krein}

In this subsection we provide different sufficient conditions for the parameter $\vartheta$ in $\mathcal G\times\mathcal G$ 
such that the corresponding extension 
\begin{equation*}
  A_\vartheta =T\upharpoonright \bigl\{ f\in \dom T : \Gamma f\in \vartheta \bigr\},\quad 
  \vartheta= \iota_+^{-1} \Theta \iota_- + M(\eta),
\end{equation*}
in Theorem~\ref{thm:main2} becomes self-adjoint in $\mathcal H$; cf. 
\cite[Theorem~4.8]{BeLa07}, \cite[Theorem~3.11]{BeLaLo11} and, e.g.\ Example~\ref{exm:assumption_vartheta}. 
In Proposition~\ref{pro:assumptionpara} below
we will make use standard perturbation results as the Kato-Rellich theorem, thus we will restrict ourselves to 
operators $\vartheta$ instead of relations.
Recall also the following notions from perturbation theory: If $\mathfrak M$ is a linear operator acting between two Banach spaces
then a sequence $(x_k)_{k\in\mathbb N}\subset\dom \mathfrak M$ is called {\it $\mathfrak M$-bounded} if $(x_k)_{k\in\mathbb N}$ 
is bounded with respect to the 
graph norm of $\mathfrak M$. 
A linear operator $\theta$ is said to {\it relatively compact} with respect to $\mathfrak M$ if $\dom  \mathfrak M\subset \dom \theta$ 
and $\theta$ maps $\mathfrak M$-bounded sequences into sequences which have convergent subsequences.

\begin{Pro}\label{pro:assumptionpara}
  Let $\{\mathcal G,\Gamma_0,\Gamma_1\}$ be a quasi boundary triple for $T\subset A^*$ with $A_j = T \upharpoonright \ker\Gamma_j$, $j=0,1$,
  and Weyl function $M$, and assume that $A_1$ is self-adjoint in $\mathcal H$ and that there exists $\eta \in \rho(A_0)\cap \rho(A_1) \cap\mathbb R$.
  Furthermore, suppose that
  $\mathscr G_0$ and $\mathscr G_1$ are dense in $\mathcal G$ and equip $\mathscr G_0$ and $\mathscr G_1$ with norms which 
  satisfy {\rm (i)-(ii)}
  in Corollary~{\rm \ref{cor:Gamma_1_extension}} and Proposition~{\rm \ref{pro:Gamma_0_extension}}, respectively.  
  
  If $\vartheta$ is a symmetric operator in $\mathcal G$ such that 
  \begin{equation}\label{thetatheta}
  \mathscr G_0 \subset \dom \vartheta\quad\text{and}\quad \ran \vartheta\upharpoonright \mathscr G_0 \subset \mathscr G_1,
  \end{equation}
  and one of the followings conditions {\rm (i)-(iii)} hold,
\begin{enumerate}[\rm (i)]
  \item $\vartheta$ regarded as an operator from $\mathscr G_0$ to $\mathscr G_1$ is compact,
  \item $\vartheta$ regarded as an operator from $\mathscr G_0$ to $\mathscr G_1$ is relatively compact with respect 
    to $M(\eta)$ regarded as an operator from $\mathscr G_0$ to $\mathscr G_1$,
  \item there exist $c_1>0$ and $c_2\in[0,\, 1)$ such that 
    \begin{equation*}
      \| \vartheta x\|_{\mathscr G_1} 
      \leq c_1 \| x\|_{\mathscr G_1'} + c_2 \| M(\eta) x\|_{\mathscr G_1},\qquad x\in \mathscr G_0, 
    \end{equation*}
  \end{enumerate}
  then $A_{\vartheta} = T\upharpoonright \{ f\in \dom T : \Gamma f \in \vartheta \}$
  is self-adjoint in $\mathcal H$.
\end{Pro}

\begin{proof} 
  Note first that condition (i) is a special case of condition (ii). Hence it suffices to prove the proposition under conditions (ii) or (iii). 
  By \eqref{thetatheta} the restriction $\theta:=\vartheta\upharpoonright\mathscr G_0$ maps into $\mathscr G_1$ and the corresponding extensions
  of $A$ in $\mathcal H$ satisfy $A_\theta\subset A_\vartheta$. 
  We show below that (ii) or (iii) imply the self-adjointness of $A_\theta$ 
  and hence, as $A_\vartheta$ is symmetric by Lemma~\ref{lem:symmetric}, the self-adjointness of $A_\vartheta$. 

  By Corollary~\ref{cor:main2} the operator $A_{\theta} = T\upharpoonright \{ f\in \dom T : \Gamma f \in\theta\}$ 
  is self-adjoint in $\mathcal H$ 
  if and only if $\Theta = \iota_+ (\theta - M(\eta) ) \iota_-^{-1}$ is self-adjoint in $\mathcal G$. 
  Since $\vartheta$ is assumed to be a symmetric operator the same holds for $\theta$, $\iota_+ \theta \iota_-^{-1}$ and $\Theta$.
  From Lemma~\ref{lem:exweyl}~(iv) we obtain that $\mathfrak M := M(\eta)\upharpoonright\mathscr G_0$ is an isomorphism onto $\mathscr G_1$.
  Thus the symmetric operator $-\iota_+ \mathfrak M \iota_-^{-1}$ defined on $\iota_- \mathscr G_0$ is surjective and hence
  self-adjoint in  $\mathcal G$. Therefore 
  \begin{equation}\label{grossthetawunderbar}
  \Theta = \iota_+ (\theta -  \mathfrak M ) \iota_-^{-1} 
  = -  \iota_+  \mathfrak M \iota_-^{-1} + \iota_+ \theta \iota_-^{-1} 
  \end{equation}
  can be regarded as an additive symmetric perturbation of the self-adjoint operator $-  \iota_+  \mathfrak M \iota_-^{-1}$, and the assertion of the
  proposition holds if we show that $\Theta$ is self-adjoint in $\mathcal G$.
  
  Assume first that condition (ii) holds, that is, $\theta$ is relatively compact with respect to $\mathfrak M$, and hence also 
  with respect to $-\mathfrak M$. 
  Making use of the fact that
  $\iota_+ : \mathscr G_1 \to \mathcal G$ and $\iota_- :\mathscr G_1' \to \mathcal G$
  are isometric isomorphisms it is not difficult to verify that $\iota_+ \theta \iota_-^{-1}$ is relatively compact with respect to 
  $-  \iota_+  \mathfrak M \iota_-^{-1}$ in $\mathcal G$. Hence by well known perturbation results the operator
  $\Theta$ in \eqref{grossthetawunderbar} is self-adjoint in $\mathcal G$, see, e.g.\ \cite[Theorem~9.14]{We00}.
  
  Suppose now that (iii) holds and set $\xi=\iota_- x$ for $x\in\mathscr G_0$. Then
   \begin{equation*}
    \| \iota_+ \theta \iota_-^{-1} \xi \|_{\mathcal G} = \|  \theta x\|_{\mathscr G_1} \leq 
   c_1 \| x\|_{\mathscr G_1'} + c_2 \| \mathfrak M x\|_{\mathscr G_1} =
   c_1 \| \xi \|_{\mathcal G} + c_2 \| \iota_+  \mathfrak M \iota_-^{-1} \xi \|_{\mathcal G} 
  \end{equation*}
  shows that the symmetric operator $\iota_+ \theta \iota_-^{-1}$ is $\mathcal \iota_+ \mathfrak M\iota_-^{-1}$-bounded with
  a relative bound $c_2<1$. Hence the Kato-Rellich theorem \cite[Theorem~X.12]{ReSi75} implies that $\Theta$ in \eqref{grossthetawunderbar} is a self-adjoint 
  operator in $\mathcal G$.
 \end{proof}
The next proposition is of the same flavor as Proposition~\ref{pro:assumptionpara}. It can be proved similarly with the help of a
variant of the Kato-Rellich theorem due to W\"ust; cf.~\cite[Theorem~X.14]{ReSi75} and \cite{W71}. 

\begin{Pro}\label{pro:assumptionpara2}
  Let the assumptions be as in Proposition~{\rm\ref{pro:assumptionpara}} and assume 
  that there exists $c > 0$ such that 
  \begin{equation*}
    \| \vartheta x\|_{\mathscr G_1} 
    \leq c \| x\|_{\mathscr G_1'} + \| M(\eta) x\|_{\mathscr G_1},\qquad x\in \mathscr G_0.
  \end{equation*}
  Then $A_{\vartheta} = T\upharpoonright \{ f\in \dom T : \Gamma f \in \vartheta \}$ is
  essentially self-adjoint in $\mathcal H$.
\end{Pro}

\begin{Exm}\label{exm:kreinneumann2}
  Let $\vartheta$ be a symmetric operator in $\mathcal G$ with $\mathscr G_0 \subset \dom \vartheta$, such that
  $\vartheta$ is continuous from $(\mathscr G_0,\, \|\cdot\|_{\mathscr G_1'})$ to $\mathscr G_1$.
  Then condition (iii) in Proposition~\ref{pro:assumptionpara} is satisfied with $c_2=0$ and hence the
  extension $A_\vartheta$ of $A$ is self-adjoint in $\mathcal H$.
  
  Now consider $\vartheta := M(\eta) \upharpoonright \mathscr G_0$ as an operator from $\mathscr G_0$ to $\mathscr G_1$.
  Then Proposition~\ref{pro:assumptionpara2} implies that 
  $A_\vartheta$ is essentially self-adjoint in $\mathcal H$. In fact, as in Example~\ref{exm:kreinneumann} one verifies 
  $A_\vartheta = A\dotplus\widehat{\mathcal N}_\eta(T)$, which is a proper restriction of 
  $\widetilde A_\vartheta = A\dotplus\widehat{\mathcal N}_\eta(A^*)$ from Example~\ref{exm:kreinneumann}.
\end{Exm}
For completeness we provide  a version of Kre\u{\i}n's formula for quasi boundary triples in Corollary~\ref{cor:krein} 
which can be viewed as a direct consequence of
Kre\u{\i}n's formula for the ordinary boundary triple in Theorem~\ref{thm:regularisation}. A similar
type of resolvent formula can also be found in \cite[Theorem 7.26]{DeHaMaSn12} for generalized boundary triples.
For the convenience of the reader we first recall Kre\u{\i}n's formula for ordinary boundary triples, see, e.g.\ \cite{DeMa91}.
For the definition of the point, continuous and residual spectrum of a closed linear relation we refer the reader to the appendix.

\begin{Thm}\label{pro:krein}
  Let $\{\mathcal G,\Gamma_0,\Gamma_1\}$ be an ordinary boundary triple for $A^*$ with $\gamma$-field $\gamma$ and Weyl function $M$ and 
  $A_0 = A^* \upharpoonright \ker \Gamma_0$, 
  let $\Theta$ be a closed linear relation in $\mathcal G$ and let
  $A_\Theta$ be the corresponding closed extension in Theorem~\ref{thm:boundary_triple}.
  Then for all $\lambda\in \rho(A_0)$ the following assertions {\rm(i)}-{\rm(iv)} hold.
  \begin{enumerate}[\rm(i)]
  \item $\lambda \in \sigma_p(A_\Theta)$ if and only if $0\in \sigma_p(\Theta-M(\lambda))$, in this case
    \[
    \ker (A_\Theta -\lambda) = \gamma(\lambda) \ker(\Theta - M(\lambda)),
    \]
  \item $\lambda \in \sigma_c(A_\Theta)$ if and only if $0\in \sigma_c(\Theta-M(\lambda))$,
  \item $\lambda \in \sigma_r(A_\Theta)$ if and only if $0\in \sigma_r(\Theta-M(\lambda))$,
  \item $\lambda \in \rho(A_\Theta)$ if and only if $0\in \rho(\Theta-M(\lambda))$ and the formula
    \begin{equation*}
      (A_\Theta - \lambda)^{-1} = (A_0 - \lambda)^{-1} + \gamma(\lambda) \bigl( \Theta - M(\lambda) \bigr)^{-1} \gamma(\bar \lambda)^*
    \end{equation*}
    holds for all $\lambda \in \rho(A_0) \cap \rho(A_\Theta)$.
  \end{enumerate}
\end{Thm}

The next corollary contains a variant of Kre\u{\i}n's formula for quasi boundary triples; cf.~\cite[Theorem 2.8]{BeLa07},
\cite[Theorem 3.6]{BeLaLo11}, and \cite[Theorem 6.16]{BeLa12} for other versions of Kre\u{\i}n's formula for the resolvent difference
of canonical extensions in the quasi boundary triple framework.

\begin{Cor}\label{cor:krein}
  Let $\{\mathcal G,\Gamma_0,\Gamma_1\}$ be a quasi boundary triple for $T\subset A^*$ 
  with $\gamma$-field $\gamma$, Weyl function $M$, $A_j = T \upharpoonright {\ker\Gamma_j}$, $j=0,1$, 
  such that $A_1$ is self-adjoint in $\mathcal H$, there exists $\eta \in \rho(A_0)\cap\mathbb R$ and $\mathscr G_0,\, \mathscr G_1$ 
  are dense in $\mathcal G$.  
  Moreover let $\vartheta\subset \mathscr G_1' \times \mathscr G_0'$ be a linear relation in $\ran \widetilde \Gamma$ such that the extension 
  \begin{equation*}
  \widetilde A_{\vartheta} = A^* \upharpoonright \bigl\{ f\in \dom A^* : \widetilde \Gamma f \in \vartheta \bigr\}
  \end{equation*}
  is closed in $\mathcal H$. 
  Then for all $\lambda\in \rho(A_0)$ the following assertions {\rm(i)}-{\rm(iv)} hold.
  \begin{enumerate}[\rm(i)]
  \item $\lambda \in \sigma_p(\widetilde A_\vartheta)$ 
    if and only if $0\in \sigma_p(\iota_{+} (\vartheta - \widetilde M(\lambda) ) \iota_{-}^{-1})$, in this case
    \[
    \ker (\widetilde A_\vartheta -\lambda) = \widetilde \gamma(\lambda) \ker(\vartheta - \widetilde M(\lambda)),
    \]
  \item $\lambda \in \sigma_c(\widetilde A_\vartheta)$ if and only if $0\in \sigma_c(\iota_{+} ( \vartheta - \widetilde M(\lambda) ) \iota_{-}^{-1})$,
  \item $\lambda \in \sigma_r(\widetilde A_\vartheta)$ if and only if $0\in \sigma_r(\iota_{+} ( \vartheta - \widetilde M(\lambda) ) \iota_{-}^{-1})$,
  \item $\lambda \in \rho(\widetilde A_\vartheta)$ if and only if $0\in \rho( \iota_{+} ( \vartheta - \widetilde M(\lambda) ) \iota_{-}^{-1})$ and 
  \begin{equation*}
    ( \widetilde A_\vartheta -\lambda )^{-1} 
    = ( A_{0}-\lambda )^{-1} 
    + \widetilde \gamma(\lambda)  \bigl(\vartheta-\widetilde M(\lambda)\bigr)^{-1} \widetilde \gamma(\bar\lambda)'
  \end{equation*}
  holds for all $\lambda\in\rho(\widetilde A_\vartheta)\cap\rho(A_0)$.
  \end{enumerate}
\end{Cor}

\begin{proof}
  Let $\{\mathcal G,\Upsilon_0,\Upsilon_1\}$ be the ordinary boundary triple for $A^*$ 
  in Theorem~\ref{thm:regularisation} with $A_0=A^*\upharpoonright \ker \Upsilon_0$, $\gamma$-field $\beta$ and 
  corresponding Weyl function $\mathcal M$ in \eqref{weihnachten}.
  By assumption we have $\vartheta\subset\ran\widetilde\Gamma$.
  According to Corollary~\ref{cor:main3} the linear relation $\Theta=\iota_+ (\vartheta - \widetilde M(\eta) ) \iota_-^{-1}$ is closed in $\mathcal G$ 
  and it follows that $\widetilde A_\vartheta$ and
  \begin{equation*}
  A_\Theta = A^* \upharpoonright \bigl\{ f\in \dom A^* : \Upsilon f \in \Theta \bigr\} 
  \end{equation*}
  coincide. Since $\mathcal M(\lambda) = \iota_+ ( \widetilde M(\lambda) - \widetilde M(\eta) ) \iota_-^{-1}$ by \eqref{weihnachten} we obtain the
  identity
  $\Theta - \mathcal M(\lambda) =  \iota_{+} ( \vartheta - \widetilde M(\lambda) ) \iota_{-}^{-1}$ and from 
  $\beta (\lambda) = \widetilde \gamma(\lambda)\iota_-^{-1}$ and $ \beta ( \bar\lambda )^* = \iota_+ \widetilde \gamma(\bar \lambda)'$ we then conclude
  \begin{equation}\label{stoerterm}
  \beta (\lambda) \bigl( \Theta - \mathcal M (\lambda)\bigr)^{-1} \beta ( \bar\lambda )^* 
  = \widetilde \gamma(\lambda)  \bigl( \vartheta- \widetilde M(\lambda) \bigr)^{-1}\widetilde \gamma(\bar \lambda)'.
  \end{equation}
  Now the assertions follow from Theorem~\ref{pro:krein}, $\widetilde A_\vartheta=A_\Theta$ and \eqref{stoerterm}. 
  Note that $( \vartheta- \widetilde M(\lambda) )^{-1} \subset \mathscr G_1 \times \mathscr G_1'$ in \eqref{stoerterm} since 
  $\vartheta- \widetilde M(\lambda) \subset \mathscr G_1' \times \mathscr G_1$ by
  Lemma~\ref{lem:exweyl}~(v). 
\end{proof}
%
%
\section{Applications to elliptic boundary value problems}\label{sec:applications}
In this section the abstract theory from Section~\ref{sec:qbt} and Section~\ref{sec:extension} 
is applied to elliptic differential operators. 
In Section~\ref{sec:laplace} we first study the Laplacian on 
bounded Lipschitz-, quasi-convex and $C^{1,r}$-domains with $r\in (\frac 12,1]$.
Then we investigate $2\,m$th order elliptic differential operators on bounded smooth domains in Section~\ref{sec:elliptic} and
second order elliptic differential operators on domains with compact boundary in 
Section~\ref{sec:ellipticsecond}. 

Throughout this section let $\Omega \subset \mathbb R^n$, $n\geq 2$, be a domain with boundary $\partial\Omega$ 
(which is at least Lipschitz). In Section~\ref{sec:laplace} and Section~\ref{sec:elliptic} the domain $\Omega$ is assumed to be
bounded, in Section~\ref{sec:ellipticsecond} the domain $\Omega$ may be unbounded as well but its boundary $\partial\Omega$ is assumed to be compact.
We denote by $H^s(\Omega)$ the Sobolev spaces of order $s\in \mathbb R$ on $\Omega$ and by
$H^s(\partial\Omega)$ the Sobolev spaces on $\partial\Omega$ of order $s$ (with at least $s\in [-1,\,1]$ in the Lipschitz case).  
By $H^s_0 (\Omega)$ we denote the closure of $C_0^\infty(\Omega)$ in $H^s(\Omega)$, $s\geq 0$, and
with $C^\infty(\overline\Omega)$ the functions in $C_0^\infty(\mathbb R^n)$ restricted to $\Omega$;
see, e.g.\ \cite[Chapter~3]{ML00}.
\subsection{A description of all self-adjoint extensions of the Laplacian on bounded Lipschitz domains}\label{sec:laplace}
In this subsection we give a complete description of the self-adjoint extensions of the Laplacian $-\Delta = -\sum_{j=1}^n \partial_j^2$ on 
a bounded Lipschitz domain $\Omega$ in terms of linear operators and relations $\Theta$ in 
$L^2(\partial\Omega)$ with the help of Theorem~\ref{thm:main3}. 
This description extends the one by Gesztesy and Mitrea in \cite{GeMi11}, 
where the class of so-called quasi-convex domains was treated; cf.~\cite[Definition~8.9]{GeMi11}. 
In addition we introduce Hilbert spaces $\mathscr G_0$ and $\mathscr G_1$ 
such that the Dirichlet- and Neumann trace operator admit continuous and 
surjective extensions from the maximal domain of the Laplacian onto the anti-dual spaces $\mathscr G_1'$ and $\mathscr G_0'$ respectively.

Let $\Omega\subset \mathbb R^n$, $n\geq 2$, be a bounded Lipschitz domain.
For $s\geq 0$ we define the Hilbert spaces
\begin{equation*}
  H^s_\Delta(\Omega) := \bigl\{f\in H^s(\Omega) : \Delta f \in L^2(\Omega)\bigr\}
\end{equation*}
equipped with the norms induced by  
\begin{equation*}
  (f,\, g)_{ H^s_\Delta(\Omega)}:=(f,\, g)_{H^s(\Omega)}+(\Delta f,\,\Delta g)_{L^2(\Omega)},\quad f,\, g\in H^s_\Delta(\Omega).
\end{equation*}
Note that for $s\geq 2$ the spaces $H^s_\Delta(\Omega)$ coincide with $H^s(\Omega)$. 
Define the minimal and maximal realization of the Laplacian in $L^2(\Omega)$  by 
\[
\Delta_{\min} := - \Delta \upharpoonright H^2_0(\Omega)\quad\text{and}\quad  \Delta_{\max} := -\Delta \upharpoonright H^0_\Delta(\Omega),
\]
respectively, and let $A:=\Delta_{\min}$. It follows from the Poincar\'e inequality that the norm induced by 
$H^0_\Delta(\Omega)$ is equivalent to the $H^2$-norm on $H^2_0(\Omega)$.
Hence a usual distribution type argument yields
\begin{equation*}
A = \Delta_{\min} =\Delta_{\max}^* \quad \text{and} \quad  A^* = \Delta_{\min}^* = \Delta_{\max};
\end{equation*}
cf. \cite[VI.~\textsection~29]{Tr80}. We mention that $A$ is a closed, densely defined, symmetric operator in $L^2(\Omega)$ with equal 
infinite deficiency indices.
Let $\mathfrak n=(\mathfrak n_1,\,\mathfrak n_2,\,\dots,\,\mathfrak n_n)^\top$ 
be the unit vector field pointing out of $\Omega$, which exists almost everywhere, see, e.g. \cite{ML00,Wl87}. 
The Dirichlet and Neumann trace operator $\tau_D$ and $\tau_N$ defined by
\begin{equation*}
  \tau_D f := f \upharpoonright_{\partial\Omega}, \quad \tau_N f :=  
  \mathfrak n \cdot \nabla f \upharpoonright_{\partial\Omega}, \quad f\in  C^\infty(\overline\Omega),
\end{equation*}
admit continuous extensions to operators 
\begin{equation}\label{eqn:trace2}
  \tau_D : H^s_\Delta(\Omega) \to H^{s-1/2}(\partial\Omega)\quad\text{and}\quad \tau_N : H^s_\Delta(\Omega)\to H^{s-3/2}(\partial\Omega)
\end{equation} 
for all $s\in [\frac 12,\, \frac 32]$. In particular, according to \cite[Lemma~3.1 and Lemma~3.2]{GeMi11} the extensions 
$\tau_D$ and $\tau_N$ in \eqref{eqn:trace2} 
are both surjective if $s=\frac 12$ and $s=\frac 32$.

In the next theorem we define a quasi boundary triple for the Laplacian  
\begin{equation}\label{eqn:DeltaT}
  T := -\Delta \upharpoonright H^{3/2}_\Delta(\Omega) = A^* \upharpoonright H^{3/2}_\Delta(\Omega) \subset \Delta_{\max}
\end{equation}
on the bounded Lipschitz domain $\Omega$ with $\Gamma_0$ and $\Gamma_1$ as the natural trace maps. 
In this setting it turns out that the spaces $\mathscr G_0$ and $\mathscr G_1$ from Definition~\ref{def:g0g1} are dense in $L^2(\partial\Omega)$,
the $\gamma$-field coincides with a family of Poisson operators and the values of the Weyl function are  
Dirichlet-to-Neumann maps (up to a minus sign). 
\begin{Thm}\label{thm:LipBLtriple}
  Let $\Omega$ be a bounded Lipschitz domain, let $T$ be as in \eqref{eqn:DeltaT} and let
  \begin{equation*}
  \Gamma_0, \Gamma_1 : H^{3/2}_\Delta(\Omega) \to L^2(\partial\Omega), \qquad \Gamma_0 f := \tau_D f, \quad \Gamma_1 f := -\tau_N f.
  \end{equation*}  
  Then $\{L^2(\partial\Omega),\Gamma_0,\Gamma_1\}$ is a quasi boundary triple for $T \subset A^* = \Delta_{\max}$ 
  such that the minimal realization $A=\Delta_{\min}$ coincides with $T\upharpoonright \ker\Gamma$ and the following statements hold.
  \begin{enumerate}[\rm(i)]
  \item The Dirichlet realization $\Delta_{D}$ and Neumann realization $\Delta_{N}$ correspond to $\ker \Gamma_0$ and $\ker \Gamma_1$,
\begin{equation}\label{eqn:LipBLtriple1}
\begin{split}
  \Delta_{D}&:= T\upharpoonright \ker \Gamma_0=\Delta_{\max} \upharpoonright \bigl\{ f\in H^{3/2}_\Delta(\Omega) : \tau_D f = 0\bigr\}, \\
  \Delta_{N}&:= T\upharpoonright \ker \Gamma_1=\Delta_{\max} \upharpoonright \bigl\{ f\in H^{3/2}_\Delta(\Omega) : \tau_N f = 0\bigr\},
  \end{split}
\end{equation}
    respectively, and both operators
    are self-adjoint in $L^2(\Omega)$.
  \item The spaces 
    \begin{equation*} 
      \mathscr G_0 = \ran (\Gamma_0 \upharpoonright \ker \Gamma_1)\quad\text{and}\quad  
      \mathscr G_1 = \ran (\Gamma_1 \upharpoonright \ker \Gamma_0)
    \end{equation*}
    are dense in $L^2(\partial\Omega)$.
  \item The values $\gamma(\lambda): L^2(\partial\Omega)\supset H^1(\partial\Omega)\to L^2(\Omega)$ of the $\gamma$-field are given by 
    \[
    \gamma(\lambda) \varphi = f, \qquad \varphi\in H^1(\partial\Omega), \quad \lambda\in\rho(\Delta_{D}), 
    \]
    where $f\in L^2(\Omega)$ is the unique solution of the boundary value problem 
    \begin{equation}\label{eqn:laplacePDE}
        (-\Delta -\lambda) f = 0, \qquad
        \tau_D f = \varphi.
    \end{equation}
  \item The values $M (\lambda): L^2(\partial\Omega)\supset H^1(\partial\Omega)\to L^2(\partial\Omega)$ of the Weyl function  
    are Dirichlet-to-Neumann maps given by 
    \begin{equation*} 
      M(\lambda) \varphi = -\tau_N f,  \qquad \varphi\in H^1(\partial\Omega), \quad \lambda\in\rho(\Delta_{D}),
    \end{equation*}
    where $f=\gamma(\lambda)\varphi$ is the unique solution of \eqref{eqn:laplacePDE}.
    The operators $M(\lambda)$ are bounded from $H^1(\partial\Omega)$ to $L^2(\partial\Omega)$
    and if, in addition, $\lambda\in\rho(\Delta_{N})$ then the Neumann-to-Dirichlet map 
    $M(\lambda)^{-1}$ is a compact operator in $L^2(\partial\Omega)$. 
\end{enumerate}
\end{Thm}

\begin{proof}
  We check that $\{L^2(\partial\Omega),\Gamma_0,\Gamma_1\}$ is a quasi boundary triple for $T\subset A^*$.
  From \cite[Theorem~2.6 and~2.10, Lemma~3.4 and~4.8]{GeMi08} 
  we obtain that the Dirichlet and Neumann Laplacian in \eqref{eqn:LipBLtriple1} are both self-adjoint in $L^2(\Omega)$; for the
  $H^{3/2}$-regularity of the Dirichlet domain see also \cite{JeKe81, JeKe95}. In particular, item (iii) in Definition~\ref{qbt} is valid
  and assertion (i) of the theorem holds.
 
  The fact that $\ran\Gamma$ is dense in $L^2(\partial\Omega)\times L^2(\partial\Omega)$ will follow below when we verify 
  assertion (ii) of the theorem. 
  For the moment we note that item (ii) in Definition~\ref{qbt} holds.

  The continuity of the trace maps $\tau_D, \, \tau_N : H^{3/2}_\Delta(\Omega) \to L^2(\partial\Omega)$ and the fact that
  $C^\infty(\overline\Omega)$ is dense in $H^{3/2}_\Delta(\Omega)$ (see \cite[Lemme~3]{CoDa98}) yield
  Green's identity 
  \begin{eqnarray*}
  (T f,g)_{L^2(\Omega)} - (f, T g)_{L^2(\Omega)} &=& (-\Delta f,g)_{L^2(\Omega)} - (f, -\Delta g)_{L^2(\Omega)} \\
    &=& (-\tau_N f, \tau_D g)_{L^2(\partial\Omega)} - (\tau_D f,-\tau_N g)_{L^2(\partial\Omega)}\\
    &=& (\Gamma_1 f,\Gamma_0 g)_{L^2(\partial\Omega)} - (\Gamma_0 f,\Gamma_1 g)_{L^2(\partial\Omega)} 
  \end{eqnarray*}
  for all $f,\, g\in H^{3/2}_\Delta(\Omega)$, that is, condition (i) in Definition~\ref{qbt} holds.

  Furthermore, as $C^\infty(\overline\Omega)$ is dense in $H^0_\Delta(\Omega)=\dom A^*$ it follows that
  $\overline T=A^*=\Delta_{\max}$ holds. 
  Therefore $\{L^2(\partial\Omega),\Gamma_0,\Gamma_1\}$ is a quasi boundary triple for $T$. 
  Hence we also obtain $T\upharpoonright \ker \Gamma = A = \Delta_{\min}$ from the fact that $\ker \Gamma = \dom A$ holds in every quasi boundary triple.
   
  Next we verify assertion (ii) (which also implies property (ii) in the definition of a quasi boundary triple). 
  Recall that $\ran \Gamma_1 = L^2(\partial\Omega)$ by \eqref{eqn:trace2} and suppose that
  $h\,\bot\, \mathscr G_0$. Choose $f\in\dom\Gamma_1$ such that $h=\Gamma_1 f$. Then for all 
  $g\in\ker\Gamma_1=\dom \Delta_{N}$ Green's identity yields
  \begin{eqnarray*}
  0&=& (h,\Gamma_0 g)_{L^2(\partial\Omega)}=(\Gamma_1 f,\Gamma_0 g)_{L^2(\partial\Omega)} - (\Gamma_0 f,\Gamma_1 g)_{L^2(\partial\Omega)}\\
   &=& (T f,g)_{L^2(\Omega)} - (f, \Delta_{N} g)_{L^2(\Omega)}
  \end{eqnarray*}
  and since $\Delta_{N}$ is selfadjoint by (i) we obtain $f\in\dom \Delta_{N}=\ker\Gamma_1$ and hence $h=\Gamma_1 f=0$, that is, $\mathscr G_0$
  is dense in $L^2(\partial\Omega)$.
  The fact that $\mathscr G_1$ is dense in $L^2(\partial\Omega)$ follows from 
  \cite[Lemma~6.3 and Corollary~6.5]{GeMi11}
  since the subspace $\ran(\tau_N \upharpoonright \{ f\in H^2(\Omega) : \tau_D f = 0 \})$ 
  of $\mathscr G_1$ is dense in $L^2(\partial\Omega)$. This shows assertion (ii). Since 
  $\mathscr G_0\times\mathscr G_1\subset \ran \Gamma$ also $\ran\Gamma$ is dense in 
  $L^2(\partial\Omega)\times L^2(\partial\Omega)$ as noted above.

  Most of the assertions in (iii) and (iv) are immediate consequences of the definition of the $\gamma$-field and the Weyl function corresponding to the 
  quasi boundary triple $\{L^2(\partial\Omega),\Gamma_0,\Gamma_1\}$. For the boundedness of $M(\lambda)$ regarded as an operator
from $H^1(\partial\Omega)$ into $L^2(\partial\Omega)$ and the compactness of $M(\lambda)^{-1}$ as an operator in $L^2(\partial\Omega)$ we
refer to \cite[Theorem~3.7 and Remark~3.8]{GeMi08}. 
\end{proof}
Let $\{L^2(\partial\Omega),\Gamma_0,\Gamma_1\}$ be the quasi boundary triple for $T \subset A^* =\Delta_{\max}$ 
from Theorem~\ref{thm:LipBLtriple} with Weyl function $M$. 
Equip the spaces 
$\mathscr G_0$ and 
$\mathscr G_1$ with the norms induced by
\begin{equation}\label{eqn:norms}
  \begin{split}
    (\varphi,\psi)_{\mathscr G_0} &:= (\Sigma^{-1/2} \varphi, \Sigma^{-1/2} \psi)_{L^2(\partial\Omega)}, \qquad
    \Sigma = \im(- M(i)^{-1}), \\
    (\varphi,\psi)_{\mathscr G_1} &:= (\Lambda^{-1/2} \varphi, \Lambda^{-1/2} \psi)_{L^2(\partial\Omega)},\qquad \Lambda= \overline{\im M(i)};  
  \end{split}
\end{equation} 
cf.~Section~\ref{2.3}.
As an immediate consequence of Proposition~\ref{pro:Gamma_0_extension} and Corollary~\ref{cor:Gamma_1_extension}, 
see also Definition~\ref{def:extweyl}, Lemma~\ref{lem:exgamma} and Lemma~\ref{lem:exweyl}, we obtain 
a trace theorem for the Dirichlet and Neumann trace operator on the maximal domain of the Laplacian.
\begin{Cor}\label{cor:exttau}
  Let $\Omega$ be a bounded Lipschitz domain. Then the following statements hold.
  \begin{enumerate}[\rm(i)]
  \item The Dirichlet trace operator $\tau_D$ and Neumann trace operator $\tau_N$ can be extended by continuity to surjective 
    mappings 
    \begin{equation*}
      \widetilde \tau_D : H^0_\Delta(\Omega) \to \mathscr G_1'\quad\text{and}\quad
      \widetilde \tau_N : H^0_\Delta(\Omega) \to \mathscr G_0'
    \end{equation*}
    such that $\ker \widetilde \tau_D = \ker \tau_D = \dom \Delta_{D}$ and 
      $\ker \widetilde \tau_N = \ker \tau_N = \dom \Delta_{N}$.
  \item For all $\lambda\in\rho(\Delta_{D})$ the values of the $\gamma$-field $\gamma$ from Theorem~{\rm\ref{thm:LipBLtriple}} admit continuous extensions 
    \begin{equation*}
      \widetilde \gamma(\lambda) : \mathscr G_1' \to L^2(\partial\Omega), \quad \varphi \mapsto \widetilde \gamma(\lambda) \varphi = f
    \end{equation*}
    where $f\in L^2(\Omega)$ is the unique solution of \eqref{eqn:laplacePDE}. 
    In particular, the space $\mathscr G_1'$ is maximal in the sense that whenever $f\in L^2(\Omega)$ is a solution 
    of the Dirichlet problem \eqref{eqn:laplacePDE} then the boundary value $\varphi$ belongs to $\mathscr G_1'$.
  \item For all $\lambda\in\rho(\Delta_{D})$ the values $M(\lambda)$ of the Weyl function $M$ from Theorem~{\rm\ref{thm:LipBLtriple}} admit continuous extensions 
    \begin{equation*}
      \widetilde M(\lambda) : \mathscr G_1' \to \mathscr G_0', \quad \varphi \mapsto \widetilde M(\lambda) \varphi = -\widetilde\tau_N f,
      \quad \lambda\in\rho(\Delta_{D}),
    \end{equation*}
    where $f=\widetilde\gamma(\lambda) \varphi$ is the unique solution of \eqref{eqn:laplacePDE}.
  \end{enumerate}
  \end{Cor}
Applying Theorem~\ref{thm:regularisation} to the quasi boundary triple 
$\{L^2(\partial\Omega),\Gamma_0,\Gamma_1\}$
from Theorem~{\rm\ref{thm:LipBLtriple}} we get a Lipschitz domain version of the 
ordinary boundary triple for the Laplacian as it appears already in the smooth case in \cite{Gr68}, 
see also, e.g.\ \cite{Be10,BeLa12,BrGrWo09,Ma10}. 
Recall that there exist isometric isomorphisms $\iota_+ : \mathscr G_1 \to L^2(\partial\Omega)$, 
$\iota_- : \mathscr G_1' \to L^2(\partial\Omega)$ such that $( \iota_- x',\, \iota_+ x)_{L^2(\partial\Omega)} =  
\langle x',\, x  \rangle_{\mathscr G_1' \times \mathscr G_1}$; cf. \eqref{iotas}. 
\begin{Cor}\label{cor:Grubbtriplelipschitz}
   Let $\eta\in \rho(\Delta_D)\cap \mathbb R$ and let
   $\Upsilon_0, \Upsilon_1 : H^0_{\Delta}(\Omega) \to L^2(\partial\Omega)$ be given by
  \begin{equation*} 
    \Upsilon_0 f := \iota_- \widetilde \tau_D f, \quad \Upsilon_1 f := -\iota_+ \tau_N f_D,
    \quad f=f_D+f_\eta\in\dom \Delta_D\dotplus\mathcal N_\eta(A^*).
  \end{equation*}
  Then $\{L^2(\partial\Omega),\Upsilon_0,\Upsilon_1\}$ is an ordinary boundary triple for $A^*=\Delta_{\max}$ with
  $A^*\upharpoonright\ker\Upsilon_0=\Delta_D$ and
  \[
  A^*\upharpoonright\ker\Upsilon_1=\Delta_{\min}\dot + \bigl\{(f_\eta,\eta f_\eta)^\top: 
  -\Delta f_\eta=\eta f_\eta,\,f_\eta\in H^0_\Delta(\Omega)\bigr\}.
  \]
\end{Cor}

In the present setting Theorem~\ref{thm:main3} can be applied 
to the quasi boundary triple from Theorem~\ref{thm:LipBLtriple}. This
yields a description of all self-adjoint extensions $\Delta_\vartheta \subset \Delta_{\max}$ of 
the minimal Laplacian $\Delta_{\min}$ in $L^2(\Omega)$ on bounded Lipschitz domains.

\begin{Cor}\label{cor:mainlaplace}
  Let $\Omega$ be a bounded Lipschitz domain, 
  $\mathscr G_0$, $\mathscr G_1$ be as in Theorem~{\rm\ref{thm:LipBLtriple}}, 
  $\eta\in\mathbb R\cap\rho(\Delta_D)\cap \rho(\Delta_N)$ and 
  $\widetilde M(\eta) : \mathscr G_1' \to \mathscr G_0'$ be the extended Dirichlet-to-Neumann map.
  Then the mapping
  \begin{equation*}
    \Theta\mapsto \Delta_\vartheta = \Delta_{\max}\upharpoonright\bigl\{ f\in H^0_\Delta(\Omega) : 
    \vartheta\widetilde \tau_D f +\widetilde\tau_N f= 0 \bigr\},
    \quad \vartheta = \iota_+^{-1} \Theta \iota_- + \widetilde M(\eta),
  \end{equation*}
  establishes a bijective correspondence between all
  closed (symmetric,  self-adjoint, (maximal) dissipative, (maximal) accumulative) 
  linear relations $\Theta$ in $L^2(\partial\Omega)$ and 
  all closed (symmetric,  self-adjoint, (maximal) dissipative, (maximal) accumulative, respectively) extensions 
  $\Delta_\vartheta \subset A^* =\Delta_{\max}$ of 
  $A = \Delta_{\min}$ in $L^2(\Omega)$.
  Moreover, the following regularity result holds: 
  If $\Delta_s$ is an extension of $T$ in \eqref{eqn:DeltaT} such that $\Delta_s \subset A^* =\Delta_{\max}$ 
  then
  \begin{equation}\label{eqn:mainlaplace2}
  \dom \Theta \subset \ran \bigl( \iota_- \widetilde\tau_D \upharpoonright \dom \Delta_s \bigr) \quad \text{implies} \quad 
  \dom \Delta_\vartheta \subset \dom \Delta_s.
  \end{equation}
\end{Cor}

We note that the abstract propositions from Section~\ref{sec:krein} can be applied to 
the quasi boundary triple $\{L^2(\partial\Omega),\Gamma_0,\Gamma_1\}$, see also Section~\ref{sec:ellipticsecond}. We leave the formulations
to the reader and state only a version of Kre\u{\i}n's formula as in Corollary~\ref{cor:krein}.

\begin{Cor}\label{cor:laplacekrein}
  Let $\Omega$ be a bounded Lipschitz domain, $\widetilde\gamma(\lambda): \mathscr G_1' \to L^2(\Omega)$ 
  and $\widetilde M(\lambda): \mathscr G_1' \to \mathscr G_0'$ be 
  the extended $\gamma$-field and Dirichlet-to-Neumann map from Corollary~{\rm\ref{cor:exttau}}.
  Let $\vartheta\subset \mathscr G_1' \times \mathscr G_0'$ be a linear relation in $\ran (\widetilde \tau_D, -\widetilde \tau_N)$ 
  such that 
  \begin{equation*}
    \Delta_{\vartheta} = \Delta_{\max} \upharpoonright \bigl\{ f\in H^0_{\Delta}(\Omega) : \vartheta \widetilde \tau_D f + \widetilde \tau_N f = 0 \bigr\}
  \end{equation*}
  is closed in $L^2(\Omega)$.
  Then for all $\lambda\in \rho(\Delta_{D})$ the following assertions {\rm(i)}-{\rm(iv)} hold.
  \begin{enumerate}[\rm(i)]
  \item $\lambda \in \sigma_p(\Delta_{\vartheta})$ 
    if and only if $0\in \sigma_p(\iota_{+} ( \vartheta -\widetilde M(\lambda) ) \iota_{-}^{-1})$, in this case
    \begin{equation*}
    \ker (\Delta_{\vartheta}-\lambda) = \widetilde\gamma(\lambda) \ker(\vartheta -\widetilde M(\lambda));
    \end{equation*}
  \item $\lambda \in \sigma_c(\Delta_{\vartheta})$ if and only if $0\in \sigma_c(\iota_{+} ( \vartheta-\widetilde M(\lambda)  ) \iota_{-}^{-1})$;
  \item $\lambda \in \sigma_r(\Delta_{\vartheta})$ if and only if $0\in \sigma_r(\iota_{+} ( \vartheta-\widetilde M(\lambda)  ) \iota_{-}^{-1})$;
  \item $\lambda \in \rho(\Delta_{\vartheta})$ if and only if $0\in \rho( \iota_{+} ( \vartheta-\widetilde M(\lambda) ) \iota_{-}^{-1})$ and 
  \begin{equation*}
    ( \Delta_{\vartheta} -\lambda )^{-1} 
    = (\Delta_{D} -\lambda )^{-1} 
    + \widetilde \gamma(\lambda)  \bigl(\vartheta-\widetilde M(\lambda)\bigr)^{-1} \widetilde \gamma(\bar\lambda)'
  \end{equation*}
  holds for all $\lambda\in\rho(\Delta_{\vartheta})\cap\rho(\Delta_{D})$.
  \end{enumerate}
\end{Cor}
 
In the following we slightly improve Lemma~\ref{lem:symmetric} by using the fact 
that $\ker \tau_N = \ker \widetilde \tau_N = \dom \Delta_N$.

\begin{Lem}\label{lem:symmetric4}
  Let $\Omega$ be a bounded Lipschitz domain and let $\vartheta$ be a linear relation in $L^2(\partial\Omega)$. 
  Then 
  \begin{equation*}
    \Delta_{\vartheta} := \Delta_{\max} \upharpoonright  
    \bigl\{ f\in H^0_{\Delta}(\Omega) : \vartheta \widetilde \tau_D f + \widetilde \tau_N f = 0 \bigr\}
  \end{equation*}
  has regularity $\dom \Delta_\vartheta \subset H^{3/2}_{\Delta}(\Omega)$. Moreover, $\Delta_\vartheta$ is symmetric in $L^2(\Omega)$
  if and only if $\vartheta$ is symmetric $L^2(\partial\Omega)$.
\end{Lem}

\begin{proof}
  For $f\in\dom \Delta_\vartheta$ we have $\vartheta \widetilde \tau_D f=-\widetilde \tau_N f\in L^2(\partial\Omega)$ as $\vartheta$ is
  assumed to be a linear relation in $L^2(\partial\Omega)$. By \eqref{eqn:trace2}
  there exists $g\in H^{3/2}_\Delta(\Omega)$ such that $\tau_N g =\widetilde \tau_N f $ and
  hence 
  \[
  f-g \in  \ker \widetilde \tau_N = \ker \tau_N = \dom \Delta_N \subset H^{3/2}_\Delta(\Omega).
  \] Therefore 
  $f=(f-g)+g\in H^{3/2}_\Delta(\Omega)$ and $\dom \Delta_\vartheta \subset H^{3/2}_{\Delta}(\Omega)$ holds. In particular, we have
  \begin{equation}\label{eqn:symmetric9}
    \Delta_{\vartheta} = \Delta_{\max} \upharpoonright  
    \bigl\{ f\in H^{3/2}_{\Delta}(\Omega) : \vartheta\Gamma_0 f - \Gamma_1 f=0  \bigr\},
  \end{equation}
  where $\{L^2(\partial\Omega),\Gamma_0,\Gamma_1\}$ is the quasi boundary triple from Theorem~\ref{thm:LipBLtriple}.
  Then by Lemma~\ref{lem:symmetric} $\Delta_\vartheta$ is symmetric in $L^2(\Omega)$
  if and only if $\vartheta$ is symmetric $L^2(\partial\Omega)$.
\end{proof}

The next theorem is a slightly improved Lipschitz domain version of \cite[Theorem~4.8]{BeLa07}, see also \cite[Theorem~6.21]{BeLa12}.

\begin{Thm}\label{eqn:boundedtheta}
  Let $\Omega$ be a bounded Lipschitz domain and let $\vartheta$ be a bounded self-adjoint operator in $L^2(\partial\Omega)$. 
  Then 
  \begin{equation}\label{eqn:extselfad}
    \Delta_\vartheta := \Delta_{\max} \upharpoonright \bigl\{ f\in H^0_\Delta(\Omega) : \vartheta \widetilde \tau_D f + \widetilde \tau_N f =0  \bigr\}
  \end{equation}
  is a self-adjoint operator in $L^2(\Omega)$ with compact resolvent, semibounded from below and  
  regularity $\dom \Delta_\vartheta \subset H^{3/2}_{\Delta}(\Omega)$.
\end{Thm}

\begin{proof}
  It follows from Lemma~\ref{lem:symmetric4} that $\dom \Delta_\vartheta \subset H^{3/2}_{\Delta}(\Omega)$ holds and hence $\Delta_\vartheta$ is given by
  \eqref{eqn:symmetric9}, where $\{L^2(\partial\Omega),\Gamma_0,\Gamma_1\}$ is the quasi boundary triple for $T \subset \Delta_{\max}$ 
  from Theorem~\ref{thm:LipBLtriple} with Weyl function $M$. According to Theorem~\ref{thm:LipBLtriple}~(iv) the Neumann-to-Dirichlet maps
  $M(\lambda)^{-1}$, $\lambda\in\rho(\Delta_D)\cap\rho(\Delta_N)$, are compact operators in $L^2(\partial\Omega)$, and hence
  \cite[Theorem~6.21]{BeLa12} implies that $\Delta_\vartheta$ is a self-adjoint operator in $L^2(\Omega)$. The compactness of the resolvent 
  of $\Delta_\vartheta$ follows from \cite[Theorem~4.8]{BeLa07} applied to the quasi boundary triple $\{L^2(\partial\Omega),\Gamma_1,-\Gamma_0\}$ 
  and the parameter $\Theta=-\vartheta^{-1}$.
 
  It remains to show that $\Delta_\vartheta$ is semibounded from below. If $\vartheta=0$ this is obviously true. 
  Suppose $\vartheta\not= 0$, let $0<\varepsilon\leq 1/\|\vartheta \|$ and choose 
  $c_\varepsilon >0$ such that
  \begin{equation*}
    \|\tau_D g\|^2_{L^2(\partial\Omega)} \leq \varepsilon \|\nabla g\|_{L^2(\Omega)^n}^2 + c_\varepsilon \| g\|_{L^2(\Omega)}^2, \quad g\in H^1(\Omega);
  \end{equation*}
  see, e.g. \cite[Lemma~4.2]{GeMi09}. For $f\in \dom \Delta_\vartheta$ Green's identity together with 
  $-\tau_N f = \vartheta \tau_D f$ (see \eqref{eqn:extselfad}) implies
  \begin{equation*}
    \begin{split}
      (\Delta_\vartheta f, \,f)_{L^2(\Omega)} &= \|\nabla f\|_{L^2(\Omega)^n}^2 + (\vartheta  \tau_D f,\, \tau_D f)_{L^2(\partial\Omega)}\\
      &\geq \|\nabla f\|_{L^2(\Omega)^n}^2  - \|\vartheta \| \, \|\tau_D f\|^2_{L^2(\partial\Omega)}\\
      &\geq \|\nabla f\|_{L^2(\Omega)^n}^2  - \varepsilon\|\vartheta \| \|\nabla f\|_{L^2(\Omega)^n}^2 - c_\varepsilon \|\vartheta \|\| f\|_{L^2(\Omega)}^2\\
      &\geq -c_\varepsilon \|\vartheta \| \, \| f\|^2_{L^2(\partial\Omega)}.
    \end{split}
  \end{equation*}
\end{proof}

In the next corollary we formulate a version of Theorem~\ref{eqn:boundedtheta} for Robin boundary conditions.

\begin{Cor}
  Let $\Omega$ be a bounded Lipschitz domain and let $\alpha\in L^\infty(\partial\Omega)$ be 
  a real function on $\partial \Omega$. 
  Then
  \begin{equation}\label{eqn:robin}
    \Delta_\alpha := \Delta_{\max} \upharpoonright \bigl\{ f\in H^0_\Delta(\Omega) : 
    \alpha\cdot \widetilde \tau_D f + \widetilde \tau_N f = 0 \bigr\}
  \end{equation}
  is self-adjoint operator in $L^2(\Omega)$ with compact resolvent, semibounded from below and  
  regularity $\dom \Delta_\alpha \subset H^{3/2}_{\Delta}(\Omega)$. 
  In \eqref{eqn:robin} the multiplication with $\alpha$ is understood as an operator in $L^2(\partial\Omega)$. 
\end{Cor}

In the end of this subsection we establish the link to \cite{GeMi11} and briefly discuss two more special cases of bounded Lipschitz domains: 
so-called quasi-convex domains in Theorem~\ref{thm:GeMi} and $C^{1,r}$-domains with $r\in(\frac 12,1]$ in Theorem~\ref{thm:C1r}.

For the definition of quasi-convex domains we refer to \cite[Definition~8.9]{GeMi11}.
We mention that all convex domains, all almost-convex domains,
all domains that satisfy a local exterior ball condition, as well as all $C^{1,r}$-domains with $r\in(\frac 12,1]$ are quasi-convex, 
for more details on almost-convex domains see~\cite{MiTaVa05}. 
The key feature of a quasi-convex domain is that the Dirichlet- and Neumann Laplacian have $H^2$-regularity, 
\begin{equation}\label{eqn:quasireg}
  \dom \Delta_{D}\subset H^2(\Omega),\quad \dom \Delta_{N} \subset H^2(\Omega).
\end{equation}
For the next theorem we recall the definition of the tangential gradient operator 
\begin{equation*}
\nabla_{\tan}: H^1(\partial\Omega) \to L^2(\partial\Omega)^n,\quad 
\nabla_{\tan} f := \Bigl( \sum_{j=1}^n \mathfrak n_j \partial_{\tau_{j,k}} f \Bigr)^\top_{k=1,\dots,n}
\end{equation*} 
from \cite[(6.1)]{GeMi11}. Here $\partial_{\tau_{j,k}} := \mathfrak n_j \partial_k - \mathfrak n_k \partial_j$,
$j,\, k\in \{1,\dots,n\}$, are the first-order tangential differential operators
acting continuously from $H^1(\partial\Omega)$ to $L^2(\partial\Omega)$.

\begin{Thm}\label{thm:GeMi}
  Let $\Omega$ be a quasi-convex domain. Then the following statements hold.
  \begin{enumerate}[\rm(i)]
  \item The spaces $\mathscr G_0$ and $\mathscr G_1$ in Theorem~{\rm\ref{thm:LipBLtriple}} are given by 
    \begin{equation*}
      \begin{split}
        \mathscr G_0 &= \bigl \{ \varphi \in H^1(\partial\Omega) : \nabla_{\tan} \varphi \in H^{1/2}(\partial\Omega)^n \bigr \}, \\ 
        \mathscr G_1 &= \bigl \{ \psi \in L^2(\partial\Omega) :  \psi \, \mathfrak n \in H^{1/2}(\partial\Omega)^n  \bigr\},
      \end{split}
    \end{equation*}
    and for the norms $\|\cdot\|_{\mathscr G_0}$ and $\|\cdot\|_{\mathscr G_1}$ induced by the inner products in \eqref{eqn:norms} 
    the following equivalences hold: 
    \begin{equation*}
      \begin{split}
      \| \varphi\|_{\mathscr G_0} &\sim \| \varphi \|_{L^2(\partial\Omega)} +  \| \nabla_{\tan} \varphi \|_{H^{1/2}(\partial\Omega)^n}, 
      \quad \varphi\in \mathscr G_0,\\ 
      \|\psi\|_{\mathscr G_1} &\sim \| \psi \, \mathfrak n \|_{H^{1/2}(\partial\Omega)^n}, \quad \psi\in \mathscr G_1.
      \end{split}
    \end{equation*}
  \item The Dirichlet trace operator $\tau_D$ and Neumann trace operator $\tau_N$ admit continuous, surjective extensions to 
    \begin{equation*}
      \begin{split}
        \widetilde \tau_D &: H^0_\Delta(\Omega) \to \bigl(\bigl\{ \psi \in L^2(\partial\Omega) :  
        \psi \, \mathfrak n \in H^{1/2}(\partial\Omega)^n \bigr\}\bigr)', \\ 
        \widetilde \tau_N &: H^0_\Delta(\Omega) \to\bigl( \bigl\{ \varphi \in H^1(\partial\Omega) : 
        \nabla_{\tan} \varphi \in H^{1/2}(\partial\Omega)^n \bigr\}\bigr)'.
      \end{split}
    \end{equation*}
  \end{enumerate}
\end{Thm}
\begin{proof}
  Let $\Omega$ be a bounded Lipschitz domain.
  According to \cite{MaMiSh10} the trace operator 
  $f\mapsto (\tau_D f, \tau_N f)^\top$, $f\in C^\infty(\overline\Omega)$,
  admits a continuous extension to a mapping from $H^2(\Omega)$ onto 
  the space of all $(\varphi,\psi)^\top \in H^1(\partial\Omega) \times L^2(\partial\Omega)$
  such that $\nabla_{\tan} \varphi + \psi\, \mathfrak n \in H^{1/2}(\partial\Omega)^n$; here $H^1(\partial\Omega) \times L^2(\partial\Omega)$ is
  equipped with the norm
  \begin{equation*}
    \| \varphi \|_{H^1(\partial\Omega)} +  \| \psi \|_{L^2(\partial\Omega)} + \| \nabla_{\tan} \varphi + \psi \, 
    \mathfrak n\|_{H^{1/2}(\partial\Omega)^n}.
  \end{equation*}
  The kernel of this extension of $(\tau_D,\, \tau_N)^\top$ is $H^2_0(\Omega)$. 
  This implies that 
  the Dirichlet trace operator $\tau_D$ admits a continuous extension to a surjective mapping from
  \begin{equation*}
    \{ f\in H^2(\Omega) : \tau_N f = 0 \} \quad \text{onto} \quad 
    \bigl  \{ \varphi \in H^1(\partial\Omega) : \nabla_{\tan} \varphi \in H^{1/2}(\partial\Omega)^n \bigr \}
  \end{equation*}
  and the Neumann trace operator $\tau_N$ admits a continuous extension to a surjective mapping from 
  \begin{equation*}
    \{ f\in H^2(\Omega) : \tau_D f = 0 \} \quad \text{onto} \quad 
    \bigl \{ \psi \in L^2(\partial\Omega) :  \psi \, \mathfrak n \in H^{1/2}(\partial\Omega)^n  \bigr\};
  \end{equation*}
  cf.~\cite[Lemma~6.3 and Lemma~6.9]{GeMi11}. 
  Now let $\Omega$ be a quasi-convex domain. 
  Then according to \cite[Lemma~8.11]{GeMi11} 
  the regularity properties \eqref{eqn:quasireg} hold, and since $\mathscr G_0$, $\mathscr G_1$ 
  are Hilbert spaces, which are dense in $L^2(\partial\Omega)$ the assertions follow from
  Proposition~\ref{pro:Gamma_0_extension} and Corollary~\ref{cor:Gamma_1_extension}.  
\end{proof}

  We note that
  Theorem~\ref{thm:GeMi} is essentially the same as
  \cite[Theorems~6.4 and 6.10]{GeMi11}, and also implies \cite[Corollaries~10.3 and 10.7]{GeMi11}.
  Theorem~\ref{thm:GeMi} together with Corollary~\ref{cor:mainlaplace} 
  yields results of similar form as in \cite[Sections 14 and 15]{GeMi11};
  the Kre\u{\i}n type resolvent formulas in \cite[Section 16]{GeMi11} can also be viewed as consequences of Corollary~\ref{cor:laplacekrein}.

In the next theorem we treat the case of  $C^{1,r}$-domains with $r \in (\frac 12,1]$. In a similar manner as above this theorem 
combined with the earlier abstract results leads to various results on self-adjoint realizations or
Krein type resolvent formulas in the flavour of \cite{GeMi11}.

\begin{Thm}\label{thm:C1r}
  Let $\Omega$ be a $C^{1,r}$-domain with $r \in (\frac 12,1]$. Then 
  the following statements hold.
 \begin{itemize}
  \item [{\rm (i)}] The spaces $\mathscr G_0$ and $\mathscr G_1$ in Theorem~{\rm\ref{thm:LipBLtriple}} are given by
  \begin{equation*}
    \mathscr G_0 = H^{3/2}(\partial\Omega)\quad\text{and}\quad \mathscr G_1 = H^{1/2}(\partial\Omega)
  \end{equation*}
  and the norms induced by the inner products in \eqref{eqn:norms} are equivalent to the usual norms in 
   $ H^{3/2}(\partial\Omega)$ and $ H^{1/2}(\partial\Omega)$, respectively.
  \item [{\rm (ii)}] The Dirichlet trace operator $\tau_D$ and Neumann trace operator $\tau_N$ admit continuous, 
  surjective extensions to 
    \begin{equation*}
      \tau_D : H^0_\Delta(\Omega) \to H^{-1/2}(\partial\Omega)\quad\text{and}\quad
      \tau_N : H^0_\Delta(\Omega) \to H^{-3/2}(\partial\Omega).
    \end{equation*}
 \end{itemize}
 Moreover, the following regularity result holds: For $0 \leq s \leq \frac 32$
  \begin{equation}\label{eqn:lapreg}
    \dom \Theta \subset H^s(\partial\Omega) \quad\text{implies}\quad \dom \Delta_\Theta \subset H^s_\Delta(\Omega).
  \end{equation}
\end{Thm}
 \begin{proof}
   Note that \eqref{eqn:quasireg} holds for the Dirichlet and Neumann Laplacian and
   that the trace operator $f\mapsto (\tau_D,\tau_N)^\top$, $f \in C^\infty(\overline\Omega)$, 
   admits a continuous extension to a mapping from $H^2(\Omega)$ onto 
   $H^{3/2}(\partial\Omega)\times H^{1/2}(\partial\Omega)$, see, e.g.~\cite[Theorem~2]{Ma87}. 
   Hence 
   statements (i) and (ii) follow from Proposition~\ref{pro:Gamma_0_extension} and Corollary~\ref{cor:Gamma_1_extension}. 
   It remains to verify the regularity result \eqref{eqn:lapreg}.
   Let $\Delta_s := \Delta_{\max} \upharpoonright H^s_\Delta(\Omega)$ with $0\leq s\leq \frac 32$, so that
   $T$ in \eqref{eqn:DeltaT} is contained in $\Delta_s \subset A^* =\Delta_{\max}$. Since 
   $\ran (\widetilde \tau_D \upharpoonright \dom \Delta_s) = H^{s-1/2}(\partial\Omega)$ 
   and $\iota_-$ is an isometry from $H^{s-1/2}(\partial\Omega)$ onto $H^{s}(\partial\Omega)$ 
   the assertion \eqref{eqn:lapreg} follows from the abstract regularity result \eqref{eqn:mainlaplace2} in Corollary~{\rm\ref{cor:mainlaplace}}.
\end{proof}
\subsection{Elliptic differential operators of order $2m$ on bounded smooth domains}\label{sec:elliptic}
In this subsection we illustrate some of the abstract results from Section~\ref{sec:qbt} and Section~\ref{sec:extension}
for elliptic differential operators of order $2m$ on a bounded smooth domain. The description of the selfadjoint realizations
in this case can already be found in Grubb \cite{Gr68}, other extension properties obtained below can be found in the monograph of
Lions and Magenes \cite{LiMa72}. 
We also refer the reader to the classical contributions \cite{BF62,Be65,Br60,F62,Gr68,LiMa72,Sc59} for more details on the notation 
and references, and to, e.g. \cite{BrGrWo09,Gr12,Ma10} for 
some recent connected publications.

Let $\Omega \subset \mathbb R^n$, $n\geq 2$, be a bounded domain 
with $C^\infty$-boundary $\partial\Omega$. 
Let $A$ and $T$ be the realizations of the $2m$-th order, properly elliptic, formally self-adjoint differential expression
\begin{equation*}
  \mathscr L := \sum_{|\alpha|, |\beta| \leq m} (-1)^{|\alpha|} \partial^\alpha a_{\alpha \beta} \, \partial^\beta, \qquad  
  a_{\alpha \beta}\in C^\infty(\overline \Omega), 
\end{equation*}
on $H^{2m}_0(\Omega)$ and $H^{2m}(\Omega)$, respectively; cf. \cite[Chapter~2.1]{LiMa72} for more details. 
As in Section~\ref{sec:laplace} we define the Hilbert spaces 
\begin{equation}\label{eqn:HsL}
  H^s_{\mathscr L}(\Omega) := \bigl\{f\in H^s(\Omega) : \mathscr L f \in L^2(\Omega) \bigr\},\quad s\geq 0,
\end{equation}
with norms induced by the inner products given by 
\begin{equation}\label{eqn:HsLip}
( f, g)_{ H^s_{\mathscr L}(\Omega)} := (f, g)_{H^s(\Omega)} 
  +  (\mathscr L f, \mathscr L g)_{L^2(\Omega)}, \quad f,g\in H^s_{\mathscr L}(\Omega).
\end{equation}
We note that $H^s_{\mathscr L}(\Omega) = H^s(\Omega)$ with equivalent norms if $s\geq 2m$ 
and that $C^\infty(\overline\Omega)$ is dense in $H^s_{\mathscr L}(\Omega)$ for $s\geq 0$.
The minimal and the maximal realization of the differential expression $\mathscr L$ are given by
\begin{equation*}
    \mathscr L_{\min} := A =  \mathscr L \upharpoonright H^{2m}_0(\Omega)\quad\text{and}\quad 
    \mathscr L_{\max} := A^* = \mathscr L \upharpoonright H^0_{\mathscr L}(\Omega),
\end{equation*}
respectively.
We mention that $A$ is a closed, densely defined, symmetric operator in $L^2(\Omega)$ with equal 
infinite deficiency indices.

In the next theorem a quasi boundary triple for the elliptic differential operator 
$T$ is defined. Here we make use of normal systems $D =\{D_j\}_{j=0}^{m-1}$ and $N=\{N_j\}_{j=0}^{m-1}$ 
of boundary differential operators,
\begin{eqnarray}\label{eqn:Dop}
  D_j f &:=& \sum_{|\beta| \leq m_j}  b_{j \beta}\, \partial^\beta f \upharpoonright_{\partial\Omega}, \quad f \in H^{2m}(\Omega), 
  \quad m_j \leq 2 m-1, \\ \label{eqn:Nop}
  N_j f &:=& \sum_{|\beta| \leq \mu_j}  c_{j \beta}\, \partial^\beta f \upharpoonright_{\partial\Omega}, \quad f \in H^{2m}(\Omega), 
  \quad \mu_j \leq 2 m-1,
  \end{eqnarray}
with $C^\infty$ coefficients $b_{j \beta}, c_{j \beta}$ on $\partial\Omega$ 
and which cover $\mathscr L$ on $\partial\Omega$; 
cf. \cite[Chapter~2.1]{LiMa72}.

\begin{Thm}\label{thm:BLtriple}
  Let $D$ be a normal system of boundary differential operators as in \eqref{eqn:Dop}.
  Then there exists a normal system of boundary differential operators $N$ of the form \eqref{eqn:Nop} 
  of order $\mu_j = 2m-m_j-1$, such that 
  $\{L^2(\partial\Omega)^m,\Gamma_0,\Gamma_1\}$,
  \[ 
  \Gamma_0,\Gamma_1 : H^{2m}(\Omega) \to L^2(\partial\Omega)^m,\quad \Gamma_0 f := D f, \quad \Gamma_1 f := N f
  \]
  is a quasi boundary triple for $T \subset A^*$.
  The minimal realization $A=\mathscr L_{\min}$ coincides with $T\upharpoonright\ker\Gamma$
  and the following statements hold.
  \begin{enumerate}[\rm(i)]
  \item The Dirichlet realization $\mathscr L_{D}$ and Neumann realization $\mathscr L_{N}$ correspond to $\ker \Gamma_0$ and $\ker \Gamma_1$, 
    \begin{equation*}
      \begin{split}
        \mathscr L_D &:=  T\upharpoonright \ker \Gamma_0 = \mathscr L_{\max} \upharpoonright \bigl\{ f\in H^{2m} (\Omega): D f =0\bigr \}, \\
        \mathscr L_N &:=  T\upharpoonright \ker \Gamma_1 = \mathscr L_{\max} \upharpoonright \bigl\{ f\in H^{2m} (\Omega): N f =0\bigr \},
      \end{split}
    \end{equation*}
    respectively, and $ \mathscr L_D$ is self-adjoint in $L^2(\Omega)$.
  \item The spaces 
    \begin{equation}\label{eqn:G0G1ell}
      \begin{split}
        \mathscr G_0 &:= \ran (\Gamma_0\upharpoonright \ker\Gamma_1) = \prod_{j=0}^{m-1} H^{2m-m_j-1/2}(\partial\Omega), \\ 
        \mathscr G_1 &:= \ran (\Gamma_1\upharpoonright \ker\Gamma_0) = \prod_{j=0}^{m-1} H^{m_j+1/2}(\partial\Omega)
      \end{split}
    \end{equation}
    are dense in $L^2(\partial\Omega)^m$.
  \item The values $\gamma(\lambda) : L^2(\partial\Omega)^m \supset \prod_{j=0}^{m-1} H^{2m-m_j-1/2}(\partial\Omega) \to L^2(\Omega)$  
    of the $\gamma$-field are given by
    \begin{equation*}
      \gamma(\lambda) \varphi = f, \qquad \varphi\in \prod_{j=0}^{m-1} H^{2m-m_j-1/2}(\partial\Omega), \quad \lambda \in \rho(\mathscr L_D),
    \end{equation*}
    where $f\in L^2(\Omega)$ is the unique solution of the boundary value problem 
    \begin{equation}\label{eqn:2morderPDE}
        (\mathscr L -\lambda) f = 0,\qquad  D f = \varphi.
    \end{equation}
  \item The values 
    $M(\lambda) :  L^2(\partial\Omega)^m \supset \prod_{j=0}^{m-1} H^{2m-m_j-1/2}(\partial\Omega) \to  L^2(\partial\Omega)^m$ 
    of the Weyl function are given by
    \begin{equation*}
      M(\lambda) \varphi = N f, \qquad \varphi\in \prod_{j=0}^{m-1} H^{2m-m_j-1/2}(\partial\Omega), \quad \lambda \in \rho(\mathscr L_D),
    \end{equation*}
    where $f = \gamma(\lambda) \varphi$ is the unique solution of \eqref{eqn:2morderPDE}.
  \end{enumerate}
\end{Thm}

\begin{proof}
  First we remark that $C^\infty(\overline\Omega)$, and hence $H^{2m}(\Omega)$, is dense in $H^0_{\mathscr L}(\Omega)$.
  This implies $\overline T = A^*$. 
  According to \cite[Chapter~2.1]{LiMa72} for a given normal system $D$ 
  of boundary differential operators as in \eqref{eqn:Dop} there exists a system
  a normal system $N$ of boundary differential operators of the form \eqref{eqn:Nop} 
  of order $\mu_j = 2m-m_j-1$ such that $\{D, N\}$ is a Dirichlet system of order $2\,m$, which acts as a mapping 
  from $H^{2m}(\Omega)$ onto 
  \begin{equation}\label{eqn:elltrace}
    \prod_{j=0}^{m-1} H^{2m-m_j-1/2}(\partial\Omega)\times 
    \prod_{j=0}^{m-1} H^{m_j+1/2}(\partial\Omega)\hookrightarrow L^2(\partial\Omega)^{2m}.
  \end{equation} 
  The kernel of this map is $H^{2m}_0(\Omega)$ and Green's formula 
  \begin{equation*}
    \begin{split}
      (\mathscr L f,\, g)_{L^2(\Omega)} - (f,\, \mathscr L & g)_{L^2(\Omega)}\\  
      &= (N f,\, D g)_{L^2(\partial\Omega)^m} - (D f,\, N g)_{L^2(\partial\Omega)^m}
    \end{split}
\end{equation*}
  holds for all $f,g\in H^{2m}(\Omega)$; cf.~\cite[Theorem~2.2.1]{LiMa72}.
  From \eqref{eqn:elltrace} we conclude that \eqref{eqn:G0G1ell} holds and
  the spaces $\mathscr G_0$ and $\mathscr G_1$ are dense in $L^2(\partial\Omega)^m$.
  This also implies that $\ran \Gamma$ is dense in $L^2(\partial\Omega)^m\times L^2(\partial\Omega)^m$. 
  Moreover $A_0:=T \upharpoonright \ker \Gamma_0 = \mathscr L_D$ is self-adjoint in $L^2(\Omega)$ by \cite[Theorem~2.8.4]{LiMa72}.
  Hence $\{L^2(\partial\Omega)^m,\Gamma_0,\Gamma_1\}$ is a quasi boundary triple for $T\subset A^*$
  with $T\upharpoonright \ker \Gamma = \mathscr L_{\min} = A$.
  The remaining statements follows from the definition of the $\gamma$-field and the Weyl function. 
\end{proof}
The next two corollaries show that the abstract theory from Section~\ref{2.3} implies some fundamental extension results due to Lions and Magenes. 
The proofs immediately follow from Proposition~\ref{pro:Gamma_0_extension}, Corollary~\ref{cor:Gamma_1_extension} 
and standard interpolation theory of Sobolev spaces, see also Lemma~\ref{lem:exgamma} and Lemma~\ref{lem:exweyl}.
\begin{Cor}\label{cor:BLtriple}
  Let $\{L^2(\partial\Omega)^m,\Gamma_0,\Gamma_1\}$ be the quasi boundary triple for $T\subset A^*$ from Theorem~{\rm\ref{thm:BLtriple}}
  with Weyl function $M$. Then the following statements hold.
  \begin{enumerate}[\rm(i)]
  \item The mapping $\Gamma_0 = D$ admits a continuous extension to a surjective mapping
    \begin{equation}\label{eqn:Dext}
      \widetilde D : H^0_{\mathscr L}(\Omega) \to \prod_{j=0}^{m-1} H^{-m_j-1/2}(\partial\Omega)
    \end{equation}
    such that $\ker \widetilde D = \ker D = \dom \mathscr L_D$.
  \item The norm 
    \begin{equation*}
      \| \Lambda^{-1/2} f \|_{L^2(\partial\Omega)^m},\quad \Lambda := \overline{\im M(i)}, \quad f\in \prod_{j=0}^{m-1} H^{m_j+1/2}(\partial\Omega),
    \end{equation*}
    defines an equivalent norm on $\prod_{j=0}^{m-1} H^{m_j+1/2}(\partial\Omega)$.
  \end{enumerate}
\end{Cor}

In the next corollary we assume, in addition, that $\mathscr L_N=T\upharpoonright\ker\Gamma_1$ is self-adjoint.

\begin{Cor}\label{cor:BLtriple2}
 Let $\{L^2(\partial\Omega)^m,\Gamma_0,\Gamma_1\}$ be the quasi boundary triple for $T\subset A^*$ from Theorem~{\rm\ref{thm:BLtriple}}
  with $\gamma$-field $\gamma$ and Weyl function $M$. 
  Assume that the realization $\mathscr L_N$ of $\mathscr L$ is self-adjoint in $L^2(\Omega)$.
  Then the following statements hold.
  \begin{enumerate}[\rm(i)]
  \item 
    The mapping $\Gamma_1 = N$ admits a continuous extension to
    a surjective mapping
    \begin{equation} \label{eqn:Next}
      \widetilde N : H^0_{\mathscr L}(\Omega) \to \prod_{j=0}^{m-1} H^{-2m+m_j+1/2}(\partial\Omega)
    \end{equation}
    such that $\ker \widetilde N = \ker N = \dom \mathscr L_N$.
  \item The norm 
    \[
    \| \Sigma^{-1/2} f \|_{L^2(\partial\Omega)^m},\quad \Sigma := \overline{\im(-M(i)^{-1})}, \quad f\in\prod_{j=0}^{m-1} H^{2m-m_j-1/2}(\partial\Omega),
    \]
    defines an equivalent norm on $\prod_{j=0}^{m-1} H^{2m-m_j-1/2}(\partial\Omega)$.
  \item The values of the $\gamma$-field $\gamma$ and the Weyl function $M$ admit continuous extensions 
    \begin{equation*}
      \begin{split}
      \widetilde \gamma(\lambda) &: \prod_{j=0}^{m-1} H^{-m_j-1/2}(\partial\Omega)\to L^2(\Omega),\\
      \widetilde M(\lambda) &:  \prod_{j=0}^{m-1} H^{-m_j-1/2}(\partial\Omega)\to \prod_{j=0}^{m-1} H^{-2m+m_j+1/2}(\partial\Omega)
      \end{split}
    \end{equation*}
    for all $\lambda \in \rho(\mathscr L_D)$.
  \item The restrictions 
    \begin{equation} \label{eqn:megatrace}
      \begin{split}
        \widetilde D \upharpoonright H^s_{\mathscr L}(\Omega) &: H^s_{\mathscr L}(\Omega) \to \prod_{j=0}^{m-1} H^{s-m_j-1/2}(\partial\Omega), \\
        \widetilde N \upharpoonright H^s_{\mathscr L}(\Omega) &: H^s_{\mathscr L}(\Omega) \to \prod_{j=0}^{m-1} H^{s-2m+m_j+1/2}(\partial\Omega)
      \end{split}
    \end{equation}
    are continuous and surjective for all $s\in[0,2\,m]$.
  \end{enumerate}
\end{Cor}

 Corollary~\ref{cor:BLtriple} and Corollary~\ref{cor:BLtriple2} imply that the maximal possible domain for a quasi boundary triple with boundary mappings $\widetilde D$ 
and $\widetilde N$ is given by the space $H^{2m-1/2}_{\mathscr L}(\Omega)$, see also \cite{Be65}.
\begin{Pro}\label{pro:BLtriples}
  Let $s\in[0, 2m]$, $T_s := \mathscr L_{\max} \upharpoonright H^s_{\mathscr L}(\Omega)$, assume that $\mathscr L_N$ is self-adjoint and let   
  \begin{equation*}
    \begin{split}
      \Gamma_0^s&: H^s_{\mathscr L} (\Omega) \to\prod_{j=0}^{m-1} H^{s-m_j-1/2}(\partial\Omega),\quad \Gamma_0^s f := \widetilde D f, \\
      \Gamma_1^s&: H^s_{\mathscr L} (\Omega) \to\prod_{j=0}^{m-1} H^{s-2m+m_j+1/2}(\partial\Omega),\quad\Gamma_1^s f := \widetilde N f.
    \end{split}
  \end{equation*}
  Then the spaces 
  \begin{equation*}
    \begin{split}
      \mathscr G_0 &= \ran (\Gamma_0^s\upharpoonright \ker\Gamma_1^s) = \prod_{j=0}^{m-1} H^{2m-m_j-1/2}(\partial\Omega),\\
      \mathscr G_1 &= \ran (\Gamma_1^s\upharpoonright \ker\Gamma_0^s) = \prod_{j=0}^{m-1} H^{m_j+1/2}(\partial\Omega)
    \end{split}
  \end{equation*} 
  are dense in $L^2(\partial\Omega)$ and do not depend on $s$.
  Moreover, if $s\in [2m-\frac 12, 2m]$ then $\ran \Gamma_0^s\subset L^2(\partial\Omega)^m$, 
  $\ran \Gamma_1^s \subset L^2(\partial\Omega)^m$, and 
  $\{L^2(\partial\Omega)^m,\Gamma_0^s,\Gamma_1^s\}$ is a quasi boundary triple for $T_s \subset A^* = \mathscr L_{\max}$.
\end{Pro}

By applying Theorem~\ref{thm:regularisation} to the quasi boundary triple 
$\{L^2(\partial\Omega)^m,\Gamma_0,\Gamma_1\}$
from Theorem~\ref{thm:BLtriple} we get the ordinary boundary triple introduced by Grubb in \cite{Gr68}, 
see also \cite{BrGrWo09, Gr09} and \cite[Proposition~3.5,~5.1]{Ma10}. 
Recall that there exist isometric isomorphisms 
\begin{equation}\label{iotas2}
  \iota_\pm : \prod_{j=0}^{m-1}H^{\pm (m_j + 1/2)}(\partial \Omega) \to L^2(\partial \Omega)^m
\end{equation}
such that
$(\iota_- x',\, \iota_+ x)_{L^2(\partial\Omega)^m} =  \langle x',\, x  \rangle_{\mathscr G_1' \times \mathscr G_1}$
holds for all 
\begin{equation}\label{iotas3}
x\in\mathscr G_1 = \prod_{j=0}^{m-1} H^{m_j+1/2}(\partial\Omega) \quad\text{and}\quad
x'\in\mathscr G_1'=\prod_{j=0}^{m-1} H^{-m_j-1/2}(\partial\Omega);
\end{equation}
cf.~\cite[(3.4)--(3.6)]{Ma10}.
\begin{Cor}\label{cor:Grubbtriple}
  Let $\eta\in \rho(\mathscr L_D)\cap \mathbb R$ and define $\Upsilon_0, \Upsilon_1 : H^0_{\mathscr L}(\Omega) \to L^2(\partial\Omega)^m$ by
  \begin{equation*} 
  \Upsilon_0 f := \iota_- \widetilde D f, \quad \Upsilon_1 f := \iota_+ N f_D, \quad  f=f_D + f_\eta \in \dom \mathscr L_D\dotplus\mathcal N_\eta(A^*).
  \end{equation*}
  Then $\{L^2(\partial\Omega)^m,\Upsilon_0,\Upsilon_1\}$ is an ordinary boundary triple for $A^* = \mathscr L_{\max}$ with
  $A^*\upharpoonright\ker\Upsilon_0=\mathscr L_D$ and  
  \[
  A^*\upharpoonright\ker\Upsilon_1=\mathscr L_{\min}\dot + \bigl\{(f_\eta,\eta f_\eta)^\top: 
  \mathscr L f_\eta=\eta f_\eta,\,f_\eta\in H^0_{\mathscr L} (\Omega)\bigr\}.
  \]
\end{Cor}

As an example of the consequences of the abstract results from Section~\ref{sec:qbt} and Section~\ref{sec:extension}
we state only a version of Kre\u{\i}n's formula for the case of $2m$-th order elliptic differential operators.
We leave it to the reader to formulate the other corollaries from the general results, e.g. the description of
the closed (symmetric,  self-adjoint, (maximal) dissipative, (maximal) accumulative, respectively) extensions 
  $\mathscr L_\vartheta \subset \mathscr L_{\max}$ of 
  $\mathscr L_{\min}$ in $L^2(\Omega)$, regularity results or sufficient criteria for self-adjointness, see also Section~\ref{sec:ellipticsecond}
  for the second order case.
\begin{Cor}\label{cor:superkrein}
  Let $\{L^2(\partial\Omega)^m,\Gamma_0,\Gamma_1\}$ be the quasi boundary triple from Theorem~\ref{thm:BLtriple}, and let 
  $\widetilde \gamma(\lambda)$ and $\widetilde M(\lambda)$, $\lambda \in \rho(\mathscr L_D)$, be the extended $\gamma$-field and Weyl function, 
  respectively. Assume that $\mathscr L_N$ is self-adjoint, that
  \begin{equation*}
    \vartheta \subset \prod_{j=0}^{m-1} H^{-m_j-1/2}(\partial\Omega) \times \prod_{j=0}^{m-1} H^{-2m+m_j+1/2}(\partial\Omega)
  \end{equation*}
  is a linear relation in $\ran (\widetilde D, \widetilde N)$ and that the corresponding extension 
  \begin{equation*}
    \mathscr L_{\vartheta} := \mathscr L_{\max} \upharpoonright \bigl\{ f\in H^0_{\mathscr L}(\Omega) : \vartheta 
     \widetilde D f -\widetilde N f =0  \bigr\}
  \end{equation*}
  is closed in $L^2(\Omega)$.
  Then for all $\lambda\in \rho(\mathscr L_{D})$ the following assertions {\rm(i)}-{\rm(iv)} hold:
  \begin{enumerate}[\rm(i)]
  \item $\lambda \in \sigma_p(\mathscr L_{\vartheta})$ 
    if and only if $0\in \sigma_p(\iota_{+} ( \vartheta-\widetilde M(\lambda) ) \iota_{-}^{-1})$, in this case
    \[
    \ker (\mathscr L_{\vartheta}-\lambda) =\widetilde \gamma(\lambda) \ker(\vartheta-\widetilde M(\lambda)),
    \]
  \item $\lambda \in \sigma_c(\mathscr L_{\vartheta})$ if and only if $0\in \sigma_c(\iota_{+} ( \vartheta-\widetilde M(\lambda)  ) \iota_{-}^{-1})$,
  \item $\lambda \in \sigma_r(\mathscr L_{\vartheta})$ if and only if $0\in \sigma_r(\iota_{+} ( \vartheta-\widetilde M(\lambda)  ) \iota_{-}^{-1})$,
  \item $\lambda \in \rho(\mathscr L_{\vartheta})$ if and only if $0\in \rho( \iota_{+} ( \vartheta-\widetilde M(\lambda) ) \iota_{-}^{-1})$ and 
  \begin{equation*}
    ( \mathscr L_{\vartheta} -\lambda )^{-1} 
    = (\mathscr L_{D} -\lambda )^{-1} 
    +\widetilde \gamma(\lambda)  \bigl(\vartheta-\widetilde M(\lambda)\bigr)^{-1} \widetilde \gamma(\bar\lambda)'
  \end{equation*}
  holds for all $\lambda\in\rho(\mathscr L_{\vartheta})\cap\rho(\mathscr L_{D})$.
  \end{enumerate}
\end{Cor}
\subsection{Second order elliptic differential operators on smooth domains with compact boundary}\label{sec:ellipticsecond}
In this section we pay particular attention to a special second order case which appears in the literature in different contexts, 
see, e.g., \cite{Be10,BeLa12,BeLaLo11,BeLaLo13,Gr11,Gr11a,Gr11b}. 

Let $\Omega\subset\mathbb R^n$, $n\geq 2$, be a bounded 
or unbounded domain with a compact $C^\infty$-smooth boundary $\partial\Omega$ and consider the 
second order differential expression on $\Omega$ given by
\begin{equation*}
  \mathscr L =  - \sum_{j,\, k=1}^n \partial_j  a_{j k} \partial_k  + a
\end{equation*}
with coefficients $a_{j k}\in C^\infty(\overline\Omega)$ such that $a_{j k}(x) = a_{k j}(x)$ for all $x\in\overline\Omega$ and 
$j,\, k \in \{1,\dots,n\}$,  
and $a\in L^\infty(\Omega)$ real. In the case that $\Omega$ is unbounded we also assume that the first partial derivatives
of the functions $a_{jk}$ are bounded in $\Omega$.
Furthermore, for some $c>0$ the ellipticity condition 
$\sum_{j,\, k=1}^n a_{j k}(x) \xi_j \xi_k \geq c \sum_{k=1}^n \xi_k^2$ is assumed to hold for all $\xi\in \mathbb R^n$ and $x\in\overline\Omega$.
As in Section~\ref{sec:elliptic} we define the Hilbert spaces $H^s_{\mathscr L}(\Omega)$ and inner products via \eqref{eqn:HsL} and \eqref{eqn:HsLip},
respectively. 
The minimal and maximal realization of the differential expression $\mathscr L$ are 
\begin{equation*}
  A = \mathscr L_{\min} = \mathscr L \upharpoonright H^2_0(\Omega)\quad\text{and}\quad
  A^* = \mathscr L_{\max} = \mathscr L \upharpoonright H^0_{\mathscr L}(\Omega),
\end{equation*}
and we set $T := \mathscr L \upharpoonright H^2(\Omega)$. The minimal operator
$A$ is a closed, densely defined, symmetric operator in $L^2(\Omega)$ with equal 
infinite deficiency indices. The Dirichlet and Neumann trace operator are defined by
\begin{equation*}
  \tau_D = f \upharpoonright_{\partial\Omega}\quad\text{and}\quad \tau_N f 
  =  \sum_{j,\, k=1}^n a_{j k} \mathfrak n_j \partial_k f\upharpoonright_{\partial\Omega},\quad f\in C^\infty(\overline\Omega),
\end{equation*}
and extended by continuity to a surjective mapping
$(\tau_D,\tau_N)^\top:H^2(\Omega)\rightarrow H^{3/2}(\partial\Omega)\times H^{1/2}(\partial\Omega)$; cf. \cite{LiMa72}.
Here $\mathfrak n=(\mathfrak n_1,\,\mathfrak n_2,\,\dots,\,\mathfrak n_n)^\top$ 
denotes the unit vector field pointing out of $\Omega$.

The next theorem is a variant of Theorem~\ref{thm:LipBLtriple} and Theorem~\ref{thm:BLtriple} with 
$D=\tau_D$ and $N=-\tau_N$; cf.~\cite{BeLa12,BeLaLo11}. We do not repeat the proof here and refer only to \cite[Theorem 5]{Br60} and 
\cite[Theorem 7.1]{Be65} 
for the self-adjointness of $\mathscr L_D$ and $\mathscr L_N$, respectively.
As in the previous theorems the spaces $\mathscr G_0$ and $\mathscr G_1$ from Definition~\ref{def:g0g1} turn out to be dense in $L^2(\partial\Omega)$,
the $\gamma$-field coincides with a family of Poisson operators and the values of the Weyl function are (up to a minus sign) 
Dirichlet-to-Neumann maps.
\begin{Thm}\label{thm:BLtriplesecond}
Let $T=\mathscr L\upharpoonright H^2(\Omega)$ and let
  \[ 
  \Gamma_0,\Gamma_1 : H^2(\Omega) \to L^2(\partial\Omega),\qquad \Gamma_0 f := \tau_D f, \quad \Gamma_1 f := -\tau_N f.
  \]
  Then $\{L^2(\partial\Omega),\Gamma_0,\Gamma_1\}$ 
  is a quasi boundary triple for $T \subset A^*=\mathscr L_{\max}$ such that the minimal realization $A=\mathscr L_{\min}$ coincides with
  $T\upharpoonright\ker\Gamma$ and the following statements hold.
  \begin{enumerate}[\rm(i)]
  \item The Dirichlet realization $\mathscr L_{D}$ and Neumann realization $\mathscr L_{N}$ correspond to $\ker \Gamma_0$ and $\ker \Gamma_1$, 
    \begin{equation*}
      \begin{split}
        \mathscr L_D &:=  T\upharpoonright \ker \Gamma_0 = \mathscr L_{\max} \upharpoonright \bigl\{ f\in H^{2} (\Omega): \tau_D f =0\bigr \}, \\
        \mathscr L_N &:=  T\upharpoonright \ker \Gamma_1 = \mathscr L_{\max} \upharpoonright \bigl\{ f\in H^{2} (\Omega): \tau_N f =0\bigr \},
      \end{split}
    \end{equation*}
    respectively, and both operators are self-adjoint in $L^2(\Omega)$. 
  \item The spaces 
    \begin{equation*}
      \begin{split}
        \mathscr G_0 &:= \ran (\Gamma_0\upharpoonright \ker\Gamma_1) = H^{3/2}(\partial\Omega), \\ 
        \mathscr G_1 &:= \ran (\Gamma_1\upharpoonright \ker\Gamma_0) = H^{1/2}(\partial\Omega)
      \end{split}
    \end{equation*}
    are dense in $L^2(\partial\Omega)$.
  \item The values $\gamma(\lambda) : L^2(\partial\Omega) \supset H^{3/2}(\partial\Omega) \to L^2(\Omega)$  
    of the $\gamma$-field are given by
    \begin{equation*}
      \gamma(\lambda) \varphi = f, \qquad \varphi\in H^{3/2}(\partial\Omega), \quad \lambda \in \rho(\mathscr L_D),
    \end{equation*}
    where $f\in L^2(\Omega)$ is the unique solution of the boundary value problem 
    \begin{equation}\label{eqn:secondPDE}
        (\mathscr L -\lambda) f = 0 \qquad \tau_D f = \varphi.
    \end{equation}
  \item The values 
    $M(\lambda) :  L^2(\partial\Omega)\supset H^{3/2}(\partial\Omega) \to  L^2(\partial\Omega)$ 
    of the Weyl function are given by
    \begin{equation*}
      M(\lambda) \varphi = -\tau_N f, \qquad \varphi\in H^{3/2}(\partial\Omega), \quad \lambda \in \rho(\mathscr L_D),
    \end{equation*}
    where $f = \gamma(\lambda) \varphi$ is the unique solution of \eqref{eqn:secondPDE}.
  \end{enumerate}
\end{Thm}

Let $\{L^2(\partial\Omega),\Gamma_0,\Gamma_1\}$ be the quasi boundary triple from Theorem~\ref{thm:BLtriplesecond}.
In the same way as in  \eqref{eqn:Dext} and \eqref{eqn:Next}   
we obtain that $(\tau_D,\tau_N )^\top$ admits a continuous extension to a mapping 
\begin{equation*}
  (\widetilde \tau_D, \widetilde \tau_N)^\top : H^0_{\mathscr L}(\Omega) \to H^{-1/2}(\partial\Omega) \times H^{-3/2}(\partial\Omega),
\end{equation*}
where for all $s\in [0,2]$ the restrictions
\begin{equation*}
 \begin{split}
  \widetilde \tau_D \upharpoonright H^s_{\mathscr L}(\Omega)&: H^s_{\mathscr L}(\Omega) \to H^{s-1/2}(\partial\Omega),\\
  \widetilde \tau_N \upharpoonright H^s_{\mathscr L}(\Omega)&: H^s_{\mathscr L}(\Omega) \to H^{s-3/2}(\partial\Omega),
 \end{split}
\end{equation*}
are continuous and surjective; 
cf.~\eqref{eqn:megatrace}.

The quasi boundary triples in the next proposition were first introduced in \cite{BeLa07} on the domains
$H^2(\Omega)$ and $H^{3/2}_{\mathscr L}(\Omega)$. We note that the latter space coincides with the first order Beals space 
$\mathscr B^1_{\mathscr L}(\Omega)$, see \cite{Be65}.

\begin{Pro}\label{pro:secondBLtriples}
  Let $s\in[0, 2]$, $T_s := \mathscr L_{\max} \upharpoonright H^s_{\mathscr L}(\Omega)$, and let 
  \begin{equation*}
    \begin{split}
      \Gamma_0^s: H^s_{\mathscr L} (\Omega) \to H^{s-1/2}(\partial\Omega),\quad \Gamma_0^s f &:= \widetilde\tau_D f, \\
      \Gamma_1^s: H^s_{\mathscr L} (\Omega) \to H^{s-3/2}(\partial\Omega),\quad \Gamma_1^s f &:= -\widetilde\tau_N f.
    \end{split}
  \end{equation*}
  Then the spaces 
  \begin{equation*}
    \begin{split}
      \mathscr G_0 &= \ran (\Gamma_0^s\upharpoonright \ker\Gamma_1^s) = H^{3/2}(\partial\Omega),\\
      \mathscr G_1 &= \ran (\Gamma_1^s\upharpoonright \ker\Gamma_0^s) = H^{1/2}(\partial\Omega)
    \end{split}
  \end{equation*} 
  are dense in $L^2(\partial\Omega)$ and do not depend on $s$.
  Moreover, if $s\in [\frac 32, 2]$ then $\ran \Gamma_0^s\subset L^2(\partial\Omega)$, 
  $\ran \Gamma_1^s \subset L^2(\partial\Omega)$, and 
  $\{L^2(\partial\Omega),\Gamma_0^s,\Gamma_1^s\}$ is a quasi boundary triple for $T_s \subset A^* = \mathscr L_{\max}$.
\end{Pro}
If we apply Theorem~\ref{thm:regularisation} to a quasi boundary triple from 
Proposition~\ref{pro:secondBLtriples} we obtain the second order version of the 
ordinary boundary triple $\{L^2(\partial\Omega),\Upsilon_0,\Upsilon_1\}$ for $A^*=\mathscr L_{\max}$ 
from Corollary~\ref{cor:Grubbtriple}. This boundary triple appears already in~\cite{Gr68} in an implicit form,
see also \cite{Be10, BeLa12, BrGrWo09, Gr09, Ma10, PoRa09}. 
Let again $\iota_\pm : H^{\pm 1/2}(\partial \Omega) \to L^2(\partial \Omega)$ be a pair of isometric isomorphisms such that
\begin{equation*}
(\iota_- x',\, \iota_+ x)_{L^2(\partial\Omega)} =  \langle x',\, x  \rangle_{H^{-1/2}(\partial\Omega) \times H^{1/2}(\partial\Omega)}
\end{equation*}
holds for all $x\in H^{1/2}(\partial\Omega)$ and $x'\in H^{-1/2}(\partial\Omega)$; cf.~\eqref{iotas}, \eqref{iotas2} and \eqref{iotas3}.
\begin{Cor}\label{cor:Grubbtriplesecond}
   Let $\eta\in \rho(\mathscr L_D)\cap \mathbb R$ and define
   $\Upsilon_0, \Upsilon_1 : H^0_{\mathscr L}(\Omega) \to L^2(\partial\Omega)$ by
  \begin{equation*} 
    \Upsilon_0 f := \iota_-\widetilde \tau_D f, \quad \Upsilon_1 f := -\iota_+ \tau_N f_D,
    \quad f=f_D+f_\eta\in\dom \mathscr L_D\dotplus\mathcal N_\eta(A^*).
  \end{equation*}
  Then $\{L^2(\partial\Omega),\Upsilon_0,\Upsilon_1\}$ is an ordinary boundary triple for $A^*=\mathscr L_{\max}$ with
  $A^*\upharpoonright\ker\Upsilon_0=\mathscr L_D$ and  
  \[
  A^*\upharpoonright\ker\Upsilon_1=\mathscr L_{\min}\dot + \bigl\{(f_\eta,\eta f_\eta)^\top: 
  \mathscr L f_\eta=\eta f_\eta,\,f_\eta\in H^0_{\mathscr L}(\Omega)\bigr\}.
  \]
\end{Cor}
%

As in Section~\ref{sec:laplace} we apply Theorem~\ref{thm:main3} to the quasi boundary triple from Theorem~\ref{thm:BLtriplesecond}. 
The regularity statement can be proven in the same way as in Theorem~\ref{thm:C1r}. 
\begin{Cor}\label{cor:mainellsecond}
  Let $\eta\in \mathbb R \cap \rho(\mathscr L_D)\cap \rho(\mathscr L_N)$ 
  and $\widetilde M(\eta):H^{-1/2}(\partial\Omega)\to H^{-3/2}(\partial\Omega)$ be the extended Dirichlet-to-Neumann map. 
  Then the mapping
  \begin{equation*}
   \Theta\mapsto \mathscr L_\vartheta = \mathscr L_{\max}\upharpoonright\bigl\{ f\in H^0_{\mathscr L}(\Omega) : 
    \vartheta\widetilde \tau_D f +\widetilde\tau_N f= 0 \bigr\},
    \quad \vartheta = \iota_+^{-1} \Theta \iota_- + \widetilde M(\eta),
  \end{equation*}
  establishes a bijective correspondence between all
  closed (symmetric,  self-adjoint, (maximal) dissipative, (maximal) accumulative) 
  linear relations $\Theta$ in $L^2(\partial\Omega)$ and 
  all closed (symmetric,  self-adjoint, (maximal) dissipative, (maximal) accumulative, respectively) extensions $\mathscr L_\vartheta\subset \mathscr L_{\max}$ of 
  $\mathscr L_{\min}$ in $L^2(\Omega)$.
  Moreover, the following regularity result holds: For $s \in [0,2]$
  \begin{equation*}
  \dom \Theta \subset H^s(\partial\Omega) \quad \text{implies} \quad \dom \mathscr L_{\vartheta} \subset H^s_{\mathscr L}(\Omega).
  \end{equation*}
\end{Cor}

Next we state a version of Lemma~\ref{lem:symmetric4} for the realizations of the second order elliptic differential expression in $\mathscr L$.

\begin{Lem}\label{lem:symmetric3}
  Let $\vartheta$ be a linear relation in $L^2(\partial\Omega)$. Then 
  \begin{equation*}
    \mathscr L_{\vartheta} := \mathscr L_{\max} \upharpoonright  
    \bigl\{ f\in H^0_{\mathscr L}(\Omega) : \vartheta \widetilde\tau_D f + \widetilde\tau_N f = 0\bigr\}
  \end{equation*}
  has regularity $\dom \mathscr L_\vartheta \subset H^{3/2}_{\mathscr L}(\Omega)$. 
  Moreover, $\mathscr L_\vartheta$ is symmetric in $L^2(\Omega)$ if and only if 
  $\vartheta$ is symmetric $L^2(\partial\Omega)$.
\end{Lem}

The next corollary is a consequence of Proposition~\ref{pro:assumptionpara} and Proposition~\ref{pro:assumptionpara2}. In item (i) we obtain 
an additional regularity statement.  

\begin{Cor}\label{cor:assumself}
  Let $\eta \in\mathbb R\cap \rho(\mathscr L_D)\cap \rho(\mathscr L_N)$ and
  $M(\eta):H^{3/2}(\partial\Omega)\to H^{1/2}(\partial\Omega)$ be the Dirichlet-to-Neumann map from Theorem~\ref{thm:BLtriplesecond}~{\rm (iv)}. 
  Let $\vartheta$ be a symmetric linear operator in $L^2(\partial\Omega)$ such that 
  \begin{equation}\label{eqn:varthetadom}
    H^{3/2}(\partial\Omega) \subset \dom \vartheta\quad\text{and}\quad  
    \ran \bigl(\vartheta\upharpoonright {H^{3/2}(\partial\Omega)}\bigr) \subset H^{1/2}(\partial\Omega),
  \end{equation}
  and assume that there exist $c_1>0$ and $c_2\in[0,1]$ such that 
  \begin{equation*}
    \|\vartheta x\|_{H^{1/2}(\partial\Omega)} 
    \leq c_1 \| x\|_{H^{-1/2}(\partial\Omega)} + c_2 \| M(\eta) x\|_{H^{1/2}(\partial\Omega)},\quad x\in H^{3/2}(\partial\Omega).
  \end{equation*}
  Then the following statements hold. 

  \begin{itemize}
   \item [{\rm (i)}] If $c_2\in[0,1)$ then 
    \begin{equation}\label{juli}
  \mathscr L_{\vartheta} = \mathscr L_{\max} \upharpoonright \bigl\{f\in H^0_{\mathscr L}(\Omega):\vartheta\widetilde \tau_D f + \widetilde\tau_N f = 0\bigr\}.
  \end{equation}
   is self-adjoint in $L^2(\Omega)$ with regularity
  $\dom \mathscr L_\vartheta \subset H^2(\Omega)$.
  \item [{\rm (ii)}] If $c_2=1$ then $\mathscr L_{\vartheta}$ in \eqref{juli} is essentially self-adjoint in $L^2(\Omega)$ with regularity
  $\dom \mathscr L_\vartheta \subset H^{3/2}_{\mathscr L}(\Omega)$. 
  \end{itemize}
\end{Cor}

\begin{proof}
  (i) The restriction $\theta := \vartheta \upharpoonright H^{3/2}(\partial\Omega) : H^{3/2}(\partial\Omega) \to H^{1/2}(\partial\Omega)$ satisfies
  the assumptions in Proposition~\ref{pro:assumptionpara}~(iii) and hence we conclude that 
  \begin{equation*}
  \mathscr L_{\theta} = \mathscr L_{\max} \upharpoonright \bigl\{f\in H^2(\Omega):\theta \tau_D f + \tau_N f = 0\bigr\}
  \end{equation*}
  is self-adjoint in $L^2(\Omega)$.  
  According to Lemma~\ref{lem:symmetric3} the operator $\mathscr L_\vartheta$ is a symmetric extension of the self-adjoint operator 
  $\mathscr L_\theta$ and hence both coincide.
  
  (ii) follows in the same way as (i) from Proposition~\ref{pro:assumptionpara2} and Lemma~\ref{lem:symmetric3}.
\end{proof}

In the next example we consider a one parameter family $\mathscr L_{\vartheta_\alpha}$ of extensions of $\mathscr L_{\min}$ 
which correspond to $\vartheta_\alpha = \alpha \, M(\eta)$. It turns out that for $\alpha\not=1$ the extensions are self-adjoint 
and for $\alpha=1$ essentially self-adjoint.

\begin{Exm}
  Let $M(\eta):H^{3/2}(\partial\Omega)\to H^{1/2}(\partial\Omega)$ be as in Corollary~\ref{cor:assumself} and consider the symmetric operators
  $\vartheta_\alpha := \alpha \, M(\eta)$, $\alpha\in\mathbb R$, in $L^2(\partial\Omega)$ with 
  $\dom \vartheta_\alpha = H^{3/2}(\partial\Omega)$ and $\alpha \in\mathbb R$.
  Then according to Corollary~\ref{cor:assumself} the extension
  \begin{equation*}
   \begin{split}
    \mathscr L_{\vartheta_\alpha}&=\mathscr L_{\max} \upharpoonright \bigl\{f\in H^0_{\mathscr L}(\Omega):\vartheta_\alpha \widetilde \tau_Df 
  + \widetilde\tau_N f = 0\bigr\}\\
  &=\mathscr L_{\max} \upharpoonright \bigl\{f\in H^2(\Omega):\alpha M(\eta) \tau_Df 
  + \tau_N f = 0\bigr\}
   \end{split}
  \end{equation*}
  in \eqref{juli} is self-adjoint if $|\alpha|<1$ and essentially self-adjoint if $|\alpha|=1$. Here we have used 
  $\widetilde\tau_D f=\tau_D f$ and $\widetilde\tau_N f=\tau_N f$ for $f\in H^2(\Omega)$. 
  It follows in the same way as in 
  Example~\ref{exm:kreinneumann} that 
  \begin{equation*}
  \begin{split}
   \mathscr L_{\vartheta_1}&=\mathscr L_{\max} \upharpoonright \bigl\{f\in H^2(\Omega): M(\eta)\tau_Df 
  + \tau_N f = 0\bigr\}\\   
   &=\mathscr L_{\min}\dot + \bigl\{(f_\eta,\eta f_\eta)^\top: 
  \mathscr L f_\eta=\eta f_\eta,\,f_\eta\in H^2(\Omega)\bigr\}.
  \end{split}
  \end{equation*} 
  We also remark that
  \[
  \overline{\mathscr L}_{\!\!\vartheta_1} =\mathscr L_{\min}\dot + \bigl\{(f_\eta,\eta f_\eta)^\top: 
  \mathscr L f_\eta=\eta f_\eta,\,f_\eta\in H^0_{\mathscr L}(\Omega)\bigr\}=\mathscr L_{\min}\dot +\widehat {\mathcal N}_\eta(A^*).
  \]
  For $\alpha\leq -1$ and $\alpha>1$ we make use of Corollary~\ref{cor:main2}. For this we set 
  \[
  \Theta_\alpha:=\iota_+ \bigl(\vartheta_\alpha - M(\eta) \bigr) \iota_-^{-1}= (\alpha - 1)  \iota_+ M(\eta) \iota_-^{-1},\quad 
  \dom \Theta_\alpha = H^2(\partial\Omega),
  \]
  and note that the operators $\Theta_\alpha$ are self-adjoint in $L^2(\partial\Omega)$. Hence Corollary~\ref{cor:main2} yields that
  for $\alpha\leq -1$ and $\alpha>1$
  the extensions $\mathscr L_{\vartheta_\alpha}$ are self-adjoint in $L^2(\Omega)$.
\end{Exm}

The following example is related to the case $\alpha=1$ in the above example. It contains an observation which can 
also be interpreted from a slightly more abstract point of view. Namely, Example~\ref{cex:1} shows that there exists 
a quasi boundary triple $\{\mathcal G,\Gamma_0,\Gamma_1\}$ for $T \subset A^*$ and a self-adjoint relation $\vartheta$ in $\mathcal G$ 
with $\vartheta\subset \ran \Gamma$ such that the extension 
$A_\vartheta := T \upharpoonright \{ f \in \dom T : \Gamma f \in \vartheta \}$
is not self-adjoint in $\mathcal H$; cf. Section~\ref{sec:parameterization}.

\begin{Exm}\label{cex:1}
  Let $\{L^2(\partial\Omega),\Gamma_0^s,\Gamma_1^s\}$ be the quasi boundary triple from Proposition~\ref{pro:secondBLtriples} for  $s=\frac 32$
  defined on the domain of
  \[
  T_{3/2} = \mathscr L_{\max}\upharpoonright H^{3/2}_{\mathscr L}(\Omega)\subset A^*.
  \]
  The values of the corresponding Weyl function $M_{3/2}$ are mappings from $H^1(\partial\Omega)$ to $L^2(\partial\Omega)$. For 
  $\eta \in \mathbb R\cap\rho(\mathscr L_D) \cap \rho(\mathscr L_N)$ set $\vartheta  := M_{3/2}(\eta) $ with $\dom \vartheta=H^1(\partial\Omega)$.
  Then $\vartheta$ is a bijective symmetric operator in $L^2(\partial\Omega)$ and hence self-adjoint. As in Example~\ref{exm:kreinneumann}
  one verifies that the corresponding extension $\mathscr L_\vartheta$ is given by
  \begin{equation*}
    \mathscr L_\vartheta
    = \mathscr L_{\max} \upharpoonright \bigl\{ f \in H^{3/2}_{\mathscr L}(\Omega) : \vartheta \widetilde \tau_D f + \widetilde \tau_N f = 0 \bigr\} \\
    = \mathscr L_{\min} \dotplus \widehat{\mathcal N}_\eta(T_{3/2})
  \end{equation*}
  and that $\overline{\mathscr L}_{\!\!\vartheta}=\mathscr L_{\min} \dotplus \widehat{\mathcal N}_\eta(A^*)=A^*\upharpoonright\ker\Upsilon_0$ holds;
  here $\Upsilon_0$ is the boundary mapping from Corollary~\ref{cor:Grubbtriplesecond}.
  Therefore $\mathscr L_\vartheta$ is a proper restriction of the self-adjoint extension 
  $\overline{\mathscr L}_{\!\!\vartheta}$ and it follows, in particular, that  $\mathscr L_\vartheta$ is essentially self-adjoint, but not self-adjoint in
  $L^2(\Omega)$. 
\end{Exm}

The next example is a variant of Example~\ref{exm:kreinneumann2}; cf. Proposition~\ref{pro:assumptionpara}~(iii).

\begin{Exm}
  Let $\vartheta$ be a symmetric linear operator in $L^2(\partial\Omega)$ with $H^{3/2}(\partial\Omega) \subset \dom \vartheta$ and 
  $\ran (\vartheta\upharpoonright {H^{3/2}(\partial\Omega)}) \subset H^{1/2}(\partial\Omega)$, and assume that $\vartheta$ is bounded from
  $(H^{3/2}(\partial\Omega), \|\cdot\|_{H^{-1/2}(\partial\Omega)} )$ to $H^{1/2}(\partial\Omega)$. Then the corresponding extension 
  \begin{equation*}
    \mathscr L_{\vartheta} = \mathscr L_{\max} \upharpoonright 
    \bigl\{f\in H^0_{\mathscr L}(\Omega):\vartheta\widetilde \tau_D f + \widetilde \tau_N f = 0\bigr\}
  \end{equation*}
  of $A=\mathscr L_{\min}$ is self-adjoint in $L^2(\Omega)$ with regularity
  $\dom \mathscr L_\vartheta \subset H^2(\Omega)$.
\end{Exm}

Proposition~\ref{pro:assumptionpara} together with well known compact embedding properties of Sobolev spaces 
yield some simple sufficient conditions for self-adjoint realizations of $\mathscr L$.

\begin{Pro}\label{pro:embedding}
  Let $\vartheta$ be a symmetric operator in $L^2(\Omega)$ such that \eqref{eqn:varthetadom} holds, and assume
  that $\vartheta$ is continuous as a mapping from $H^{3/2-\delta_1}(\partial\Omega)$ to $H^{1/2+\delta_2}(\partial\Omega)$, where 
  $\delta_1\in [0,\frac 32]$, $\delta_2\geq 0$ and $\delta_1+\delta_2>0$. 
  Then 
  \begin{equation*}
    \mathscr L_{\vartheta} = \mathscr L_{\max} \upharpoonright \bigl\{ f\in H^0_{\mathscr L}(\Omega): 
    \vartheta\widetilde \tau_D f + \widetilde \tau_N f = 0\bigr\}
  \end{equation*}
   is self-adjoint in $L^2(\Omega)$ with regularity $\dom \mathscr L_\vartheta \subset H^2(\Omega)$.
\end{Pro}

\begin{proof}
  Observe that at least one of the embeddings 
  $H^{3/2}(\partial\Omega) \hookrightarrow H^{3/2-\delta_1}(\partial\Omega)$ or
  $H^{1/2+\delta_2}(\partial\Omega)\hookrightarrow H^{1/2}(\partial\Omega)$ is compact; cf.  \cite[Theorem~7.10]{Wl87}. 
  Hence we conclude that 
  $\theta := \vartheta \upharpoonright H^{3/2}(\partial\Omega) : H^{3/2}(\partial\Omega) \to H^{1/2}(\partial\Omega)$
  is a compact operator. Therefore Proposition~\ref{pro:assumptionpara}~(i) yields that $\mathscr L_\theta$ is self-adjoint in $L^2(\Omega)$
  with regularity $\dom \mathscr L_\theta \subset H^2(\Omega)$; cf. the proof of Corollary~\ref{cor:assumself}. It follows from
  Lemma~\ref{lem:symmetric3} that $\mathscr L_\vartheta$ is a symmetric extension of the self-adjoint operator 
  $\mathscr L_\theta$ and hence both operators $\mathscr L_\vartheta$ and $\mathscr L_\theta$ coincide.
\end{proof}
Finally we illustrate Proposition~\ref{pro:embedding} with a simple example. 
\begin{Exm}\label{cor:robin} 
  Let $0<\varepsilon\leq\frac 32$ and assume that 
  \begin{equation*}
  \alpha \in \mathcal M\bigl(H^{3/2}(\partial\Omega),  H^{1/2+\varepsilon}(\partial\Omega)\bigr)\quad\text{or}\quad
  \alpha \in \mathcal M\bigl(H^{3/2-\varepsilon}(\partial\Omega), H^{1/2}(\partial\Omega)\bigr),
  \end{equation*}
  where $\mathcal M(\cdot,\cdot)$ denotes the space of all pointwise multipliers; cf.~\cite{MaSh08, Tr83}.
  Then it follows  from Proposition~\ref{pro:embedding} that
  \begin{equation*}
    \mathscr L_\alpha = \mathscr L_{\max} \upharpoonright 
    \bigl\{ f\in H^0_{\mathscr L}(\Omega) : \alpha\cdot\widetilde\tau_D f + \widetilde \tau_N f = 0\bigr\}
  \end{equation*}
  is self-adjoint in $L^2(\Omega)$ with regularity $\dom \mathscr L_\alpha \subset H^2(\Omega)$. In particular, since 
  $C^r(\partial\Omega) \subset \mathcal M(H^{1/2}(\partial\Omega), H^{1/2}(\partial\Omega))$ for $r\in (\frac 12,\, 1)$
  the assertion holds for all $\alpha\in C^r(\partial\Omega)$,   $r\in (\frac 12,\, 1)$.
\end{Exm}

\section*{Appendix}

Throughout this paper linear relations are used in different contexts. In this appendix we provide the necessary 
definitions. For more details we refer the reader to, e.g., \cite{A61,C98,HSS07}. 

Let $\mathcal G$ be a Hilbert space. A (closed) linear relation $\Theta$ 
in $\mathcal G$ is a (closed) linear subspace of $\mathcal G\times\mathcal G$. For the elements in a linear relation $\Theta$ usually a
vector notation is used, e.g. $\bigl(\begin{smallmatrix} x\\ x'\end{smallmatrix}\bigr)\in\Theta$.  For a linear relation $\Theta$ in $\mathcal G$ we shall write
\begin{equation*}
  \begin{split}
 \dom \Theta  :=& \left\{x \in \mathcal G : \begin{pmatrix} x \\ x' \end{pmatrix} \in \Theta\,\,\text{for some}\,\, x'\in\mathcal G\right\},\\
 \ran \Theta  :=& \left\{x' \in \mathcal G : \begin{pmatrix} x \\ x' \end{pmatrix} \in \Theta\,\,\text{for some}\,\, x\in\mathcal G\right\},\\
 \ker \Theta  :=& \left\{x \in \mathcal G : \begin{pmatrix} x \\ 0 \end{pmatrix} \in \Theta\right\},\\
 \mul \Theta  :=& \left\{x' \in \mathcal G : \begin{pmatrix} 0 \\ x' \end{pmatrix} \in \Theta\right\},
\end{split}
\end{equation*}
for the {\it domain}, {\it range}, {\it kernel} and {\it multivalued part} of $\Theta$, respectively. 
Note that each linear operator
$\Theta$ in  $\mathcal G$ is a linear relation if we identify the operator with its graph,
\begin{equation*}
\Theta=\left\{\begin{pmatrix} x\\ \Theta x\end{pmatrix}: x\in\dom \Theta\right\},
\end{equation*} 
and that a linear relation $\Theta$ is (the graph of) an operator if and only if the multivalued part of $\Theta$
is trivial, that is, $\mul \Theta=\{0\}$. 

The {\it inverse}
$\Theta^{-1}$ of a linear relation $\Theta$ in $\mathcal G$ is defined by
\begin{equation*}
\Theta^{-1}=\left\{\begin{pmatrix} x' \\ x\end{pmatrix}: \begin{pmatrix} x \\ x' \end{pmatrix} \in \Theta\right\}.
\end{equation*}
It is easy to see that $\dom \Theta^{-1}=\ran \Theta$ and $\ker \Theta^{-1}=\mul \Theta$ hold. 
The {\it sum} and {\it product} of two linear relations $\Theta_1$ and $\Theta_2$ in $\mathcal G$ are defined as
\begin{equation*}
  \begin{split}
    \Theta_1+\Theta_2 &:= \left\{\begin{pmatrix} x \\ x'+x'' \end{pmatrix} : 
    \begin{pmatrix} x \\ x' \end{pmatrix} \in \Theta_1,\, \begin{pmatrix} x  \\ x'' \end{pmatrix}\in \Theta_2\right\},\\ 
    \Theta_2 \Theta_1 &:= \left\{\begin{pmatrix} x \\ x''  \end{pmatrix}  : 
    \begin{pmatrix} x \\ x' \end{pmatrix} \in \Theta_1,\,  \begin{pmatrix} x' \\ x'' \end{pmatrix}\in \Theta_2\right\}. 
  \end{split}
\end{equation*}
If, e.g., $\Theta$ is a linear relation in $\mathcal G$, and $\iota_+:\mathscr G_1\rightarrow\mathcal G$ and 
$\iota_-:\mathscr G_1'\rightarrow\mathcal G$ are bijective linear operators then
\[
\iota_+^{-1}\Theta\iota_-=\left\{\begin{pmatrix} h \\ \iota_+^{-1} x'\end{pmatrix}: 
\begin{pmatrix} \iota_- h \\ x'\end{pmatrix}\in\Theta,\, h\in\mathscr G_1'\right\}\subset  \mathscr G_1' \times\mathscr G_1
\]
and $\iota_+^{-1}\Theta\iota_-$ can be viewed as a linear relation from $\mathscr G_1'$ in $\mathscr G_1$.

Next the definition of the resolvent set and the spectrum of a closed linear relation $\Theta$ in $\mathcal G$ is recalled.
A point $\lambda\in\mathbb C$ is
said to belong to the {\it resolvent set} $\rho(\Theta)$ of $\Theta$ 
if $(\Theta-\lambda)^{-1}$ is a bounded operator defined on $\mathcal G$.  
The {\em spectrum} $\sigma(\Theta)$ of
$\Theta$ is the complement of $\rho(\Theta)$ in $\mathbb C$, it decomposes into three disjoint components:
The {\it point spectrum} $\sigma_p(\Theta)$, {\it continuous spectrum} $\sigma_c(\Theta)$, and 
{\it residual spectrum}
$\sigma_r(\Theta)$ defined by  
\begin{equation*}
\begin{split}
\sigma_p(\Theta)&=\bigl\{\lambda\in\mathbb C : \ker(\Theta-\lambda)\not=\{0\}\bigr\},\\
\sigma_c(\Theta)&=\bigl\{\lambda\in\mathbb C:  \ker(\Theta-\lambda)=\{0\},\,\,
\ran(\Theta-\lambda) \,\,\text{dense in}\,\,
\mathcal G,\,\lambda\not\in\rho(\Theta)\bigr\},\\
\sigma_r(\Theta)&=\bigl\{\lambda\in\mathbb C: \ker(\Theta-\lambda)=\{0\},\,\,
\ran(H-\lambda)  \,\,\text{not dense in}\,\,
\mathcal G \bigr\}.
\end{split}
\end{equation*}

For a linear relation $\Theta$ in $\mathcal G$ the {\it adjoint} relation $\Theta^*$ is defined by
\[
\Theta^*:=\left\{\begin{pmatrix} y \\ y'\end{pmatrix} : (x',y)=(x,y')\,\,\text{for all}\,\, 
\begin{pmatrix} x \\ x'\end{pmatrix}\in\Theta  \right\}.
\]
It follows that the adjoint relation $\Theta^*$ is closed in $\mathcal G$ 
and that $\Theta^{**}=\overline\Theta$. Observe that $\mul\Theta^*=(\dom\Theta)^\bot$ and that,
in particular, $\Theta^*$ is an operator if and only if $\Theta$ is densely defined. This also implies that for a densely defined operator $\Theta$
the above definition of the adjoint coincides with the usual one for (unbounded) operators. A linear relation $\Theta$ in $\mathcal G$ 
is said to be
{\it symmetric} if $\Theta\subset\Theta^*$ and {\it self-adjoint} if $\Theta=\Theta^*$. We say that $\Theta$ is {\it dissipative} ({\it accumulative}) 
if $\im (x',x)\geq 0$ ($\im (x',x)\leq 0$, respectively) holds for all $(x,x')^\top\in\Theta$, and $\Theta$ is said to be
{\it maximal dissipative} ({\it maximal accumulative}) if $\Theta$ is dissipative (accumulative) and does not admit proper 
dissipative (accumulative, respectively) extensions in $\mathcal G$.

Finally we note that a selfadjoint (maximal dissipative, maximal accumulative) relation $\Theta$ in $\mathcal G$ can always be decomposed into the
direct sum of a selfadjoint (maximal dissipative, maximal accumulative, respectively) operator in the Hilbert space $\overline{\dom\Theta}$ and
a purely multivalued relation in the Hilbert space $\mul\Theta$. This also shows that the spectral theory of selfadjoint (maximal dissipative, 
maximal accumulative) operators in Hilbert spaces extends in a natural form to selfadjoint (maximal dissipative, 
maximal accumulative, respectively) relations in Hilbert spaces.

\end{document}